\documentclass[a4paper,12pt]{amsart}
\usepackage[a4paper,hscale=0.7,vscale=0.75,centering]{geometry}
\usepackage{fullpage}
\usepackage{url}
\usepackage{xcolor}
\usepackage{bbm}
\usepackage{hyperref}       
\usepackage{mathtools}
\usepackage{amssymb}
\usepackage{amsmath}
\usepackage{algorithm}
\usepackage{algorithmic}

\DeclareMathOperator*{\argmax}{argmax}
\DeclareMathOperator*{\argmin}{argmin}

\newtheorem{theorem}{Theorem}[section]
\newtheorem{lemma}[theorem]{Lemma}
\newtheorem{corollary}[theorem]{Corollary}
\newtheorem{proposition}[theorem]{Proposition}

\theoremstyle{definition}
\newtheorem{definition}[theorem]{Definition}

\newtheorem{assumption}{Assumption}

\theoremstyle{remark}
\newtheorem{remark}[theorem]{Remark}

\begin{document}
\title{Entropic mean-field min-max problems via Best Response flow}

\author{Razvan-Andrei Lascu}
\address{School of Mathematical and Computer Sciences, Heriot-Watt University, Edinburgh, UK, and Maxwell Institute for Mathematical Sciences, Edinburgh, UK}
\email{rl2029@hw.ac.uk}

\author{Mateusz B. Majka}
\address{School of Mathematical and Computer Sciences, Heriot-Watt University, Edinburgh, UK, and Maxwell Institute for Mathematical Sciences, Edinburgh, UK}
\email{m.majka@hw.ac.uk}

\author{\L ukasz Szpruch}
\address{School of Mathematics, University of Edinburgh, UK, and The Alan Turing Institute, UK and Simtopia, UK}
\email{l.szpruch@ed.ac.uk}


\begin{abstract}
    We investigate the convergence properties of a continuous-time optimization method, the \textit{Mean-Field Best Response} flow, for solving convex-concave min-max games with entropy regularization. We introduce suitable Lyapunov functions to establish exponential convergence to the unique mixed Nash equilibrium. Additionally, we demonstrate the convergence of the fictitious play flow as a by-product of our analysis.
\end{abstract}

\maketitle
\section{Introduction}
Learning equilibria in min-max games has gained tremendous popularity motivated by the latest advances in machine learning (ML) such as Generative Adversarial Networks (GANs) \cite{NIPS2014_5ca3e9b1}, adversarial learning \cite{madry2018towards}, multi-agent reinforcement learning \cite{zhang2021multiagent} and fairness in machine learning \cite{DBLP:journals/corr/EdwardsS15}. More recently, several works, e.g., \cite{pmlr-v97-hsieh19b, domingo-enrich_mean-field_2021, wang2023exponentially, yulong, trillos2023adversarial, kim2024symmetric}, have successfully demonstrated that the problems of training GANs and adversarial robustness can be viewed through the lens of min-max games over the space of probability measures.

In this work, we are concerned with the continuous-time convergence analysis of the \textit{Mean-Field Best Response} (MF-BR) flow to the unique mixed Nash equilibrium of an entropy-regularized min-max game. As illustrated in \cite{Wibisono2016AVP}, studying the convergence of optimization methods from a continuous time perspective is very fruitful in ML as a guide for better understanding numerical approximations of implementable algorithms.
\subsection{Notation and setup}
Let $\mathcal{X}$, $\mathcal{Y}$ be any subsets of $\mathbb R^d$ (in particular, we allow $\mathcal{X} = \mathcal{Y} = \mathbb{R}^d$), and let $U^{\pi}: \mathcal{X} \to \mathbb{R}$, $U^{\rho}: \mathcal{Y} \to \mathbb{R}$ be two measurable functions such that $\int_{\mathcal{X}} e^{-U^{\pi}(x)} \mathrm{d}x = \int_{\mathcal{Y}} e^{-U^{\rho}(y)} \mathrm{d}y = 1$.\footnote[1]{We omit the normalizing constants $Z_{\pi}$ and $Z_{\rho}$ since we adopt the convention that the potential functions $U^{\pi}$ and $U^{\rho}$ are shifted by $\log Z_{\pi}$ and $\log Z_{\rho}$, respectively.} For any $\mathcal{Z} \subseteq \mathbb{R}^d$, by $\mathcal{P}_{\text{ac}}(\mathcal{Z})$ we denote the space of probability measures on $\mathcal{Z}$ which are absolutely continuous with respect to the Lebesgue measure. Following a standard convention, we use the same symbol to denote a probability measure in $\mathcal{P}_{\text{ac}}(\mathcal{Z})$ as well as its density. If $\pi(x) := e^{-U^{\pi}(x)}$ and $\rho(y) := e^{-U^{\rho}(y)}$, 
then the relative entropy $\operatorname{D_{KL}}(\cdot | \pi):\mathcal{P}_{\text{ac}}(\mathcal{X}) \to [0, \infty)$ with respect to $\pi$ is given for any $\nu \in \mathcal{P}_{\text{ac}}(\mathcal{X})$ by
\begin{equation*}
\operatorname{D_{KL}}(\nu|\pi) = \int_{\mathcal{X}} \log\left(\frac{\nu(x)}{\pi(x)}\right) \nu(x)\mathrm{d}x,
\end{equation*}
and we define $\operatorname{D_{KL}}(\mu|\rho)$ analogously for any $\mu \in \mathcal{P}_{\text{ac}}(\mathcal{Y})$. Let $F:\mathcal{P}\left(\mathcal{X}\right) \times \mathcal{P}\left(\mathcal{Y}\right) \to \mathbb{R}$ be a convex-concave (possibly non-linear) function and $\sigma > 0$ be a regularization parameter. The min-max problem we study is given by
    \begin{equation}
    \label{eq:F-nonlinear}
        \min_{\nu \in \mathcal{P}_{\text{ac}}\left(\mathcal{X}\right)}\max_{\mu \in \mathcal{P}_{\text{ac}}\left(\mathcal{Y}\right)} V^{\sigma}(\nu, \mu), \text{ with } V^{\sigma}(\nu, \mu) \coloneqq F(\nu, \mu) + \frac{\sigma^2}{2}\left(\operatorname{D_{KL}}(\nu|\pi)-\operatorname{D_{KL}}(\mu|\rho)\right).
    \end{equation}
In this setting, one is typically interested in searching for \textit{mixed Nash equilibria} (MNEs) \cite{neumann_morgenstern, nash}, which, given $\sigma > 0,$ are defined as pairs of measures $(\nu_{\sigma}^*, \mu_{\sigma}^*) \in \mathcal{P}_{\text{ac}}(\mathcal{X}) \times \mathcal{P}_{\text{ac}}(\mathcal{Y})$ that satisfy
\begin{equation}
\label{eq: nashdefmeasures}
    V^{\sigma}(\nu_{\sigma}^*, \mu) \leq V^{\sigma}(\nu_{\sigma}^*, \mu_{\sigma}^*) \leq V^{\sigma}(\nu, \mu_{\sigma}^*), \quad \text{for all } (\nu, \mu) \in \mathcal{P}(\mathcal{X}) \times \mathcal{P}(\mathcal{Y}).
\end{equation}
Note that when $F$ is bilinear and $\sigma = 0,$ i.e., when $V^0(\nu,\mu) = \int_{\mathcal{Y}} \int_{\mathcal{X}} f(x,y) \nu(\mathrm{d}x)\mu(\mathrm{d}y),$ for some $f:\mathcal{X} \times \mathcal{Y} \to \mathbb R,$ measures characterized by \eqref{eq: nashdefmeasures} are MNEs in the classical sense of two-player zero-sum games. 

We establish the existence of MNEs for \eqref{eq:F-nonlinear} in Theorem \ref{thm: 2.6} in Section \ref{appendix: AppC}. Since $\nu \mapsto F(\nu, \mu)$ and $\mu \mapsto F(\nu, \mu)$ are convex and concave (see Assumption \ref{assumption: assump-F-conv-conc}), respectively, Lemma \ref{lemma:Uniqueness-saddle-point} in Section \ref{appendix: AppC} guarantees uniqueness of the MNE of game \eqref{eq:F-nonlinear}. In Proposition \ref{prop:gamma-convergence}, we show that $V^{\sigma}(\nu_{\sigma}^*, \mu_{\sigma}^*)$ converges to $F(\nu^*, \mu^*)$ as $\sigma \to 0,$ where $(\nu^*, \mu^*) \in \mathcal{P}(\mathcal{X}) \times \mathcal{P}(\mathcal{Y})$ is an MNE for game \eqref{eq:F-nonlinear} without regularization. We would like to stress that Proposition \ref{prop:gamma-convergence} does not imply the $\Gamma$-convergence of $V^{\sigma}$ to $F$ as $\sigma\to 0,$ and hence the claim about the $\Gamma$-convergence made in \cite{lascu_fisher-rao_2024} on page $2$ in reference to the present paper is not true.

In what follows, we will introduce the MF-BR on the space $\left(\mathcal{P}_{\text{ac}}(\mathcal{X}) \times \mathcal{P}_{\text{ac}}(\mathcal{Y}), \operatorname{TV}\right),$ where $\operatorname{TV}$ denotes the Total Variation (see Definition \ref{def:KRwasserstein} in Section \ref{app: AppB}). 

\subsection{Mean-field best response dynamics}
\label{subsec: MF-BR}
Best response (BR) is a learning algorithm initially proposed in \cite{10.2307/2938230, RePEc:eee:jetheo:v:57:y:1992:i:2:p:343-362,hofbauer_stability_1995} for games on $\mathbb R^d$ (i.e., with finite dimensional sets of strategies) with the purpose of evaluating the payoff function of two-player zero-sum games at the Nash equilibrium. In this learning process, at each round of the game, each player plays their best response against the \textit{current} strategies of the other players. The convergence analysis of BR both in the discrete and continuous time setup has been studied in detail for games on $\mathbb R^d$; see e.g. \cite{HARRIS1998238, Hofbauer}. In the present paper, we introduce the Mean-Field Best Response (MF-BR) flow, which is an infinite-dimensional counterpart of the classical BR algorithm.

In order to motivate the introduction of the MF-BR gradient flow, we start by observing that, given $\sigma > 0,$ the MNE $(\nu_{\sigma}^*, \mu_{\sigma}^*)$ of \eqref{eq:F-nonlinear} solves 
\begin{equation}
\label{eq: mne-argmin-argmax}
    \begin{cases}
        \nu_{\sigma}^* = \argmin_{\nu \in \mathcal{P}_{\text{ac}}(\mathcal{X})} \left\{F(\nu, \mu_{\sigma}^*) + \frac{\sigma^2}{2}\operatorname{D_{KL}}(\nu|\pi)\right\},\\
        \mu_{\sigma}^* = \argmax_{\mu \in \mathcal{P}_{\text{ac}}(\mathcal{Y})} \left\{F(\nu_{\sigma}^*, \mu) - \frac{\sigma^2}{2}\operatorname{D_{KL}}(\mu|\rho)\right\}.
    \end{cases}
\end{equation}
According to Proposition \ref{prop: 2.5} in Section \ref{appendix: AppC}, which characterizes $(\nu_{\sigma}^*, \mu_{\sigma}^*)$ via a first-order condition, we have that the MNE $(\nu_{\sigma}^*, \mu_{\sigma}^*)$ satisfying \eqref{eq: mne-argmin-argmax} is given implicitly by the equations
\begin{equation}
\label{eq: nu-implicit}
\nu_{\sigma}^*(x) = \frac{1}{Z(\nu_{\sigma}^*, \mu_{\sigma}^*)} \exp{\left( -\frac{2}{\sigma^2}\frac{\delta F}{\delta \nu} (\nu_{\sigma}^*, \mu_{\sigma}^*, x) - U^{\pi}(x) \right)},
\end{equation}
\begin{equation}
\label{eq: mu-implicit}
    \mu_{\sigma}^*(y) = \frac{1}{Z'(\nu_{\sigma}^*,\mu_{\sigma}^*)} \exp{\left( \frac{2}{\sigma^2}\frac{\delta F}{\delta \mu} (\nu_{\sigma}^*, \mu_{\sigma}^*, y) - U^{\rho}(y) \right)},
\end{equation}
where $Z(\nu_{\sigma}^*, \mu_{\sigma}^*)$ and $Z'(\nu_{\sigma}^*, \mu_{\sigma}^*)$ are normalizing constants, and $\frac{\delta F}{\delta \nu}, \frac{\delta F}{\delta \mu}$ are \emph{flat derivatives} of $F$ (see Definition \ref{def:fderivative} in Section \ref{app: AppB}). The key idea for defining the mean-field BR flow is to show that the MNE, for which the equations \eqref{eq: nu-implicit} and \eqref{eq: mu-implicit} hold, satisfies a fixed-point problem. For $\sigma > 0,$ we define $\Psi_{\sigma}: \mathcal{P}_{\text{ac}}(\mathcal{X}) \times \mathcal{P}_{\text{ac}}(\mathcal{Y}) \to \mathcal{P}_{\text{ac}}(\mathcal{X})$ and $\Phi_{\sigma}: \mathcal{P}_{\text{ac}}(\mathcal{X}) \times \mathcal{P}_{\text{ac}}(\mathcal{Y}) \to \mathcal{P}_{\text{ac}}(\mathcal{Y})$ by
\begin{equation}
\label{eq:argminpsi-density}
    \Psi_{\sigma}(\nu, \mu)(x) = \frac{1}{Z(\nu, \mu)}\exp\left({-\frac{2}{\sigma^2}\frac{\delta F}{\delta \nu}(\nu, \mu, x) - U^{\pi}(x)}\right),
\end{equation}
\begin{equation}
\label{eq:argmaxphi-density}
    \Phi_{\sigma}(\nu, \mu)(y) = \frac{1}{Z'(\nu, \mu)}\exp\left({\frac{2}{\sigma^2}\frac{\delta F}{\delta \mu}(\nu, \mu, y) - U^{\rho}(y)}\right), 
\end{equation}
for all $(x,y) \in \mathcal{X} \times \mathcal{Y}$ Lebesgue almost surely, and where $Z(\nu,\mu)$ and $Z'(\nu,\mu)$ are normalizing constants depending on $\nu$ and $\mu$. Observe that the MNE $(\nu_{\sigma}^*, \mu_{\sigma}^*)$ satisfies \eqref{eq: nu-implicit} and \eqref{eq: mu-implicit} if and only if it is also a solution to the fixed point problem
\begin{equation}
\label{eq:fixed-point-problem2}
\begin{cases}
    \nu(x) = \Psi_{\sigma}(\nu, \mu)(x)\\
    \mu(y) = \Phi_{\sigma}(\nu, \mu)(y).
\end{cases}
\end{equation}
\begin{remark}
    We note that, for given $(\nu, \mu) \in \mathcal{P}(\mathcal{X}) \times \mathcal{P}(\mathcal{Y})$ and $\sigma > 0,$ the maps $\Psi_{\sigma}$ and $\Phi_{\sigma}$ satisfy the following variational representation:
\begin{equation}
\label{eq:psi-linear-F}
    \Psi_{\sigma}(\nu, \mu) = \argmin_{\nu' \in \mathcal{P}_{\text{ac}}(\mathcal{X})} \left\{\int_{\mathcal{X}} \frac{\delta F}{\delta \nu}(\nu, \mu, x)(\nu'-\nu)(\mathrm{d}x) + \frac{\sigma^2}{2}\operatorname{D_{KL}}(\nu'|\pi)\right\},
\end{equation}
\begin{equation}
\label{eq:phi-linear-F}
    \Phi_{\sigma}(\nu, \mu) = \argmax_{\mu' \in \mathcal{P}_{\text{ac}}(\mathcal{Y})} \left\{\int_{\mathcal{Y}} \frac{\delta F}{\delta \mu}(\nu, \mu, y)(\mu'-\mu)(\mathrm{d}y) - \frac{\sigma^2}{2}\operatorname{D_{KL}}(\mu'|\rho)\right\}.
\end{equation}
It is essential to stress that since $\Psi_{\sigma}$ and $\Phi_{\sigma}$ are the minimizer and maximizer (i.e., \textit{best responses}) of entropy-regularized linearizations of $F$ we can define $\Psi_{\sigma}$ and $\Phi_{\sigma}$ explicitly as in \eqref{eq:argminpsi-density} and \eqref{eq:argmaxphi-density}. Otherwise, if we would consider $\Psi_{\sigma}$ and $\Phi_{\sigma}$ to be the minimizer and maximizer of $V^{\sigma},$ then due to the non-linearity of $F,$ expressions \eqref{eq:argminpsi-density} and \eqref{eq:argmaxphi-density} would become implicit. 
\end{remark}

Let $(\nu_t)_{t \in [0, \infty)} \subset \mathcal{P}_{\text{ac}}(\mathcal{X})$ and $(\mu_t)_{t \in [0, \infty)} \subset \mathcal{P}_{\text{ac}}(\mathcal{Y})$ denote the strategies of each player. Then, 
since finding the unique MNE of \eqref{eq:F-nonlinear} is equivalent to finding the unique fixed point which solves \eqref{eq:fixed-point-problem2}, it is natural that the pair of strategies $(\nu_t, \mu_t)_{t \geq 0}$ evolves on $\left(\mathcal{P}_{\text{ac}}(\mathcal{X}) \times \mathcal{P}_{\text{ac}}(\mathcal{Y}), \operatorname{TV}\right)$ along the flow given by 
\begin{equation}
\label{eq: mean-field-br}
\begin{cases}
    \mathrm{d}\nu_t(x) = \alpha\left(\Psi_{\sigma}(\nu_t, \mu_t)(x) - \nu_t(x) \right)\mathrm{d}t,\\
    \mathrm{d}\mu_t(y) = \alpha\left(\Phi_{\sigma}(\nu_t, \mu_t)(y) - \mu_t(y) \right)\mathrm{d}t, \quad t \geq 0,
\end{cases}
\end{equation}
for some initial condition $(\nu_0, \mu_0) \in \mathcal{P}_{\text{ac}}(\mathcal{X}) \times \mathcal{P}_{\text{ac}}(\mathcal{Y})$, and a parameter (learning rate) $\alpha > 0$. Note that a similar algorithm has been studied by \cite{https://doi.org/10.48550/arxiv.2202.05841, nitanda_fictitious} in the context of a different class of optimization problems on $\left(\mathcal{P}_{p}(\mathbb R^d), \mathcal{W}_p\right)$.

Since the only solution to the fixed point problem \eqref{eq:fixed-point-problem2} is automatically the unique MNE of the game, it follows that \eqref{eq:fixed-point-problem2} gives a strong indication for considering the map
\begin{equation*}
    t \mapsto \operatorname{D_{KL}}(\nu_t|\Psi_{\sigma}(\nu_t,\mu_t)) + \operatorname{D_{KL}}(\mu_t|\Phi_{\sigma}(\nu_t,\mu_t))
\end{equation*} 
as a suitable Lyapunov function in the subsequent convergence analysis of the MF-BR dynamics. Indeed, it holds that $\operatorname{D_{KL}}(\nu|\Psi_{\sigma}(\nu,\mu)) + \operatorname{D_{KL}}(\mu|\Phi_{\sigma}(\nu,\mu)) \geq 0,$ for all $(\nu, \mu) \in \mathcal{P}_{\text{ac}}(\mathcal{X}) \times \mathcal{P}_{\text{ac}}(\mathcal{Y})$, with equality if and only if the fixed point problem \eqref{eq:fixed-point-problem2} is satisfied. Hence, if we can show that $t \mapsto \operatorname{D_{KL}}(\nu_t|\Psi_{\sigma}(\nu_t,\mu_t)) + \operatorname{D_{KL}}(\mu_t|\Phi_{\sigma}(\nu_t,\mu_t))$ converges to zero as $t \to \infty$, then we know that the unique MNE of \eqref{eq:F-nonlinear} has been attained.

Another appropriate Lyapunov function, especially in the case of the best response dynamics in games (as demonstrated in the context of games on $\mathbb{R}^d$ in e.g. \cite{HARRIS1998238, Hofbauer}), is the so-called Nikaidò-Isoda (NI) error \cite{Nikaid1955NoteON}, which, for all $(\nu, \mu) \in \mathcal{P}\left(\mathcal{X}\right) \times \mathcal{P}\left(\mathcal{Y}\right)$, can be defined as
\begin{equation*}
    \text{NI}(\nu,\mu) \coloneqq \max_{\mu' \in \mathcal{P}\left(\mathcal{Y}\right)} V^{\sigma}(\nu, \mu') - \min_{\nu' \in \mathcal{P}\left(\mathcal{X}\right)} V^{\sigma}(\nu', \mu).
\end{equation*}
From the saddle point condition \eqref{eq: nashdefmeasures}, it follows that $\text{NI}(\nu,\mu) \geq 0$ and $\text{NI}(\nu,\mu) = 0$ if and only if $(\nu, \mu)$ is a MNE. Therefore, if we prove that $t \mapsto \operatorname{NI}(\nu_t, \mu_t)$ converges to zero as $t \to \infty$, then we have precisely shown convergence of the flow to the unique MNE of \eqref{eq:F-nonlinear}. Consequently, given that $\operatorname{NI}(\nu_t, \mu_t) \to 0$ as $t \to \infty,$ we immediately obtain $\operatorname{D_{KL}}({\nu_t}|\nu_{\sigma}^*) + \operatorname{D_{KL}}({\mu_t}|\mu_{\sigma}^*) \to 0$ as $t \to \infty$ due to Lemma \ref{lemma: NI-equals-KL}, and hence $\operatorname{TV}^2({\nu_t},\nu_{\sigma}^*) + \operatorname{TV}^2({\mu_t},\mu_{\sigma}^*) \to 0$ as $t \to \infty$ due to Pinsker's inequality.

\subsubsection{Sketch of convergence proof for the MF-BR flow}
Our convergence result for the MF-BR flow extends the work \cite{https://doi.org/10.48550/arxiv.2202.05841} from the case of a single-player optimization problem to the class of games \eqref{eq:F-nonlinear}, which requires novel Lyapunov functions. 
We also work with the Total Variation instead of the Wasserstein distance, which allows for less regularity of $F$. For any $m,m' \in \mathcal{P}(\mathcal{M})$, with $\mathcal{M} \subseteq \mathbb R^d$, let $\operatorname{D_J}(m,m') \coloneqq \operatorname{D_{KL}}(m|m') + \operatorname{D_{KL}}(m'|m)$ denote Jeffreys divergence \cite{jeffreys} between $m$ and $m'$. Assuming the existence of the flow $(\nu_t, \mu_t)_{t \geq 0}$ satisfying \eqref{eq: mean-field-br}, and the differentiability of the map \(t \mapsto \operatorname{D_{KL}}({\nu_t}|{\Psi_{\sigma}(\nu_t,\mu_t)}) + \operatorname{D_{KL}}({\mu_t}|{\Phi_{\sigma}(\nu_t,\mu_t)})\) for all \(t > 0\), which will be established in Proposition \ref{prop:Existence-flows-FP} and Theorem \ref{thm:exponential-kl}, we can show that
\begin{equation*}
    \frac{\mathrm{d}}{\mathrm{d}t} \left(\operatorname{D_{KL}}({\nu_t}|{\Psi_{\sigma}(\nu_t,\mu_t)}) + \operatorname{D_{KL}}({\mu_t}|{\Phi_{\sigma}(\nu_t,\mu_t)})\right) \leq -\alpha \big(\operatorname{D_J}(\nu_t|\Psi_{\sigma}(\nu_t, \mu_t)) + \operatorname{D_J}(\mu_t|\Phi_{\sigma}(\nu_t, \mu_t))\big).
\end{equation*}
Applying Gronwall's inequality gives 
\begin{equation*}
        \operatorname{D_{KL}}({\nu_t}|{\Psi_{\sigma}(\nu_t,\mu_t)}) + \operatorname{D_{KL}}({\mu_t}|{\Phi_{\sigma}(\nu_t,\mu_t)}) \leq e^{- \alpha t}\left(\operatorname{D_{KL}}({\nu_0}|{\Psi_{\sigma}(\nu_0,\mu_0)}) + \operatorname{D_{KL}}({\mu_0}|{\Phi_{\sigma}(\nu_0,\mu_0)})\right).
\end{equation*}
Using Lemma \ref{lemma: NI-equals-KL} from Section \ref{appendix: AppC}, that is
\begin{equation}
\label{eq:NI-bounded-above}
    \frac{\sigma^2}{2}\left(\operatorname{D_{KL}}({\nu_t}|\nu_{\sigma}^*) + \operatorname{D_{KL}}({\mu_t}|\mu_{\sigma}^*)\right) \leq \operatorname{NI}(\nu_t, \mu_t) \leq \frac{\sigma^2}{2}\left(\operatorname{D_{KL}}({\nu_t}|{\Psi_{\sigma}(\nu_t,\mu_t)}) + \operatorname{D_{KL}}({\mu_t}|{\Phi_{\sigma}(\nu_t,\mu_t)})\right),
\end{equation} 
we obtain that
\begin{equation*}
         \operatorname{NI}({\nu_t}, {\mu_t}) \leq \frac{\sigma^2}{2}e^{-\alpha t}\left(\operatorname{D_{KL}}({\nu_0}|{\Psi_{\sigma}(\nu_0,\mu_0)}) + \operatorname{D_{KL}}({\mu_0}|{\Phi_{\sigma}(\nu_0,\mu_0)})\right),
\end{equation*}
and consequently
\begin{equation*}
   \operatorname{D_{KL}}({\nu_t}|\nu_{\sigma}^*) + \operatorname{D_{KL}}({\mu_t}|\mu_{\sigma}^*) \leq e^{-\alpha t}\left(\operatorname{D_{KL}}({\nu_0}|{\Psi_{\sigma}(\nu_0,\mu_0)}) + \operatorname{D_{KL}}({\mu_0}|{\Phi_{\sigma}(\nu_0,\mu_0)})\right),
\end{equation*}
\begin{equation*}
    \operatorname{TV}^2({\nu_t},\nu_{\sigma}^*) + \operatorname{TV}^2({\mu_t},\mu_{\sigma}^*) \leq \frac{1}{2}e^{-\alpha t}\left(\operatorname{D_{KL}}({\nu_0}|{\Psi_{\sigma}(\nu_0,\mu_0)}) + \operatorname{D_{KL}}({\mu_0}|{\Phi_{\sigma}(\nu_0,\mu_0)})\right).
\end{equation*}
\subsection{Our contribution}
\label{sec: suitable-lyapunov}
We prove the existence of the MF-BR flow $(\nu_t, \mu_t)_{t \geq 0}$ and, independently of initialization, we prove its convergence with rate $\mathcal{O}\left(e^{-\alpha t}\right)$ to the unique MNE of \eqref{eq:F-nonlinear} via the Lyapunov function $t \mapsto \operatorname{D_{KL}}(\nu_t|\Psi_{\sigma}(\nu_t,\mu_t)) + \operatorname{D_{KL}}(\mu_t|\Phi_{\sigma}(\nu_t,\mu_t)).$ Consequently, using \eqref{eq:NI-bounded-above}, we obtain convergence of the MF-BR flow with rate $\mathcal{O}\left(\frac{\sigma^2}{2}e^{-\alpha t}\right)$ with respect to $t \mapsto \operatorname{NI}(\nu_t, \mu_t),$ and with rate $\mathcal{O}\left(e^{-\alpha t}\right)$ with respect to $t \mapsto \operatorname{D_{KL}}({\nu_t}|\nu_{\sigma}^*) + \operatorname{D_{KL}}({\mu_t}|\mu_{\sigma}^*)$ and $t \mapsto \operatorname{TV}^2({\nu_t},\nu_{\sigma}^*) + \operatorname{TV}^2({\mu_t},\mu_{\sigma}^*).$ We show that for games with $F(\nu, \mu) = \int_{\mathcal{Y}}\int_{\mathcal{X}} f(x,y) \nu(\mathrm{d}x) \mu(\mathrm{d}y)$, where $f:\mathcal{X} \times \mathcal{Y} \to \mathbb{R}$ is bounded, the convergence rate with respect to $t \mapsto \operatorname{NI}(\nu_t, \mu_t)$ becomes $\mathcal{O}\left(e^{-\alpha t}\right)$ (and is hence independent of the regularization parameter $\sigma$).

\subsection{Related works}
\label{subsec:LRintro}
\subsubsection{Best response and fictitious play dynamics}
\label{sec: review-fict-play}
For the class of two-player zero-sum games with payoff function $\mathbb R^{d_1} \times \mathbb R^{d_2} \ni (x,y) \mapsto x^T A y \in \mathbb R,$ where $x,y$ denote the strategies of the players and $A \in \mathbb R^{d_1 \times d_2}$ denotes the payoff matrix, and assuming that the game may have multiple Nash equilibria, \cite[Theorem $9$]{HARRIS1998238} establishes that continuous-time BR converges to the set of Nash equilibria with exponential rate $e^{-t}$ along the Nikaidò-Isoda (NI) error \cite{Nikaid1955NoteON}. Later, assuming that the payoff function of the game is continuous convex-concave, that the strategy spaces are compact and convex and that the game may have multiple Nash equilibria, \cite{Hofbauer} proves that continuous-time BR converges to the set of Nash equilibria with rate $e^{-t}$ along the NI error. 

In contrast to \cite{HARRIS1998238} and \cite{Hofbauer}, we consider an infinite-dimensional two-player zero-sum game on the space of probability measures. In our setting, the strategy spaces can be any subsets of $\mathbb R^d$, not necessarily compact and convex. An argument from \cite{HARRIS1998238}, which was later formalized in \cite{Hofbauer}, showed that continuous-time BR and continuous-time fictitious play (see e.g. \cite{brown:fp1951, HARRIS1998238, Ostrovski} for details on the fictitious play algorithm) are in fact equivalent up to a rescale in time (see Remark \ref{remark: fp-equiv-br}).

More recently, there has been interest in the convergence analysis of continuous-time fictitious play, continuous-time BR and their discrete-time counterparts in the context of zero-sum stochastic games; see e.g. \cite{https://doi.org/10.48550/arxiv.2010.04223, LESLIE2020105095, pmlr-v162-baudin22a}, and mean-field games; see e.g. \cite{CH17, HS19, PPLGEP20, BLP21}. In this context, we would like to point out that, after the first version of the present paper appeared on ArXiv, \cite{kim2024symmetric} extended this work and proposed a particle algorithm, consisting of two nested loops, which implements the MF-BR flow \eqref{eq: mean-field-br} with theoretical convergence guarantees. The inner loop computes the best responses $\Psi_{\sigma}, \Phi_{\sigma}$ via Langevin dynamics, while the outer loop updates \eqref{eq: mean-field-br} via an explicit Euler scheme.

\subsubsection{Wasserstein and Fisher-Rao gradient flows for games}
Recently, there has been intensive research in analyzing the convergence of various types of gradient flows to the set of MNEs in a particular setup of game \eqref{eq:F-nonlinear} in which $F$ is bilinear, that is $F(\nu, \mu) = \int_{\mathcal{Y}}\int_{\mathcal{X}} f(x,y) \nu(\mathrm{d}x) \mu(\mathrm{d}y)$, regularized by the entropy instead of the relative entropy $\operatorname{D_{KL}}$, 
and where $\mathcal{X}$ and $\mathcal{Y}$ are compact smooth manifolds without boundary, 
embedded in the Euclidean space or they are Euclidean tori, and $f$ has sufficient regularity, i.e., it is at least continuously differentiable and $\nabla_x f, \nabla_y f$ satisfy Lipschitz conditions, 
see e.g. \cite{domingo-enrich_mean-field_2021, ma2022provably, yulong, wang2023exponentially}.

In this particular setting, \cite{ma2022provably,yulong} study the convergence of the \textit{Wasserstein gradient flow} and obtain exponential convergence to the MNE in the case where the flows of the players convergence at different speeds. In \cite{ma2022provably} the speeds of convergence of the flows $(\nu_t)_{t \geq 0}$ and $(\mu_t)_{t \geq 0}$ are assumed to be different in the sense that one of the flows has achieved equilibrium while the other one is still governed by the Wasserstein gradient flow equation. \cite[Theorem $5$]{ma2022provably} states that under these separated dynamics, the flow $(\nu_t,\mu_t)_{t \geq 0}$ converges (without explicit rate) to the unique MNE of the game. 

In contrast to \cite{ma2022provably}, \cite{yulong} 
proved exponential convergence of the Wasserstein gradient flow with respect to Lyapunov functions adapted from \cite{doan2021convergence}. The proof of \cite{yulong} relies on defining operators similar to our $\Psi_{\sigma}, \Phi_{\sigma},$ and assuming that these satisfy the log-Sobolev inequality. Furthermore, \cite{yulong} introduced a finite time-scale separation parameter $\eta > 0$ at the level of the Wasserstein gradient flow, so that players' strategies evolve at different speeds along the flow but none of them is at equilibrium. \cite[Theorem $2.1$]{yulong} proved that if $\eta$ depends on the log-Sobolev constant and the regularization parameter, then the timescale separated Wasserstein gradient flow converges exponentially to the MNE of the game. The rate of convergence also depends on the log-Sobolev constant and on the regularization parameter. On the contrary, we do not impose any assumptions on $\Psi_{\sigma}, \Phi_{\sigma}$ and the exponential convergence rate of the MF-BR flow we propose only depends on the learning rate $\alpha> 0,$ which can be chosen arbitrarily.

However, in discrete-time there is a trade-off in the choice of $\alpha$ and the step-size of the discretization scheme. As mentioned in Subsection \ref{sec: review-fict-play}, \cite{kim2024symmetric} considered the time discretization of the MF-BR flow \eqref{eq: mean-field-br} via an explicit Euler scheme with step-size $\tau > 0$ and learning rate $\alpha > 0$ (see Subsection 4.1 in \cite{kim2024symmetric}). In this case, \cite[Theorem $4.1$]{kim2024symmetric} shows that the number of iterations required for the iterates $(\nu_k, \mu_k)_{k \geq 0}$ generated by the explicit Euler scheme to converge to the MNE $(\nu_{\sigma}^*, \mu_{\sigma}^*)$ of \eqref{eq:F-nonlinear} in NI with accuracy $\epsilon > 0$ is $k = \mathcal{O}\left(\frac{1}{\epsilon}\log\frac{1}{\epsilon}\right).$ This can be achieved by taking $\tau = \mathcal{O}\left(\log \frac{1}{\epsilon}\right)$ and $\alpha = \mathcal{O}(\epsilon).$ In other words, convergence with high accuracy (small $\epsilon$) of the iterates $(\nu_k, \mu_k)_{k \geq 0}$ to $(\nu_{\sigma}^*, \mu_{\sigma}^*)$ in NI can be attained with a large $\tau$ at the expense of a small $\alpha.$ 

Under the same setup in which $F$ is bilinear but with $\sigma = 0$, \cite{domingo-enrich_mean-field_2021} studied the convergence of the \textit{Wasserstein-Fisher-Rao gradient flow} without explicit convergence rates. For $t_0 > 0$ (depending on parameters of the individual contributions of the Wasserstein and the Fisher-Rao components in the WFR flow), and when the Fisher-Rao component dominates the Wasserstein component of the WFR flow, \cite[Theorem $2$]{domingo-enrich_mean-field_2021} shows that the pair $\left(\frac{1}{t_0}\int_{0}^{t_0} \nu_s \mathrm{d}s, \frac{1}{t_0}\int_{0}^{t_0} \mu_s \mathrm{d}s\right)$ is an $\epsilon$-approximate MNE of the game, i.e., $\operatorname{NI}\left(\frac{1}{t_0}\int_{0}^{t_0} \nu_s \mathrm{d}s, \frac{1}{t_0}\int_{0}^{t_0} \mu_s \mathrm{d}s\right) \leq \epsilon$ with $\epsilon > 0$ arbitrary. 

In \cite{wang2023exponentially}, the discrete-time convergence of the WFR flow is considered in the case where $F$ is bilinear, $\sigma=0$ and the MNE of the game is unique. Requiring that the flow is initialized sufficiently close to the MNE, \cite[Theorem $2.2$]{wang2023exponentially} shows local exponential convergence with respect to the NI error and the WFR distance to the unique MNE of the game. 

Lately, \cite{lascu_fisher-rao_2024} studied the convergence of a Fisher-Rao (FR) gradient flow to the MNE of \eqref{eq:F-nonlinear}. Both the MF-BR and FR dynamics converge exponentially to the MNE but with rates which differ significantly in terms of $\sigma.$ The rate for MF-BR with respect to the map $t \mapsto \operatorname{D_{KL}}(\nu_t|\Psi(\nu_t,\mu_t)) + \operatorname{D_{KL}}(\mu_t|\Phi(\nu_t,\mu_t))$ is independent of $\sigma$ (and with respect to $t \mapsto \operatorname{NI}(\nu_t, \mu_t)$ the rate degenerates quadratically fast with $\sigma \to 0$), while for FR the rate degenerates exponentially fast with $\sigma \to 0$ (\cite[Theorem $2.3$]{lascu_fisher-rao_2024}). Another important aspect to compare is the assumptions used for both dynamics. While the results for both flows rely on fairly standard assumptions such as convexity-concavity of $F$ (see Assumption \ref{assumption: assump-F-conv-conc}) and boundedness of first and second order flat derivatives of $F$ (see Assumption \ref{assumption: boundedness-first-flat} and \ref{assump:F}), it is worth noting that the FR gradient flow needs an additional assumption (\cite[Assumption $4$]{lascu_fisher-rao_2024}) about the comparability of the initial condition $(\nu_0,\mu_0)$ to the reference measures $\pi$ and $\rho.$ This is a ``warm start'' condition typically needed for birth-death flows (see the discussions in \cite{lu_accelerating_2019, https://doi.org/10.48550/arxiv.2206.02774, lascu_fisher-rao_2024}). On the other hand, the initial condition $(\nu_0,\mu_0)$ in our analysis of the MF-BR flow can be an arbitrary pair of measures in $\mathcal{P}_{\text{ac}}(\mathcal{X}) \times \mathcal{P}_{\text{ac}}(\mathcal{Y}).$

\section{Main results}
\label{sec:Main-Results}

As we explained in the introduction, we study the convergence of the MF-BR and FR dynamics to the unique MNE of the entropy-regularized two-player zero-sum game given by \eqref{eq:F-nonlinear},
where $F:\mathcal{P}\left(\mathcal{X}\right) \times \mathcal{P}\left(\mathcal{Y}\right) \to \mathbb{R}$ is a non-linear function and $\sigma > 0.$ Throughout the paper, we have the following assumptions on $F$.

\begin{assumption}[Convexity-concavity of $F$]
\label{assumption: assump-F-conv-conc}
	Suppose $F$ admits first order flat derivatives with respect to both $\nu$ and $\mu$ as stated in Definition \ref{def:fderivative}.  
    Furthermore, suppose that $F$ is convex in $\nu$ and concave in $\mu$, i.e., for any $\nu, \nu' \in \mathcal{P}\left(\mathcal{X}\right)$ and any $\mu, \mu' \in \mathcal{P}\left(\mathcal{Y}\right)$, we have
		\begin{equation}
        \label{eq:convexF}
		F(\nu', \mu) - F(\nu, \mu) \geq \int_{\mathcal{X}} \frac{\delta F}{\delta \nu}(\nu,\mu,x) (\nu'-\nu)(\mathrm{d}x), 
		\end{equation}
        \begin{equation}
        \label{eq:concaveF}
        F(\nu, \mu') - F(\nu, \mu) \leq \int_{\mathcal{Y}} \frac{\delta F}{\delta \mu}(\nu,\mu,y) (\mu'-\mu)(\mathrm{d}y).
        \end{equation}
\end{assumption}
\begin{assumption}[Boundedness of first order flat derivatives]
\label{assumption: boundedness-first-flat}
    There exist constants $C_{\nu}, C_{\mu} > 0$ such that for all $(\nu, \mu) \in \mathcal{P}(\mathcal{X}) \times \mathcal{P}(\mathcal{Y})$ and for all $(x, y) \in \mathcal{X} \times \mathcal{Y}$, we have
		\begin{equation*}
		\left| \frac{\delta F}{\delta \nu} (\nu, \mu, x) \right| \leq C_{\nu}, \quad  \left| \frac{\delta F}{\delta \mu} (\nu, \mu, y) \right| \leq C_{\mu}.
		\end{equation*}
\end{assumption}
\begin{assumption}[Boundedness of second order flat derivatives]
\label{assump:F}
	Suppose $F$ admits second order flat derivatives and that there exist constants $C_{\nu, \nu}, C_{\mu, \mu}, C_{\nu, \mu}, C_{\mu, \nu} > 0$ such that for all $(\nu, \mu) \in \mathcal{P}(\mathcal{X}) \times \mathcal{P}(\mathcal{Y})$ and for all $(x, y), (x', y') \in \mathcal{X} \times \mathcal{Y}$, we have
		\begin{multline*}
		\left| \frac{\delta^2 F}{\delta \nu^2} (\nu, \mu, x, x') \right| \leq C_{\nu, \nu}, \quad \left| \frac{\delta^2 F}{\delta \mu^2} (\nu, \mu, y, y') \right| \leq C_{\mu, \mu},\\ \left| \frac{\delta^2 F}{\delta \nu \delta \mu} (\nu, \mu, y, x) \right| \leq C_{\nu, \mu}, \quad \left| \frac{\delta^2 F}{\delta \mu \delta \nu} (\nu, \mu, x, y) \right| \leq C_{\mu, \nu}.
        \end{multline*}
\end{assumption}
In Lemma \ref{lemma: symmetry-flat}, we prove that the order of the flat derivatives in $\nu$ and $\mu$ can be interchanged. Using Assumption \ref{assump:F}, it is straightforward to check that there exist constants $C_{\nu}', C_{\mu}' > 0$ such that for all $(\nu,\mu) \in \mathcal{P}_{\text{ac}}(\mathcal{X}) \times \mathcal{P}_{\text{ac}}(\mathcal{Y})$, $(\nu',\mu') \in \mathcal{P}_{\text{ac}}(\mathcal{X}) \times \mathcal{P}_{\text{ac}}(\mathcal{Y})$ and all $(x, y) \in \mathcal{X} \times \mathcal{Y}$, we have that
        \begin{equation}
        \label{eq: 2.5i}
        \left|\frac{\delta F}{\delta \nu}(\nu, \mu, x) - \frac{\delta F}{\delta \nu}(\nu', \mu', x)\right| \leq C_{\nu}'\left(\operatorname{TV}(\nu,\nu') + \operatorname{TV}(\mu,\mu')\right),
        \end{equation}
        \begin{equation}
        \label{eq: 2.5ii}
        \left|\frac{\delta F}{\delta \mu}(\nu, \mu, y) - \frac{\delta F}{\delta \mu}(\nu', \mu', y)\right| \leq C_{\mu}'\left(\operatorname{TV}(\nu,\nu') + \operatorname{TV}(\mu,\mu')\right).
        \end{equation}
\begin{remark}
\label{remark:example-f}
    Observe that an objective function $F$ given by $F(\nu, \mu) = \int_{\mathcal{Y}}\int_{\mathcal{X}} f(x,y) \nu(\mathrm{d}x) \mu(\mathrm{d}y),$ where $f:\mathcal{X} \times \mathcal{Y} \to \mathbb{R}$ is bounded but possibly non-convex-non-concave, satisfies Assumptions \ref{assumption: assump-F-conv-conc}, \ref{assumption: boundedness-first-flat}, \ref{assump:F}. Indeed, Assumption \ref{assumption: assump-F-conv-conc} is trivially satisfied by such $F,$ while Assumptions \ref{assumption: boundedness-first-flat} and \ref{assump:F} hold due to the boundedness of $f.$ Functions $F$ of this type are prototypical in applications such as the training of GANs (see, e.g., \cite{pmlr-v70-arjovsky17a, pmlr-v97-hsieh19b}) and distributionally robust optimization (see, e.g, \cite{madry2018towards, sinha2018certifiable}).
\end{remark}
The following result extends Proposition $2.8$ from \cite{https://doi.org/10.48550/arxiv.2202.05841} by showing the existence and uniqueness of the pair of flows $(\nu_t, \mu_t)_{t \geq 0}$ which solve the MF-BR system \eqref{eq: mean-field-br} on \newline $\left(\mathcal{P}_{\text{ac}}(\mathcal{X}) \times \mathcal{P}_{\text{ac}}(\mathcal{Y}), \operatorname{TV}\right).$

\begin{proposition}[Existence of gradient flow for the MF-BR dynamics]
\label{prop:Existence-flows-FP}
    Let Assumptions \ref{assumption: boundedness-first-flat}, \ref{assump:F} hold and let $(\nu_0, \mu_0) \in \mathcal{P}_{\text{ac}}(\mathcal{X}) \times \mathcal{P}_{\text{ac}}(\mathcal{Y})$. Then there exists a unique pair of flows $(\nu_t,\mu_t)_{t \geq 0}$ in $\left(\mathcal{P}_{\text{ac}}(\mathcal{X}) \times \mathcal{P}_{\text{ac}}(\mathcal{Y}), \operatorname{TV}\right)$ satisfying \eqref{eq: mean-field-br}. Moreover, the solutions depend continuously on the initial conditions and $t \mapsto \nu_t \in C^1\left([0, \infty), \mathcal{P}_{\text{ac}}(\mathcal{X})\right),$ $t \mapsto \mu_t \in C^1\left([0, \infty), \mathcal{P}_{\text{ac}}(\mathcal{Y})\right).$
\end{proposition}

We are ready to state one of the main results of the paper. 
\begin{theorem}[Convergence of MF-BR with explicit rates in $\operatorname{D_{KL}}$ and NI]
\label{thm:exponential-kl}
     Let Assumptions \ref{assumption: assump-F-conv-conc}, \ref{assumption: boundedness-first-flat}, \ref{assump:F} hold. Then the map $t \mapsto \operatorname{D_{KL}}({\nu_t}|{\Psi_{\sigma}(\nu_t,\mu_t)}) + \operatorname{D_{KL}}({\mu_t}|{\Phi_{\sigma}(\nu_t,\mu_t)})$ is differentiable for all $t > 0$, and we have that
     \begin{equation*}
    \frac{\mathrm{d}}{\mathrm{d}t} \left(\operatorname{D_{KL}}({\nu_t}|{\Psi_{\sigma}(\nu_t,\mu_t)}) + \operatorname{D_{KL}}({\mu_t}|{\Phi_{\sigma}(\nu_t,\mu_t)})\right) \leq -\alpha \big(\operatorname{D_J}(\nu_t|\Psi_{\sigma}(\nu_t, \mu_t)) + \operatorname{D_J}(\mu_t|\Phi_{\sigma}(\nu_t, \mu_t))\big),
     \end{equation*}
     where, for any $m,m' \in \mathcal{P}(\mathcal{M})$, with $\mathcal{M} \subseteq \mathbb R^d$, $\operatorname{D_J}(m,m') \coloneqq \operatorname{D_{KL}}(m|m') + \operatorname{D_{KL}}(m'|m)$ denotes Jeffreys divergence \cite{jeffreys} between $m$ and $m'$. Furthermore, suppose that $(\nu_0, \mu_0) \in \mathcal{P}_{\text{ac}}(\mathcal{X}) \times \mathcal{P}_{\text{ac}}(\mathcal{Y})$ are chosen such that $\operatorname{D_{KL}}({\nu_0}|{\Psi_{\sigma}(\nu_0,\mu_0)}) + \operatorname{D_{KL}}({\mu_0}|{\Phi_{\sigma}(\nu_0,\mu_0)}) < \infty,$ and let $\left(\nu_{\sigma}^*, \mu_{\sigma}^*\right)$ be the MNE of \eqref{eq:F-nonlinear}. Then, 
    \begin{equation*}
        \operatorname{D_{KL}}({\nu_t}|{\Psi_{\sigma}(\nu_t,\mu_t)}) + \operatorname{D_{KL}}({\mu_t}|{\Phi_{\sigma}(\nu_t,\mu_t)}) \leq e^{-\alpha t}\left(\operatorname{D_{KL}}({\nu_0}|{\Psi_{\sigma}(\nu_0,\mu_0)}) + \operatorname{D_{KL}}({\mu_0}|{\Phi_{\sigma}(\nu_0,\mu_0)})\right),
    \end{equation*}
    \begin{equation*}
         \operatorname{NI}({\nu_t}, {\mu_t}) \leq \frac{\sigma^2}{2}e^{-\alpha t}\left(\operatorname{D_{KL}}({\nu_0}|{\Psi_{\sigma}(\nu_0,\mu_0)}) + \operatorname{D_{KL}}({\mu_0}|{\Phi_{\sigma}(\nu_0,\mu_0)})\right),
    \end{equation*}
\end{theorem}
\begin{corollary}[Convergence of MF-BR with explicit rates in $\operatorname{D_{KL}},$ TV]
\label{corollary: KL-TV}
Let Assumptions \ref{assumption: assump-F-conv-conc}, \ref{assumption: boundedness-first-flat}, \ref{assump:F} hold. Suppose that $(\nu_0, \mu_0) \in \mathcal{P}_{\text{ac}}(\mathcal{X}) \times \mathcal{P}_{\text{ac}}(\mathcal{Y})$ are chosen such that $\operatorname{D_{KL}}({\nu_0}|{\Psi_{\sigma}(\nu_0,\mu_0)}) + \operatorname{D_{KL}}({\mu_0}|{\Phi_{\sigma}(\nu_0,\mu_0)}) < \infty,$ and let $\left(\nu_{\sigma}^*, \mu_{\sigma}^*\right)$ be the MNE of \eqref{eq:F-nonlinear}. Then,
 \begin{equation*}
   \operatorname{D_{KL}}({\nu_t}|\nu_{\sigma}^*) + \operatorname{D_{KL}}({\mu_t}|\mu_{\sigma}^*) \leq e^{-\alpha t}\left(\operatorname{D_{KL}}({\nu_0}|{\Psi_{\sigma}(\nu_0,\mu_0)}) + \operatorname{D_{KL}}({\mu_0}|{\Phi_{\sigma}(\nu_0,\mu_0)})\right),
\end{equation*}
\begin{equation*}
    \operatorname{TV}^2({\nu_t},\nu_{\sigma}^*) + \operatorname{TV}^2({\mu_t},\mu_{\sigma}^*) \leq \frac{1}{2}e^{-\alpha t}\left(\operatorname{D_{KL}}({\nu_0}|{\Psi_{\sigma}(\nu_0,\mu_0)}) + \operatorname{D_{KL}}({\mu_0}|{\Phi_{\sigma}(\nu_0,\mu_0)})\right).
\end{equation*}
\end{corollary}
\begin{corollary}[Convergence of MF-BR with explicit rate in NI error for bilinear $F$]
\label{corollary: convergence-NI-linear-F-FP}
    Let Assumptions \ref{assumption: boundedness-first-flat}, \ref{assump:F} hold. Then, for $F(\nu, \mu) = \int_{\mathcal{Y}}\int_{\mathcal{X}} f(x,y) \nu(\mathrm{d}x) \mu(\mathrm{d}y)$ with $f:\mathcal{X} \times \mathcal{Y} \to \mathbb{R}$ bounded and $\operatorname{NI}({\nu_0}, {\mu_0}) < \infty$, it holds that
    \begin{equation*}
        \operatorname{NI}({\nu_t}, {\mu_t}) \leq e^{-\alpha t}\operatorname{NI}({\nu_0}, {\mu_0}),
    \end{equation*}
\end{corollary}
Lastly, we would like to demonstrate that in the continuous time setting the convergence study for the MF-BR flow \eqref{eq: mean-field-br} consequently leads to the convergence of a related type of flow known as \textit{fictitious play} (FP) in the literature of games on $\mathbb R^d$ (see e.g. \cite{brown:fp1951, 10.2307/1969530, HARRIS1998238}). In the setup of min-max games on $\mathbb R^d$, it is showed in \cite{Hofbauer} that the continuous-time \emph{best response} and \emph{fictitious play} dynamics are equivalent up to a time rescale.
\begin{remark}
\label{remark: fp-equiv-br}
    An intrinsic feature of the \emph{best response} algorithm is that players know the opponent's strategy at the exact same round when they make their move as opposed to \emph{fictitious play}, where players respond best against the historical distribution of the opponent's strategies. The distinction between the FP and BR flows is that for \emph{fictitious play}, the flow equations hold at the level of the averaged-in-time strategies $(\hat{\nu}_t, \hat{\mu}_t) \coloneqq \left(\frac{1}{t}\int_0^t \nu_s \mathrm{d}s, \frac{1}{t}\int_0^t \mu_s \mathrm{d}s\right)$. We show how to recover the (mean-field) \emph{fictitious play} flow from the MF-BR flow \eqref{eq: mean-field-br}. From Proposition \ref{prop:Existence-flows-FP}, we have that, for all $(x,y) \in \mathcal{X} \times \mathcal{Y}$, the maps $[0, \infty) \ni t \mapsto \nu_t(x) \in \mathcal{X}$ and $[0, \infty) \ni t \mapsto \mu_t(y) \in \mathcal{Y}$ are in $C^1([0, \infty)),$ and solve \eqref{eq: mean-field-br}. 

    Therefore, by setting $\hat{\nu}_t = \nu_{\log t}$ and $\hat{\mu}_t = \mu_{\log t},$ for all $t \geq t_0 > 0,$ with initial condition $(\hat{\nu}_{t_0}, \hat{\mu}_{t_0}) \in \mathcal{P}_{\text{ac}}(\mathcal{X}) \times \mathcal{P}_{\text{ac}}(\mathcal{Y})$, and applying the chain rule, we obtain that
    \begin{equation}
    \label{eq:mean-field-BR}
        \begin{cases}
            \mathrm{d}\hat{\nu}_t(x) = \frac{1}{t}\mathrm{d}\nu_{\log t}(x) = \frac{\alpha}{t}\left(\Psi_{\sigma}(\nu_{\log t}, \mu_{\log t})(x) - \nu_{\log t}(x)\right)\mathrm{d}t = \frac{\alpha}{t}\left(\Psi_{\sigma}(\hat{\nu}_t, \hat{\mu}_t)(x) - \hat{\nu}_t(x)\right)\mathrm{d}t,\\
            \mathrm{d}\hat{\mu}_t(y) = \frac{1}{t}\mathrm{d}\mu_{\log t}(y) = \frac{\alpha}{t}\left(\Phi_{\sigma}(\nu_{\log t}, \mu_{\log t})(y) - \mu_{\log t}\right)(y)\mathrm{d}t = \frac{\alpha}{t}\left(\Phi_{\sigma}(\hat{\nu}_t, \hat{\mu}_t)(y) - \hat{\mu}_t(y)\right)\mathrm{d}t,
        \end{cases}
    \end{equation}
    for all $(x,y) \in \mathcal{X} \times \mathcal{Y}$, which is precisely the mean-field version of the classical fictitious play flow studied for instance in \cite{HARRIS1998238, Hofbauer, Ostrovski}.
    
This fact suggests that in continuous time one could arbitrarily choose to work with either the MF-BR flow \eqref{eq: mean-field-br} or the fictitious play flow \eqref{eq:mean-field-BR} since the convergence rates for the flow \eqref{eq:mean-field-BR} can be obtained from Theorem \ref{thm:exponential-kl} via a change in timescale. Specifically, we can show that the maps $t \mapsto \operatorname{D_{KL}}(\hat{\nu}_t|\Psi_{\sigma}(\hat{\nu}_t,\hat{\mu}_t)) + \operatorname{D_{KL}}(\hat{\mu}_t|\Phi_{\sigma}(\hat{\nu}_t,\hat{\mu}_t))$ and $t \mapsto \operatorname{NI}(\hat{\nu}_t, \hat{\mu}_t)$ decrease along the flow \eqref{eq:mean-field-BR} with rates $\mathcal{O}\left(\frac{\alpha}{t}\right)$ and $\mathcal{O}\left(\frac{\alpha\sigma^2}{2t}\right),$ respectively. Additionally, via Corollary \ref{corollary: KL-TV}, the maps $t \mapsto \operatorname{D_{KL}}(\hat{\nu}_t|\nu_{\sigma}^*) + \operatorname{D_{KL}}(\hat{\mu}_t|\mu_{\sigma}^*)$ and $t \mapsto \operatorname{TV}^2(\hat{\nu}_t,\nu_{\sigma}^*) + \operatorname{TV}^2(\hat{\mu_t},\mu_{\sigma}^*)$ decrease along the flow \eqref{eq:mean-field-BR} with rate $\mathcal{O}\left(\frac{\alpha}{t}\right).$
\end{remark}
\begin{remark}
It is worth mentioning that our setup can be generalized by restricting the reference measures $\pi, \rho$ to be probability measures but not necessarily absolutely continuous with respect to the Lebesgue measure on $\mathbb R^d.$ Indeed, on a technical level, $\mathcal{P}_{\text{ac}}(\mathcal{X})$ and $\mathcal{P}_{\text{ac}}(\mathcal{Y})$ would be changed to $\mathcal{P}_{\pi}(\mathcal{X})$ and $\mathcal{P}_{\rho}(\mathcal{Y}),$ the set of probability measures absolutely continuous with respect to $\pi$ and $\rho,$ respectively. However, for practical purposes, we follow the setup adopted in previous works such as \cite{10.1214/20-AIHP1140, https://doi.org/10.48550/arxiv.2202.05841}.
\end{remark}

\section{Proof of Theorem \ref{thm:exponential-kl}, Corollary \ref{corollary: KL-TV} and Corollary \ref{corollary: convergence-NI-linear-F-FP}} 
Before we present the proof of Theorem \ref{thm:exponential-kl}, we state some useful auxiliary results which are proved in Section \ref{appendix: AppC}. We split the proof of Theorem \ref{thm:exponential-kl} into three steps:
    \begin{itemize}
        \item First, we show that the map \(t \mapsto \operatorname{D_{KL}}({\nu_t}|{\Psi_{\sigma}(\nu_t,\mu_t)}) + \operatorname{D_{KL}}({\mu_t}|{\Phi_{\sigma}(\nu_t,\mu_t)})\) is differentiable when $(\nu_t, \mu_t)_{t \geq 0}$ satisfies the MF-BR flow.
        \item Second, we differentiate \(t \mapsto \operatorname{D_{KL}}({\nu_t}|{\Psi_{\sigma}(\nu_t,\mu_t)}) + \operatorname{D_{KL}}({\mu_t}|{\Phi_{\sigma}(\nu_t,\mu_t)})\) with respect to $t$ and show that $\frac{\mathrm{d}}{\mathrm{d}t}\left(\operatorname{D_{KL}}({\nu_t}|{\Psi_{\sigma}(\nu_t,\mu_t)}) + \operatorname{D_{KL}}({\mu_t}|{\Phi_{\sigma}(\nu_t,\mu_t)})\right)$ is bounded above by $-\alpha\left(\operatorname{D_{KL}}({\nu_t}|{\Psi_{\sigma}(\nu_t,\mu_t)}) + \operatorname{D_{KL}}({\mu_t}|{\Phi_{\sigma}(\nu_t,\mu_t)})\right)$.
        \item Lastly, we finish by applying Gronwall's inequality to obtain exponential convergence. Subsequently, we establish exponential convergence with respect to $t \mapsto \operatorname{NI}(\nu_t, \mu_t).$
    \end{itemize}  

The lemma below is an adaptation of the first part of \cite[Proposition $4.2$]{https://doi.org/10.48550/arxiv.2202.05841} to the min-max setting \eqref{eq:F-nonlinear}. 
\begin{lemma}
\label{proposition:PsiPhiBounds}
Suppose that Assumption \ref{assumption: boundedness-first-flat} holds. Then there exist constants $k_{\Psi_{\sigma}},K_{\Psi_{\sigma}}, k_{\Phi_{\sigma}},K_{\Phi_{\sigma}}$ with $0<k_{\Psi_{\sigma}}<1<K_{\Psi_{\sigma}}<\infty$ and $0<k_{\Phi_{\sigma}}<1<K_{\Phi_{\sigma}}<\infty$ such that for all $(\nu,\mu) \in \mathcal{P}(\mathcal{X}) \times \mathcal{P}(\mathcal{Y})$ and all $(x,y) \in \mathcal{X} \times \mathcal{Y}$, we have that
    \begin{equation}
    \label{eq: estimate-psi}
        k_{\Psi_{\sigma}}e^{-U^{\pi}(x)} \leq \Psi_{\sigma}(\nu, \mu)(x) \leq K_{\Psi_{\sigma}}e^{-U^{\pi}(x)},
    \end{equation}
    \begin{equation}
    \label{eq: estimate-phi}
        k_{\Phi_{\sigma}}e^{-U^{\rho}(y)} \leq \Phi_{\sigma}(\nu, \mu)(y) \leq K_{\Phi_{\sigma}}e^{-U^{\rho}(y)},
    \end{equation}
where, by an abuse of notation, $\Psi_{\sigma}(\nu, \mu)(x)$ and $\Phi_{\sigma}(\nu, \mu)(y)$ denote the densities of $\Psi_{\sigma}(\nu, \mu)$ and $\Phi_{\sigma}(\nu, \mu)$, respectively, with respect to the Lebesgue measure on $\mathcal{X}$ and $\mathcal{Y}$. Moreover, $\Psi_{\sigma}(\nu, \mu)$ and $\Phi_{\sigma}(\nu, \mu)$ belong to $\mathcal{P}(\mathcal{X})$ and $\mathcal{P}(\mathcal{Y})$.
\end{lemma}
The corollary and lemma below are extensions of \cite[Corollary $4.5$]{https://doi.org/10.48550/arxiv.2202.05841} and \cite[Lemma $4.7$]{https://doi.org/10.48550/arxiv.2202.05841}, respectively, to the min-max setting \eqref{eq:F-nonlinear}. 
\begin{corollary}
\label{corollary: A.6}
Let Assumption \ref{assumption: boundedness-first-flat}, \ref{assump:F} hold. We have the following bounds on $\nu_t(x)$ and $\mu_t(y):$
\begin{equation}
\label{equation:nu-density-bound}
\left( 1 - e^{-\alpha t} \right) k_{\Psi_{\sigma}} e^{-U^{\pi}\left(x\right)} \leq \nu_t \left(x\right) \leq \left( 1 - e^{-\alpha t} \right) K_{\Psi_{\sigma}} e^{-U^{\pi}\left(x\right)} + e^{-\alpha t}\nu_0\left(x\right),
\end{equation}
\begin{equation*}
\left( 1 - e^{-\alpha t} \right) k_{\Phi_{\sigma}} e^{-U^{\rho}\left(y\right)} \leq \mu_t \left(y\right) \leq \left( 1 - e^{-\alpha t} \right) K_{\Phi_{\sigma}} e^{-U^{\pi}\left(y\right)} + e^{-\alpha t} \mu_0\left(y\right).
\end{equation*}
hold for all \(x,y \in \mathcal{X}\). 
\end{corollary}
\begin{lemma}[Differentiability of $\operatorname{D_{KL}}$ divergence with respect to Gibbs measure]
\label{lemma:integrability-derivative-of-entropy}
Let Assumption \ref{assumption: boundedness-first-flat}, \ref{assump:F} hold and let \(s > 0\). There exist integrable functions \(f, g, \hat{f}, \hat{g}\) such that the following holds for all \((x,y) \in \mathcal{X} \times \mathcal{Y}\) 
and all \(s \leq t < +\infty\)
\[
g\left(x\right) \leq \log \frac{\nu_t\left(x\right)}{e^{-U^{\pi}\left(x\right)}} \left(\Psi_{\sigma}(\nu_t, \mu_t)(x) - \nu_t\left(x\right)\right) \leq f\left(x\right),
\]
\[
\hat{g}\left(y\right) \leq \log \frac{\mu_t\left(y\right)}{e^{-U^{\rho}\left(y\right)}} \left(\Phi_{\sigma}(\nu_t, \mu_t)(y) - \mu_t\left(y\right)\right) \leq \hat{f}\left(y\right).
\]
\end{lemma}
Finally, we state an auxiliary Lemma \ref{lemma: NI-equals-KL}, linking the NI error to the Lyapunov functions $t \mapsto \operatorname{D_{KL}}({\nu_t}|\nu_{\sigma}^*) + \operatorname{D_{KL}}({\mu_t}|\mu_{\sigma}^*)$ and \(t \mapsto \operatorname{D_{KL}}({\nu_t}|{\Psi_{\sigma}(\nu_t,\mu_t)}) + \operatorname{D_{KL}}({\mu_t}|{\Phi_{\sigma}(\nu_t,\mu_t)})\). 
The proof of this lemma can also be found in Section \ref{appendix: AppC}. A similar type of ``sandwich'' lemma can be found in \cite{yulong}.
\begin{lemma}
\label{lemma: NI-equals-KL}
    Let Assumption \ref{assumption: assump-F-conv-conc} hold and let $(\nu_{\sigma}^*, \mu_{\sigma}^*)$ be the MNE of \eqref{eq:F-nonlinear}. Then, for any $\left(\nu, \mu\right) \in \mathcal{P}(\mathcal{X}) \times \mathcal{P}(\mathcal{Y})$ and any $\sigma > 0,$ we have 
    \begin{equation*}
        \frac{\sigma^2}{2}\left(\operatorname{D_{KL}}({\nu}|\nu_{\sigma}^*) + \operatorname{D_{KL}}({\mu}|\mu_{\sigma}^*)\right) \leq \operatorname{NI}(\nu,\mu) \leq \frac{\sigma^2}{2}\left(\operatorname{D_{KL}}({\nu}|{\Psi_{\sigma}(\nu,\mu)}) + \operatorname{D_{KL}}({\mu}|{\Phi_{\sigma}(\nu,\mu)})\right).
    \end{equation*}
The right-hand side inequality becomes equality when $F$ is bilinear, i.e., $$F(\nu, \mu) = \int_{\mathcal{Y}}\int_{\mathcal{X}} f(x,y) \nu(\mathrm{d}x) \mu(\mathrm{d}y),$$ for some function $f:\mathcal{X} \times \mathcal{Y} \to \mathbb{R}.$
\end{lemma} 
\begin{proof}[Proof of Theorem \ref{thm:exponential-kl}]
\emph{Step 1: Differentiability of $\operatorname{D_{KL}}$ with respect to the MF-BR dynamics \eqref{eq: mean-field-br}:}
The differentiability of the map \(t \mapsto \operatorname{D_{KL}}({\nu_t}|{\Psi_{\sigma}(\nu_t,\mu_t)}) + \operatorname{D_{KL}}({\mu_t}|{\Phi_{\sigma}(\nu_t,\mu_t)})\) follows from a standard argument utilizing the dominated convergence theorem (see \cite[Theorem $11.5$]{schilling2017measures} for more details). Let $t > 0$ and a sequence $(t_n)_{n \in \mathbb{N}} \subset (0, \infty)$ such that $t_n \neq t$, for all $n \in \mathbb{N}$, and $\lim_{n \to \infty} t_n = t$. From the definition of $\operatorname{D_{KL}}$, we have that
    \begin{multline}
    \label{eq:integral-dct}
        \partial_t \operatorname{D_{KL}}({\nu_t}|{\Psi_{\sigma}(\nu_t,\mu_t)}) = \lim_{n \to \infty} \frac{\operatorname{D_{KL}}({\nu_{t_n}}|{\Psi_{\sigma}(\nu_{t_n},\mu_{t_n})}) - \operatorname{D_{KL}}({\nu_t}|{\Psi_{\sigma}(\nu_t,\mu_t)})}{t_n - t}\\ 
        = \lim_{n \to \infty} \int_{\mathcal{X}} \frac{1}{t_n-t}\left(\nu_{t_n}(x) \log \frac{\nu_{t_n}(x)}{\Psi_{\sigma}(\nu_{t_n},\mu_{t_n})(x)} - \nu_t(x) \log \frac{\nu_t(x)}{\Psi_{\sigma}(\nu_t,\mu_t)(x)}\right) \mathrm{d}x.
    \end{multline}  
Observe that
\begin{equation*}
        \lim_{n \to \infty} \frac{1}{t_n-t}\left(\nu_{t_n}(x) \log \frac{\nu_{t_n}(x)}{\Psi_{\sigma}(\nu_{t_n},\mu_{t_n})(x)} - \nu_t(x) \log \frac{\nu_t(x)}{\Psi_{\sigma}(\nu_t,\mu_t)(x)}\right) = \partial_t \left(\nu_t(x) \log \frac{\nu_t(x)}{\Psi_{\sigma}(\nu_t,\mu_t)(x)}\right).
\end{equation*}
Therefore, we can exchange the limit and the integral in \eqref{eq:integral-dct} using the dominated convergence theorem if there exist integrable functions $h_{\Psi_{\sigma}}, \Tilde{h}_{\Psi_{\sigma}}: \mathcal{X} \to \mathbb R$ such that, for all $t > 0$ and all $x \in \mathcal{X}$, it holds that
\begin{equation*}
        h_{\Psi_{\sigma}}(x) \leq \partial_t \left(\nu_t(x) \log \frac{\nu_t(x)}{\Psi_{\sigma}(\nu_t,\mu_t)(x)}\right) \leq \Tilde{h}_{\Psi_{\sigma}}(x).
\end{equation*}
We can argue similarly to show that $t \mapsto \operatorname{D_{KL}}({\mu_t}|{\Phi_{\sigma}(\nu_t,\mu_t)})$ is differentiable by finding integrable functions $h_{\Phi_{\sigma}}, \Tilde{h}_{\Phi_{\sigma}}: \mathcal{Y} \to \mathbb R$ such that, for all $t > 0$ and all $y \in \mathcal{Y}$, we have that
\begin{equation*}
        h_{\Phi_{\sigma}}(y) \leq \partial_t \left(\mu_t(y) \log \frac{\mu_t(y)}{\Phi_{\sigma}(\nu_t,\mu_t)(y)}\right) \leq \Tilde{h}_{\Phi_{\sigma}}(y).
\end{equation*}
    Using Lemma \ref{proposition:PsiPhiBounds}, Corollary \ref{corollary: A.6} and Lemma \ref{lemma:integrability-derivative-of-entropy}, we will show how to obtain $h_{\Psi_{\sigma}}$ and $\Tilde{h}_{\Psi_{\sigma}}$ since obtaining $h_{\Phi_{\sigma}}$ and $\Tilde{h}_{\Phi_{\sigma}}$ is analogous. First, we can rewrite
    \begin{align*}
        \partial_t \left(\nu_t(x) \log \frac{\nu_t(x)}{\Psi_{\sigma}(\nu_t,\mu_t)(x)}\right) &= \partial_t \left(\nu_t(x) \log \frac{\nu_t(x)}{e^{-U^{\pi}(x)}} + \nu_t(x) \log \frac{e^{-U^{\pi}(x)}}{\Psi_{\sigma}(\nu_t,\mu_t)(x)}\right)\\
        &= \left(1+ \log \frac{\nu_t(x)}{e^{-U^{\pi}(x)}}\right)\partial_t \nu_t(x) + \partial_t\left(\nu_t(x) \log \frac{e^{-U^{\pi}(x)}}{\Psi_{\sigma}(\nu_t,\mu_t)(x)}\right)\\ 
        &= \alpha\left(1+ \log \frac{\nu_t(x)}{e^{-U^{\pi}(x)}}\right)\left(\Psi_{\sigma}(\nu_t,\mu_t)(x) - \nu_t(x)\right)\\ &+ \alpha\left(\log \frac{e^{-U^{\pi}(x)}}{\Psi_{\sigma}(\nu_t,\mu_t)(x)}\right)\left(\Psi_{\sigma}(\nu_t,\mu_t)(x) - \nu_t(x)\right)\\ &- \nu_t(x) \partial_t \log \Psi_{\sigma}(\nu_t,\mu_t)(x).
    \end{align*}
    Next, we will bound by integrable functions each of the three terms in the last equality.
    \newline
    
    \emph{Bound for the term $\alpha\left(1+ \log \frac{\nu_t(x)}{e^{-U^{\pi}(x)}}\right)\left(\Psi_{\sigma}(\nu_t,\mu_t)(x) - \nu_t(x)\right)$:}
    By Lemma \ref{lemma:integrability-derivative-of-entropy}, there exist integrable functions $f,g: \mathcal{X} \to \mathbb R$ such that for all $x \in \mathcal{X},$
    \begin{equation*}
        f(x) \leq \log \frac{\nu_t(x)}{e^{-U^{\pi}(x)}}\left(\Psi_{\sigma}(\nu_t,\mu_t)(x) - \nu_t(x)\right) \leq g(x).
    \end{equation*}
    Estimate \eqref{equation:nu-density-bound} in Corollary \ref{corollary: A.6} further reads
    \begin{multline}
    \label{eq: new-bound-nu}
    \begin{aligned}
        0 \leq \left( 1 - e^{-\alpha t} \right) k_{\Psi_{\sigma}} e^{-U^{\pi}\left(x\right)} \leq \nu_t \left(x\right) &\leq \left( 1 - e^{-\alpha t} \right) K_{\Psi_{\sigma}} e^{-U^{\pi}\left(x\right)} + e^{-\alpha t}\nu_0\left(x\right)\\ &\leq K_{\Psi_{\sigma}} e^{-U^{\pi}\left(x\right)} + \nu_0\left(x\right).
    \end{aligned}
    \end{multline}
    Thus, combining estimate \eqref{eq: estimate-psi} from Lemma \ref{proposition:PsiPhiBounds} and the estimate above, we have 
    \begin{align*}
        h_1(x) :=-\alpha \left((K_{\Psi_{\sigma}} - k_{\Psi_{\sigma}})e^{-U^{\pi}\left(x\right)} + \nu_0\left(x\right)\right) &\leq \alpha \left(\Psi_{\sigma}(\nu_t,\mu_t)(x) - \nu_t(x)\right)\\ &\leq \alpha K_{\Psi_{\sigma}}e^{-U^{\pi}(x)} =: \Tilde{h}_1(x),
    \end{align*}
    and hence we obtain that
    \begin{equation*}
        h_1(x) + \alpha f(x) \leq \alpha\left(1+ \log \frac{\nu_t(x)}{e^{-U^{\pi}(x)}}\right)\left(\Psi_{\sigma}(\nu_t,\mu_t)(x) - \nu_t(x)\right) \leq \alpha g(x) + \Tilde{h}_1(x).
    \end{equation*}
    \emph{Bound for the term $\alpha\left(\log \frac{e^{-U^{\pi}(x)}}{\Psi_{\sigma}(\nu_t,\mu_t)(x)}\right)\left(\Psi_{\sigma}(\nu_t,\mu_t)(x) - \nu_t(x)\right)$:}
    We can split the second term as
    \begin{align*}
        \alpha\ \log \frac{e^{-U^{\pi}(x)}}{\Psi_{\sigma}(\nu_t,\mu_t)(x)}\left(\Psi_{\sigma}(\nu_t,\mu_t)(x) - \nu_t(x)\right) &= \alpha \Psi_{\sigma}(\nu_t,\mu_t)(x)\log \frac{e^{-U^{\pi}(x)}}{\Psi_{\sigma}(\nu_t,\mu_t)(x)}\\
        &- \alpha \nu_t(x)\log \frac{e^{-U^{\pi}(x)}}{\Psi_{\sigma}(\nu_t,\mu_t)(x)}.
    \end{align*}
    Using estimate \eqref{eq: estimate-psi} in Lemma \ref{proposition:PsiPhiBounds} and the fact that $\Psi_{\sigma}(\nu_t,\mu_t)(x) \geq 0,$ it follows that
    \begin{equation*}
        \Psi_{\sigma}(\nu_t,\mu_t)(x)\log \frac{1}{K_{\Psi_{\sigma}}} \leq \Psi_{\sigma}(\nu_t,\mu_t)(x)\log \frac{e^{-U^{\pi}(x)}}{\Psi_{\sigma}(\nu_t,\mu_t)(x)} \leq  \Psi_{\sigma}(\nu_t,\mu_t)(x)\log \frac{1}{k_{\Psi_{\sigma}}}
    \end{equation*}
    Note that $K_{\Psi_{\sigma}} > 1 > k_{\Psi_{\sigma}} > 0$ and so $\log \frac{1}{K_{\Psi_{\sigma}}} < 0 < \log \frac{1}{k_{\Psi_{\sigma}}}$. Therefore, using estimate \eqref{eq: estimate-psi} again we get that
    \begin{equation*}
        \alpha K_{\Psi_{\sigma}}e^{-U^{\pi}(x)}\log \frac{1}{K_{\Psi_{\sigma}}} \leq \alpha \Psi_{\sigma}(\nu_t,\mu_t)(x)\log \frac{1}{K_{\Psi_{\sigma}}},
    \end{equation*}
    and
    \begin{equation*}
        \alpha \Psi_{\sigma}(\nu_t,\mu_t)(x)\log \frac{1}{k_{\Psi_{\sigma}}} \leq \alpha K_{\Psi_{\sigma}}e^{-U^{\pi}(x)}\log \frac{1}{k_{\Psi_{\sigma}}}.
    \end{equation*}
    Hence, we obtain
    \begin{align*}
        h_2(x) := \alpha K_{\Psi_{\sigma}}e^{-U^{\pi}(x)}\log \frac{1}{K_{\Psi_{\sigma}}} &\leq \alpha \Psi_{\sigma}(\nu_t,\mu_t)(x)\log \frac{e^{-U^{\pi}(x)}}{\Psi_{\sigma}(\nu_t,\mu_t)(x)}\\
        &\leq \alpha K_{\Psi_{\sigma}}e^{-U^{\pi}(x)}\log \frac{1}{k_{\Psi_{\sigma}}} =: \Tilde{h}_2(x).
    \end{align*}
    Next, since $\nu_t(x) \geq 0$ and $\log \frac{1}{K_{\Psi_{\sigma}}} < 0 < \log \frac{1}{k_{\Psi_{\sigma}}}$, it follows from estimates \eqref{eq: estimate-psi} and \eqref{eq: new-bound-nu} that
    \begin{align*}
         \alpha \nu_t(x)\log \frac{1}{K_{\Psi_{\sigma}}} &\leq \alpha \nu_t(x)\log \frac{e^{-U^{\pi}(x)}}{\Psi_{\sigma}(\nu_t,\mu_t)(x)}\\ &\leq \alpha \nu_t(x) \log \frac{1}{k_{\Psi_{\sigma}}}\\
        &\leq \alpha \left(K_{\Psi_{\sigma}} e^{-U^{\pi}\left(x\right)} + \nu_0\left(x\right)\right)\log \frac{1}{k_{\Psi_{\sigma}}} =: -h_3(x).
    \end{align*}
    Since $\log \frac{1}{K_{\Psi_{\sigma}}} < 0,$ using again estimate \eqref{eq: new-bound-nu} gives that
    \begin{equation*}
        -\Tilde{h}_3(x):=\alpha\left(K_{\Psi_{\sigma}} e^{-U^{\pi}\left(x\right)} + \nu_0\left(x\right)\right)\log \frac{1}{K_{\Psi_{\sigma}}} \leq \alpha \nu_t(x)\log \frac{1}{K_{\Psi_{\sigma}}},
    \end{equation*}
    and hence
    \begin{equation*}
        -\Tilde{h}_3(x) \leq \alpha \nu_t(x)\log \frac{e^{-U^{\pi}(x)}}{\Psi_{\sigma}(\nu_t,\mu_t)(x)} \leq -h_3(x).
    \end{equation*}
    Therefore, we obtain
    \begin{equation*}
        h_2(x) + h_3(x) \leq \alpha \left(\log \frac{e^{-U^{\pi}(x)}}{\Psi_{\sigma}(\nu_t,\mu_t)(x)}\right)\left(\Psi_{\sigma}(\nu_t,\mu_t)(x) - \nu_t(x)\right) \leq \Tilde{h}_2(x) + \Tilde{h}_3(x).
    \end{equation*}
    \emph{Bound for the term $\nu_t(x) \partial_t \log \Psi_{\sigma}(\nu_t,\mu_t)(x)$:}
    First, using the expression of $\Psi_{\sigma}$ from \eqref{eq:argminpsi-density}, we can calculate that 
    \begin{align*}
        \partial_t \log \Psi_{\sigma}(\nu_t,\mu_t)(x) &= -\partial_t \log Z(\nu_t,\mu_t) - \frac{2}{\sigma^2} \partial_t \frac{\delta F}{\delta \nu}(\nu_t, \mu_t, x)\\
        &= -\frac{\partial_t Z(\nu_t,\mu_t)}{Z(\nu_t,\mu_t)} - \frac{2}{\sigma^2} \partial_t \frac{\delta F}{\delta \nu}(\nu_t, \mu_t, x)\\
        &= - \frac{1}{Z(\nu_t,\mu_t)} \left(\int_{\mathcal{X}} \frac{\delta Z}{\delta \nu}(\nu_t,\mu_t,z)\partial_t \nu_t(z)\mathrm{d}z + \int_{\mathcal{Y}} \frac{\delta Z}{\delta \mu}(\nu_t,\mu_t,w)\partial_t \mu_t(w)\mathrm{d}w\right)\\ 
        &- \frac{2}{\sigma^2}\left(\int_{\mathcal{X}} \frac{\delta^2 F}{\delta \nu^2}(\nu_t,\mu_t,x,z)\partial_t \nu_t(z)\mathrm{d}z + \int_{\mathcal{Y}} \frac{\delta^2 F}{\delta \mu \delta \nu}(\nu_t,\mu_t,x,w)\partial_t \mu_t(w)\mathrm{d}w\right).
    \end{align*}
    Then, using \eqref{eq:argminpsi-density}, observe that
    \begin{equation*}
        \frac{1}{Z(\nu_t,\mu_t)}\frac{\delta Z}{\delta \nu}(\nu_t,\mu_t,z) = -\frac{2}{\sigma^2}\int_{\mathcal{X}} \frac{\delta^2 F}{\delta \nu^2}(\nu_t,\mu_t,x,z) \Psi_{\sigma}(\nu_t, \mu_t)(x) \mathrm{d}x,
    \end{equation*}
    \begin{equation*}
        \frac{1}{Z(\nu_t,\mu_t)}\frac{\delta Z}{\delta \mu}(\nu_t,\mu_t,w) = -\frac{2}{\sigma^2}\int_{\mathcal{X}} \frac{\delta^2 F}{\delta \mu \delta \nu}(\nu_t,\mu_t,x,w) \Psi_{\sigma}(\nu_t, \mu_t)(x) \mathrm{d}x,
    \end{equation*}
    and hence, we obtain that
    \begin{multline}
    \label{eq: log-Psi}
    \begin{aligned}
        \partial_t \log \Psi_{\sigma}(\nu_t,\mu_t)(x) &= \frac{2}{\sigma^2}\Bigg(\int_{\mathcal{X}} \int_{\mathcal{X}} \frac{\delta^2 F}{\delta \nu^2}(\nu_t,\mu_t,x',z) \Psi_{\sigma}(\nu_t, \mu_t)(x')\partial_t \nu_t(z)\mathrm{d}x'\mathrm{d}z\\ 
        &+ \int_{\mathcal{Y}} \int_{\mathcal{X}} \frac{\delta^2 F}{\delta \mu \delta \nu}(\nu_t,\mu_t,x',w) \Psi_{\sigma}(\nu_t, \mu_t)(x') \partial_t \mu_t(w)\mathrm{d}x'\mathrm{d}w\\ 
        &- \int_{\mathcal{X}} \frac{\delta^2 F}{\delta \nu^2}(\nu_t,\mu_t,x,z)\partial_t \nu_t(z)\mathrm{d}z - \int_{\mathcal{Y}} \frac{\delta^2 F}{\delta \mu \delta \nu}(\nu_t,\mu_t,x,w)\partial_t \mu_t(w)\mathrm{d}w\Bigg)\\
        &= \frac{2\alpha}{\sigma^2}\Bigg(\int_{\mathcal{X}} \int_{\mathcal{X}} \frac{\delta^2 F}{\delta \nu^2}(\nu_t,\mu_t,x',z) \Psi_{\sigma}(\nu_t, \mu_t)(x')\Psi_{\sigma}(\nu_t, \mu_t) (z)\mathrm{d}x'\mathrm{d}z\\ 
        &- \int_{\mathcal{X}} \int_{\mathcal{X}} \frac{\delta^2 F}{\delta \nu^2}(\nu_t,\mu_t,x',z) \Psi_{\sigma}(\nu_t, \mu_t)(x')\nu_t(z)\mathrm{d}x'\mathrm{d}z\\
        &- \int_{\mathcal{X}} \frac{\delta^2 F}{\delta \nu^2}(\nu_t,\mu_t,x,z)\Psi_{\sigma}(\nu_t, \mu_t)(z)\mathrm{d}z + \int_{\mathcal{X}} \frac{\delta^2 F}{\delta \nu^2}(\nu_t,\mu_t,x,z)\nu_t(z)\mathrm{d}z\Bigg)\\ 
        &+ \frac{2\alpha}{\sigma^2}\Bigg(\int_{\mathcal{Y}} \int_{\mathcal{X}} \frac{\delta^2 F}{\delta \mu \delta \nu}(\nu_t,\mu_t,x',w) \Psi_{\sigma}(\nu_t, \mu_t)(x') \Phi_{\sigma}(\nu_t, \mu_t)(w) \mathrm{d}x'\mathrm{d}w\\
        &- \int_{\mathcal{Y}} \int_{\mathcal{X}} \frac{\delta^2 F}{\delta \mu \delta \nu}(\nu_t,\mu_t,x',w) \Psi_{\sigma}(\nu_t, \mu_t)(x') \mu_t(w)\mathrm{d}x'\mathrm{d}w\\
        &- \int_{\mathcal{Y}} \frac{\delta^2 F}{\delta \mu \delta \nu}(\nu_t,\mu_t,x,w)\Phi_{\sigma}(\nu_t, \mu_t)(w)\mathrm{d}w + \int_{\mathcal{Y}} \frac{\delta^2 F}{\delta \mu \delta \nu}(\nu_t,\mu_t,x,w)\mu_t(w)\mathrm{d}w\Bigg) ,
    \end{aligned}
    \end{multline}
    where the second equality follows from the MF-BR flow \eqref{eq: mean-field-br}. 
    From Assumption \ref{assump:F} and the fact that $\Psi_{\sigma}, \Phi_{\sigma}, \nu$ and $\mu$ are all probability density functions, it follows that
    \begin{equation*}
         -\frac{8\alpha}{\sigma^2}C_{\nu, \nu} -\frac{8\alpha}{\sigma^2}C_{\mu, \nu} \leq \partial_t \log \Psi_{\sigma}(\nu_t,\mu_t)(x) \leq \frac{8\alpha}{\sigma^2}C_{\nu, \nu} + \frac{8\alpha}{\sigma^2}C_{\mu, \nu}.
    \end{equation*}
    Multiplying the last inequality by $\nu_t(x) \geq 0$ and using \eqref{eq: new-bound-nu}, we get that
    \begin{multline*}
        -\left(\frac{8\alpha}{\sigma^2}C_{\nu, \nu} +\frac{8\alpha}{\sigma^2}C_{\mu, \nu}\right)\left(K_{\Psi_{\sigma}} e^{-U^{\pi}\left(x\right)} + \nu_0\left(x\right)\right) \leq \nu_t(x)\partial_t \log \Psi_{\sigma}(\nu_t,\mu_t)(x)\\ 
        \leq \left(\frac{8\alpha}{\sigma^2}C_{\nu, \nu} +\frac{8\alpha}{\sigma^2}C_{\mu, \nu}\right)\left(K_{\Psi_{\sigma}} e^{-U^{\pi}\left(x\right)} + \nu_0\left(x\right)\right) =: h_4(x).
    \end{multline*}
    Putting everything together, we finally obtain that
    \begin{equation*}
        h_{\Psi_{\sigma}}(x) \leq \partial_t \left(\nu_t(x) \log \frac{\nu_t(x)}{\Psi_{\sigma}(\nu_t,\mu_t)(x)}\right) \leq \Tilde{h}_{\Psi_{\sigma}}(x),
    \end{equation*}
    where 
    \begin{equation*}
        h_{\Psi_{\sigma}}(x) = h_1(x) + \alpha f(x) + h_2(x) + h_3(x) - h_4(x),
    \end{equation*}
    \begin{equation*}
        \Tilde{h}_{\Psi_{\sigma}}(x) = \Tilde{h}_1(x) + \alpha g(x) + \Tilde{h}_2(x) + \Tilde{h}_3(x) + h_4(x).
    \end{equation*}
    \emph{Step 2: The Lyapunov function decreases along the MF-BR dynamics:}
    Since the map $t \mapsto \operatorname{D_{KL}}({\nu_t}|{\Psi_{\sigma}(\nu_t,\mu_t)}) + \operatorname{D_{KL}}({\mu_t}|{\Phi_{\sigma}(\nu_t,\mu_t)})$ is differentiable for all $t > 0,$ we have that
    \begin{multline}
    \label{eq: derivative-KL}
    \begin{aligned}
        &\frac{\mathrm{d}}{\mathrm{d}t} \left(\operatorname{D_{KL}}({\nu_t}|{\Psi_{\sigma}(\nu_t,\mu_t)}) + \operatorname{D_{KL}}({\mu_t}|{\Phi_{\sigma}(\nu_t,\mu_t)})\right)\\ &= \int_{\mathcal{X}} \partial_t \left(\nu_t(x) \log \frac{\nu_t(x)}{\Psi_{\sigma}(\nu_t,\mu_t)(x)}\right) \mathrm{d}x 
        + \int_{\mathcal{Y}} \partial_t \left(\mu_t(y) \log \frac{\mu_t(y)}{\Phi_{\sigma}(\nu_t,\mu_t)(y)}\right) \mathrm{d}y\\
        &= \int_{\mathcal{X}} \log \frac{\nu_t(x)}{\Psi_{\sigma}(\nu_t,\mu_t)(x)}\partial_t \nu_t(x) \mathrm{d}x + \int_{\mathcal{X}} \nu_t(x)\left(\frac{\partial_t \nu_t(x)}{\nu_t(x)} - \partial_t \log \Psi_{\sigma}(\nu_t,\mu_t)(x)\right) \mathrm{d}x\\
        &+ \int_{\mathcal{Y}} \log \frac{\mu_t(y)}{\Phi_{\sigma}(\nu_t,\mu_t)(y)}\partial_t \mu_t(y) \mathrm{d}y + \int_{\mathcal{Y}} \mu_t(y)\left(\frac{\partial_t \mu_t(y)}{\mu_t(y)} - \partial_t \log \Phi_{\sigma}(\nu_t,\mu_t)(y)\right) \mathrm{d}y\\
        &= \alpha\int_{\mathcal{X}} \log \frac{\nu_t(x)}{\Psi_{\sigma}(\nu_t,\mu_t)(x)}\left(\Psi_{\sigma}(\nu_t,\mu_t)(x) - \nu_t(x)\right) \mathrm{d}x - \int_{\mathcal{X}} \nu_t(x)\partial_t \log \Psi_{\sigma}(\nu_t,\mu_t)(x) \mathrm{d}x\\
        &+ \alpha\int_{\mathcal{Y}} \log \frac{\mu_t(y)}{\Phi_{\sigma}(\nu_t,\mu_t)(y)}\left(\Phi_{\sigma}(\nu_t,\mu_t)(y) - \mu_t(y)\right) \mathrm{d}y - \int_{\mathcal{Y}} \mu_t(y)\partial_t \log \Phi_{\sigma}(\nu_t,\mu_t)(y) \mathrm{d}y\\
        &= -\alpha\big(\operatorname{D_{KL}}({\nu_t}|{\Psi_{\sigma}(\nu_t, \mu_t)}) + \operatorname{D_{KL}}({\Psi_{\sigma}(\nu_t, \mu_t)}|{\nu_t}) + \operatorname{D_{KL}}({\mu_t}|{\Phi_{\sigma}(\nu_t, \mu_t)}) + \operatorname{D_{KL}}({\Phi_{\sigma}(\nu_t, \mu_t)|{\mu_t}})\big)\\
        &- \int_{\mathcal{X}} \nu_t(x)\partial_t \log \Psi_{\sigma}(\nu_t,\mu_t)(x) \mathrm{d}x - \int_{\mathcal{Y}} \mu_t(y)\partial_t \log \Phi_{\sigma}(\nu_t,\mu_t)(y) \mathrm{d}y. 
    \end{aligned}
    \end{multline}
    where the third equality follows from the MF-BR flow \eqref{eq: mean-field-br} and the fact that $\int_{\mathcal{X}}\partial_t \nu_t(x) \mathrm{d}x = \int_{\mathcal{Y}}\partial_t \mu_t(y) \mathrm{d}y = 0$. Using only the first equality from \eqref{eq: log-Psi}, we obtain that
    \begin{align*}
        -\int_{\mathcal{X}} \nu_t(x)\partial_t \log \Psi_{\sigma}(\nu_t,\mu_t)(x) \mathrm{d}x &= -\frac{2}{\sigma^2}\Bigg(\int_{\mathcal{X}} \int_{\mathcal{X}} \frac{\delta^2 F}{\delta \nu^2}(\nu_t,\mu_t,x,z) \Psi_{\sigma}(\nu_t, \mu_t)(x)\partial_t \nu_t(z)\mathrm{d}x\mathrm{d}z\\ 
        &+ \int_{\mathcal{Y}} \int_{\mathcal{X}} \frac{\delta^2 F}{\delta \mu \delta \nu}(\nu_t,\mu_t,x,w) \Psi_{\sigma}(\nu_t, \mu_t)(x) \partial_t \mu_t(w)\mathrm{d}x\mathrm{d}w\\ 
        &- \int_{\mathcal{X}} \int_{\mathcal{X}} \frac{\delta^2 F}{\delta \nu^2}(\nu_t,\mu_t,x,z)\partial_t \nu_t(z)\nu_t(x)\mathrm{d}z\mathrm{d}x\\ &- \int_{\mathcal{X}} \int_{\mathcal{Y}} \frac{\delta^2 F}{\delta \mu \delta \nu}(\nu_t,\mu_t,x,w)\partial_t \mu_t(w) \nu_t(x)\mathrm{d}w\mathrm{d}x\Bigg).
    \end{align*}
    From Assumption \ref{assump:F} and the fact that $\Psi_{\sigma}(\nu_t, \mu_t)$ and $\nu_t$ are probability density functions, we have that
    \begin{align*}
        &\int_{\mathcal{X}} \int_{\mathcal{X}} \left|\frac{\delta^2 F}{\delta \nu^2}(\nu_t,\mu_t,x,z)\partial_t \nu_t(z)\nu_t(x)\right|\mathrm{d}z\mathrm{d}x\\
        &= \int_{\mathcal{X}} \int_{\mathcal{X}} \alpha\left|\frac{\delta^2 F}{\delta \nu^2}(\nu_t,\mu_t,x,z)\left(\Psi_{\sigma}(\nu_t, \mu_t)(z) - \nu_t(z)\right)\nu_t(x)\right|\mathrm{d}z\mathrm{d}x\\
        &\leq \int_{\mathcal{X}} \int_{\mathcal{X}} \alpha\left|\frac{\delta^2 F}{\delta \nu^2}(\nu_t,\mu_t,x,z)\right|\Psi_{\sigma}(\nu_t, \mu_t)(z)\nu_t(x)\mathrm{d}z\mathrm{d}x 
        + \int_{\mathcal{X}} \int_{\mathcal{X}} \alpha\left|\frac{\delta^2 F}{\delta \nu^2}(\nu_t,\mu_t,x,z)\right|\nu_t(z)\nu_t(x)\mathrm{d}z\mathrm{d}x\\
        &\leq 2\alpha C_{\nu, \nu} < \infty.
    \end{align*}
    Similarly, we have that
    \begin{equation*}
        \int_{\mathcal{X}} \int_{\mathcal{Y}} \left|\frac{\delta^2 F}{\delta \mu \delta \nu}(\nu_t,\mu_t,x,w)\partial_t \mu_t(w) \nu_t(x)\right|\mathrm{d}w\mathrm{d}x \leq 2\alpha C_{\mu, \nu} < \infty.
    \end{equation*}
    Therefore, we can apply Fubini's theorem and obtain that
    \begin{multline}
    \label{eq: nu-times-Psi}
    \begin{aligned}
        &-\int_{\mathcal{X}} \nu_t(x)\partial_t \log \Psi_{\sigma}(\nu_t,\mu_t)(x) \mathrm{d}x\\
        &= -\frac{2}{\sigma^2}\Bigg(\int_{\mathcal{X}} \int_{\mathcal{X}} \frac{\delta^2 F}{\delta \nu^2}(\nu_t,\mu_t,x,z) \left(\Psi_{\sigma}(\nu_t, \mu_t)(x) - \nu_t(x)\right)\partial_t \nu_t(z)\mathrm{d}x\mathrm{d}z\\ 
        &+ \int_{\mathcal{Y}} \int_{\mathcal{X}} \frac{\delta^2 F}{\delta \mu \delta \nu}(\nu_t,\mu_t,x,w) \left(\Psi_{\sigma}(\nu_t, \mu_t)(x) - \nu_t(x)\right) \partial_t \mu_t(w)\mathrm{d}x\mathrm{d}w\Bigg)\\
        &= -\frac{2\alpha}{\sigma^2}\int_{\mathcal{X}} \int_{\mathcal{X}} \frac{\delta^2 F}{\delta \nu^2}(\nu_t,\mu_t,x,z) \left(\Psi_{\sigma}(\nu_t, \mu_t)(x) - \nu_t(x)\right)\left(\Psi_{\sigma}(\nu_t, \mu_t)(z) - \nu_t(z)\right)\mathrm{d}x\mathrm{d}z\\ 
        &-\frac{2\alpha}{\sigma^2} \int_{\mathcal{Y}} \int_{\mathcal{X}} \frac{\delta^2 F}{\delta \mu \delta \nu}(\nu_t,\mu_t,x,w) \left(\Psi_{\sigma}(\nu_t, \mu_t)(x) - \nu_t(x)\right)\left(\Phi_{\sigma}(\nu_t, \mu_t)(w) - \mu_t(w)\right)\mathrm{d}x\mathrm{d}w\\
        &\leq -\frac{2\alpha}{\sigma^2}\int_{\mathcal{Y}} \int_{\mathcal{X}} \frac{\delta^2 F}{\delta \mu \delta \nu}(\nu_t,\mu_t,x,w) \left(\Psi_{\sigma}(\nu_t, \mu_t)(x) - \nu_t(x)\right)\left(\Phi_{\sigma}(\nu_t, \mu_t)(w) - \mu_t(w)\right)\mathrm{d}x\mathrm{d}w,
    \end{aligned}
    \end{multline}
    where the last inequality follows from 
    \begin{equation*}
        \int_{\mathcal{X}} \int_{\mathcal{X}} \frac{\delta^2 F}{\delta \nu^2}(\nu_t,\mu_t,x,z) \left(\Psi_{\sigma}(\nu_t, \mu_t)(x) - \nu_t(x)\right)\left(\Psi_{\sigma}(\nu_t, \mu_t)(z) - \nu_t(z)\right)\mathrm{d}x\mathrm{d}z \geq 0
    \end{equation*} 
    due to the convexity of the map $\nu \mapsto F(\nu, \mu)$.

   Performing similar calculations to \eqref{eq: log-Psi}, \eqref{eq: nu-times-Psi}, and using Assumption \ref{assump:F}, we can apply Fubini's theorem and obtain that
    \begin{multline}
    \label{eq: mu-times-Phi}
    \begin{aligned}
        &-\int_{\mathcal{Y}} \mu_t(y)\partial_t \log \Phi_{\sigma}(\nu_t,\mu_t)(y) \mathrm{d}y\\
        &= \frac{2}{\sigma^2}\Bigg(\int_{\mathcal{Y}} \int_{\mathcal{Y}} \frac{\delta^2 F}{\delta \mu^2}(\nu_t,\mu_t,y,w) \left(\Phi_{\sigma}(\nu_t, \mu_t)(y) - \mu_t(y)\right)\partial_t \mu_t(w)\mathrm{d}y\mathrm{d}w\\
        &+ \int_{\mathcal{X}} \int_{\mathcal{Y}} \frac{\delta^2 F}{\delta \nu \delta \mu}(\nu_t,\mu_t,y,z) \left(\Phi_{\sigma}(\nu_t, \mu_t)(y) - \mu_t(y)\right) \partial_t \nu_t(z)\mathrm{d}y\mathrm{d}z\Bigg)\\
        &\leq \frac{2\alpha}{\sigma^2}\int_{\mathcal{X}} \int_{\mathcal{Y}} \frac{\delta^2 F}{\delta \nu \delta \mu}(\nu_t,\mu_t,y,z) \left(\Phi_{\sigma}(\nu_t, \mu_t)(y) - \mu_t(y)\right)\left(\Psi_{\sigma}(\nu_t, \mu_t)(z) - \nu_t(z)\right)\mathrm{d}y\mathrm{d}z,
    \end{aligned}
    \end{multline}
    where the last inequality follows from 
    \begin{equation*}
        \int_{\mathcal{Y}} \int_{\mathcal{Y}} \frac{\delta^2 F}{\delta \mu^2}(\nu_t,\mu_t,y,w) \left(\Phi_{\sigma}(\nu_t, \mu_t)(y) - \mu_t(y)\right)\left(\Phi_{\sigma}(\nu_t, \mu_t)(w) - \mu_t(w)\right)\mathrm{d}y\mathrm{d}w \leq 0
    \end{equation*} 
    due to the concavity of the map $\mu \mapsto F(\nu, \mu)$.

    Combining \eqref{eq: derivative-KL}, \eqref{eq: nu-times-Psi} and \eqref{eq: mu-times-Phi} gives that
    \begin{align*}
        &\frac{\mathrm{d}}{\mathrm{d}t} \left(\operatorname{D_{KL}}({\nu_t}|{\Psi_{\sigma}(\nu_t,\mu_t)}) + \operatorname{D_{KL}}({\mu_t}|{\Phi_{\sigma}(\nu_t,\mu_t)})\right)\\
        &= -\alpha\big(\operatorname{D_{J}}({\nu_t}|{\Psi_{\sigma}(\nu_t, \mu_t)}) + \operatorname{D_{J}}({\mu_t}|{\Phi_{\sigma}(\nu_t, \mu_t)})\big)\\ &- \int_{\mathcal{X}} \nu_t(x)\partial_t \log \Psi_{\sigma}(\nu_t,\mu_t)(x) \mathrm{d}x - \int_{\mathcal{Y}} \mu_t(y)\partial_t \log \Phi_{\sigma}(\nu_t,\mu_t)(y) \mathrm{d}y\\
        &\leq -\alpha\big(\operatorname{D_{J}}({\nu_t}|{\Psi_{\sigma}(\nu_t, \mu_t)}) + \operatorname{D_{J}}({\mu_t}|{\Phi_{\sigma}(\nu_t, \mu_t)})\big)\\ 
        &- \frac{2\alpha}{\sigma^2}\int_{\mathcal{Y}} \int_{\mathcal{X}} \frac{\delta^2 F}{\delta \mu \delta \nu}(\nu_t,\mu_t,x,w) \left(\Psi_{\sigma}(\nu_t, \mu_t)(x) - \nu_t(x)\right)\left(\Phi_{\sigma}(\nu_t, \mu_t)(w) - \mu_t(w)\right)\mathrm{d}x\mathrm{d}w\\ 
        &+ \frac{2\alpha}{\sigma^2}\int_{\mathcal{X}} \int_{\mathcal{Y}} \frac{\delta^2 F}{\delta \nu \delta \mu}(\nu_t,\mu_t,y,z) \left(\Phi_{\sigma}(\nu_t, \mu_t)(y) - \mu_t(y)\right)\left(\Psi_{\sigma}(\nu_t, \mu_t)(z) - \nu_t(z)\right)\mathrm{d}y\mathrm{d}z.
    \end{align*}
    Again, using Assumption \ref{assump:F} to justify the use of Fubini's theorem and Lemma \ref{lemma: symmetry-flat}, the last two terms cancel and we obtain that
    \begin{equation*}
        \frac{\mathrm{d}}{\mathrm{d}t} \left(\operatorname{D_{KL}}({\nu_t}|{\Psi_{\sigma}(\nu_t,\mu_t)}) + \operatorname{D_{KL}}({\mu_t}|{\Phi_{\sigma}(\nu_t,\mu_t)})\right)
        \leq -\alpha\big(\operatorname{D_{J}}({\nu_t}|{\Psi_{\sigma}(\nu_t, \mu_t)}) + \operatorname{D_{J}}({\mu_t}|{\Phi_{\sigma}(\nu_t, \mu_t)})\big).
    \end{equation*}
    \emph{Step 3: Convergence of the MF-BR dynamics in $\operatorname{D_{KL}}$ and $\operatorname{NI}$ error:}
    From the inequality above, we have that
    \begin{multline*}
        \frac{\mathrm{d}}{\mathrm{d}t} \left(\operatorname{D_{KL}}({\nu_t}|{\Psi_{\sigma}(\nu_t,\mu_t)}) + \operatorname{D_{KL}}({\mu_t}|{\Phi_{\sigma}(\nu_t,\mu_t)})\right) \leq -\alpha\big(\operatorname{D_{KL}}({\nu_t}|{\Psi_{\sigma}(\nu_t, \mu_t)}) + \operatorname{D_{KL}}({\Psi_{\sigma}(\nu_t, \mu_t)}|{\nu_t})\\
+ \operatorname{D_{KL}}({\mu_t}|{\Phi_{\sigma}(\nu_t, \mu_t)}) + \operatorname{D_{KL}}({\Phi_{\sigma}(\nu_t, \mu_t)|{\mu_t}})\big)
\leq -\alpha\big(\operatorname{D_{KL}}({\nu_t}|{\Psi_{\sigma}(\nu_t, \mu_t)}) + \operatorname{D_{KL}}({\mu_t}|{\Phi_{\sigma}(\nu_t, \mu_t)})\big),
    \end{multline*}
    and hence we deduce the conclusion from Gronwall's inequality. The convergence with respect to $t \mapsto \operatorname{NI}({\nu_t}, {\mu_t})$ with rate $\mathcal{O}\left(\frac{\sigma^2}{2}e^{-\alpha t}\right)$ follows from Lemma \ref{lemma: NI-equals-KL}.
\end{proof}
Next, we obtain convergence with rate $\mathcal{O}\left(e^{-\alpha t}\right)$ of the MF-BR flow to the MNE of \eqref{eq:F-nonlinear} in terms of $\operatorname{D_{KL}}$ divergence and TV distance.
\begin{proof}[Proof of Corollary \ref{corollary: KL-TV}]
    The convergence with respect to $t \mapsto \operatorname{D_{KL}}({\nu_t}|\nu_{\sigma}^*) + \operatorname{D_{KL}}({\mu_t}|\mu_{\sigma}^*)$ with rate $\mathcal{O}\left(e^{-\alpha t}\right)$ follows from Lemma \ref{lemma: NI-equals-KL}. The convergence with respect to $t \mapsto \operatorname{TV}^2({\nu_t},\nu_{\sigma}^*) + \operatorname{TV}^2({\mu_t},\mu_{\sigma}^*)$ with rate $\mathcal{O}\left(e^{-\alpha t}\right)$ follows from Pinsker's inequality.
\end{proof}
Next, we obtain convergence with rate $\mathcal{O}\left(e^{-\alpha t}\right)$ (independent of $\sigma$) of the MF-BR flow in terms of the NI error when $F$ is bilinear.
\begin{proof}[Proof of Corollary \ref{corollary: convergence-NI-linear-F-FP}]
    If we set
    \begin{equation*}
        F(\nu, \mu) = \int_{\mathcal{Y}}\int_{\mathcal{X}} f(x,y) \nu(\mathrm{d}x) \mu(\mathrm{d}y),
    \end{equation*} 
    with $f:\mathcal{X} \times \mathcal{Y} \to \mathbb{R}$ bounded, then Assumption \ref{assumption: assump-F-conv-conc}, \ref{assumption: boundedness-first-flat}, \ref{assump:F} still hold according to Remark \ref{remark:example-f}, and moreover we have equality in Lemma \ref{lemma: NI-equals-KL}. Therefore, the convergence estimate with respect to the NI error reads
    \begin{align*}
        \operatorname{NI}({\nu_t}, {\mu_t}) &\leq \frac{\sigma^2}{2}e^{-\alpha t}\left(\operatorname{D_{KL}}({\nu_0}|{\Psi_{\sigma}(\nu_0, \mu_0)}) + \operatorname{D_{KL}}({\mu_0}|{\Phi_{\sigma}(\nu_0, \mu_0)})\right)\\ &= \frac{\sigma^2}{2}e^{-\alpha t} \frac{2}{\sigma^2} \operatorname{NI}({\nu_0}, {\mu_0})\\
        &= e^{-\alpha t}\operatorname{NI}({\nu_0}, {\mu_0}).
    \end{align*}
\end{proof}

\section*{Acknowledgements}
We would like to thank the anonymous referees for their comments which helped us to improve the paper. R-AL was supported by the EPSRC Centre for Doctoral Training in Mathematical Modelling, Analysis and Computation (MAC-MIGS) funded by the UK Engineering and Physical Sciences Research Council (grant EP/S023291/1), Heriot-Watt University and the University of Edinburgh. LS acknowledges the support of the UKRI Prosperity Partnership Scheme (FAIR) under EPSRC Grant EP/V056883/1 and the Alan Turing Institute.

\bibliographystyle{abbrv}
\bibliography{sn-article-template/sn-bibliography}

\appendix
\numberwithin{equation}{section}
\makeatletter 
\newcommand{\section@cntformat}{Appendix \thesection:\ }
\makeatother
\section{Technical results and proofs}
\label{appendix: AppC}
In this section, we present the proofs of the remaining results formulated in Section \ref{sec:Main-Results} of the paper. 

\subsection{Auxiliary results.}
In this subsection, we present the proofs of Lemma \ref{proposition:PsiPhiBounds}, Corollary \ref{corollary: A.6}, Lemma \ref{lemma:integrability-derivative-of-entropy} and Lemma \ref{lemma: NI-equals-KL} but first we recall a result based on \cite{conforti_game_2020}, which will be useful throughout the appendix, regarding the characterization of MNEs via first-order conditions (see also \cite[Proposition $2.5$]{10.1214/20-AIHP1140}). 

\begin{proposition}[\cite{conforti_game_2020}, Theorem $3.1$]
\label{prop: 2.5}
Assume that $F$ admits first-order flat derivative (cf. Definition \ref{def:fderivative} in Appendix \ref{app: AppB}) and that Assumption \ref{assumption: assump-F-conv-conc} holds. Then, the pair $(\nu_{\sigma}^*, \mu_{\sigma}^*) \in \mathcal{P}_{\text{ac}}(\mathcal{X}) \times \mathcal{P}_{\text{ac}}(\mathcal{Y})$ is a MNE of \eqref{eq:F-nonlinear}, i.e., $\nu_{\sigma}^* \in \argmin_{\nu' \in \mathcal{P}(\mathcal{X})} V^{\sigma}(\nu',\mu_{\sigma}^*)$ and $\mu_{\sigma}^* \in \argmax_{\mu' \in \mathcal{P}(\mathcal{Y})} V^{\sigma}(\nu_{\sigma}^*,\mu')$, if and only if it satisfies the following first-order condition for all $(x, y) \in \mathcal{X} \times \mathcal{Y}$ Lebesgue almost surely:
\begin{equation*}
    \frac{\delta F}{\delta \nu}(\nu_{\sigma}^*, \mu_{\sigma}^*, x) + \frac{\sigma^2}{2}\log \left(\frac{\nu_{\sigma}^*(x)}{\pi(x)}\right) = \operatorname{constant}, 
\end{equation*}
\begin{equation*}
    \frac{\delta F}{\delta \mu}(\nu_{\sigma}^*, \mu_{\sigma}^*, y) - \frac{\sigma^2}{2}\log \left(\frac{\mu_{\sigma}^*(y)}{\rho(y)}\right) = \operatorname{constant}.
\end{equation*}
\end{proposition}
\begin{proof}[Proof of Lemma \ref{proposition:PsiPhiBounds}]
From Assumption \ref{assumption: boundedness-first-flat}, we have the estimates
\begin{equation}
\label{eq:bound-psi-unnormalized-density}
\exp \left( -\frac{2}{\sigma^2} C_{\nu} - U^{\pi} \left(x\right) \right) \leq \exp \left(- \frac{2}{\sigma^2}\frac{\delta F}{\delta \nu}(\nu, \mu, x)  - U^{\pi} \left(x\right) \right) \leq \exp \left( \frac{2}{\sigma^2} C_{\nu} - U^{\pi} \left(x\right) \right),
\end{equation}
\begin{equation}
\label{eq:bound-psi-normalization-constant}
\exp\left(-\frac{2C_{\nu}}{\sigma^2}\right) \leq Z(\nu, \mu) \leq \exp\left(\frac{2C_{\nu}}{\sigma^2}\right).
\end{equation}
Thus, we obtain \eqref{eq: estimate-psi} with constant \(K_{\Psi_{\sigma}} = \frac{1}{k_{\Psi_{\sigma}}} = \exp \left( \frac{4}{\sigma^2} C_{\nu} \right)> 1\). Moreover, by construction, 
\[
\int_{\mathcal{X}} \Psi_{\sigma}(\nu, \mu) \left(\mathrm{d}x\right) = \int_{\mathcal{X}} \Psi_{\sigma}(\nu, \mu) \left(x\right) \mathrm{d}x 
= \frac{1}{Z(\nu,\mu)}\int_{\mathcal{X}}  \exp \left(- \frac{2}{\sigma^2} \frac{\delta F}{\delta \nu}(\nu, \mu, x) - U^{\pi} \left(x\right)\right) \mathrm{d}x = 1,
\]
and therefore \(\Psi_{\sigma}(\nu, \mu)\in \mathcal P \left(\mathcal{X}\right)\). One can argue similarly for $\Phi_{\sigma}(\nu, \mu)$. 
\end{proof}

\begin{proof}[Proof of Corollary \ref{corollary: A.6}]
From \eqref{eq: estimate-psi} and \eqref{eq: estimate-phi}, we have that, for all $t \geq 0,$
\begin{equation*}
    k_{\Psi_{\sigma}}e^{-U^{\pi}(x)} \leq \Psi_{\sigma}(\nu_t, \mu_t)(x), \quad k_{\Phi_{\sigma}}e^{-U^{\rho}(y)} \leq \Phi_{\sigma}(\nu_t, \mu_t)(y).
\end{equation*}
By Duhamel's formula we can rewrite equations in \eqref{eq: mean-field-br} as
\begin{equation}
\label{equation:duhamel-nu}
\nu_t(x) = e^{-\alpha t}\nu_0(x) + \int_0^t \alpha e^{-\alpha (t-s)} \Psi_{\sigma}(\nu_s, \mu_s)(x) \mathrm{d}s,
\end{equation}
\begin{equation}
\label{equation:duhamel-mu}
\mu_t(y) = e^{-\alpha t}\mu_0(y) + \int_0^t \alpha e^{-\alpha (t-s)} \Phi_{\sigma}(\nu_s, \mu_s)(y) \mathrm{d}s.
\end{equation}
Therefore, using \eqref{equation:duhamel-nu} and \eqref{equation:duhamel-mu}, it follows that
\begin{equation*}
\nu_t \left(x\right) \geq \int_0^t \alpha e^{-\alpha (t-s)} \Psi_{\sigma}(\nu_s, \mu_s)\left(x\right) \mathrm{d}s \geq k_{\Psi_{\sigma}}e^{-U^{\pi}(x)} \int_0^t \alpha e^{-\alpha (t-s)} \mathrm{d}s = \left( 1 - e^{-\alpha t}\right) k_{\Psi_{\sigma}} e^{-U^{\pi}\left(x\right)},
\end{equation*}
\begin{equation*}
\mu_t \left(y\right) \geq \int_0^t \alpha e^{-\alpha (t-s)} \Phi_{\sigma}(\nu_s, \mu_s)\left(y\right) \mathrm{d}s
\geq k_{\Phi_{\sigma}}e^{-U^{\rho}(y)} \int_0^t \alpha e^{-\alpha (t-s)} \mathrm{d}s = \left(1 - e^{-\alpha t}\right) k_{\Phi_{\sigma}} e^{-U^{\rho}\left(y\right)}.
\end{equation*}
The proof for the upper bounds is similar.
\end{proof}

\begin{proof}[Proof of Lemma \ref{lemma:integrability-derivative-of-entropy}]
First, we derive lower and upper bounds on $\Psi_{\sigma}(\nu_t, \mu_t)(x)\log\frac{\nu_t\left(x\right)}{e^{-U^{\pi}\left(x\right)}}.$ Using the bounds given by \eqref{equation:nu-density-bound} and \eqref{eq: estimate-psi}, we have 
\begin{align*}
\Psi_{\sigma}(\nu_t, \mu_t)(x)\log\frac{\nu_t\left(x\right)}{e^{-U^{\pi}\left(x\right)}}
&\geq \Psi_{\sigma}(\nu_t, \mu_t)(x)\log \frac{ \left( 1 - e^{-\alpha t} \right) k_{\Psi_{\sigma}} e^{-U^{\pi}\left(x\right)} }{ e^{-U^{\pi}\left(x\right)} }\\
&= \Psi_{\sigma}(\nu_t, \mu_t)(x)\log \left(\left(1 - e^{-\alpha t}\right) k_{\Psi_{\sigma}}\right)\\ &\geq \log \left( \left(1 - e^{-\alpha t} \right) k_{\Psi_{\sigma}}  \right) K_{\Psi_{\sigma}} e^{-U^{\pi}\left(x\right)} =: g_1\left(x\right),
\end{align*}
where the last inequality follows from the fact that $k_{\Psi_{\sigma}} \in (0,1)$ so that $\log \left( \left(1 - e^{-\alpha t} \right) k_{\Psi_{\sigma}}  \right) < 0.$ 

The upper bound is obtained as follows. From Duhamel's formula \eqref{equation:duhamel-nu} and \eqref{eq: estimate-psi}, we have that
\begin{align*}
\log\frac{\nu_t\left(x\right)}{e^{-U^{\pi}\left(x\right)}}
&= \log \left( e^{-\alpha t}\frac{\nu_0 \left(x\right)}{e^{-U^{\pi}\left(x\right)}} + \int_0^t \alpha e^{-\alpha (t-s)}\frac{\Psi_{\sigma}(\nu_s, \mu_s) \left(x\right)}{e^{-U^{\pi}\left(x\right)}} ds \right) \\
&\leq \log \left( e^{-\alpha t}\frac{\nu_0 \left(x\right)}{e^{-U^{\pi}\left(x\right)}} + \int_0^t \alpha e^{-\alpha (t-s)} K_{\Psi_{\sigma}} ds \right)\\ &= \log \left( e^{-\alpha t}\frac{\nu_0 \left(x\right)}{e^{-U^{\pi}\left(x\right)}} + \left(1 - e^{-\alpha t}\right) K_{\Psi_{\sigma}} \right) \\
&\leq \log \left(\left(1 - e^{-\alpha t}\right) K_{\Psi_{\sigma}} \right) + \frac{\nu_0 \left(x\right)}{e^{-U^{\pi}\left(x\right)}} \frac{e^{-\alpha t}}{\left(1 - e^{-\alpha t}\right) K_{\Psi_{\sigma}}}\\
&\leq \log K_{\Psi_{\sigma}} + \frac{\nu_0(x)}{K_{\Psi_{\sigma}}e^{-U^{\pi}\left(x\right)}}\kappa_s,
\end{align*}
where in the second inequality we used the inequality \(\log \left(x + y\right) \leq \log x + \frac{y}{x}\) and in the last inequality we maximize over $t \geq s$ and take $\kappa_s \coloneqq \sup_{t \geq s} \frac{e^{-\alpha t}}{\left(1 - e^{-\alpha t}\right)} = \frac{e^{-\alpha s}}{\left(1 - e^{-\alpha s}\right)}$. Finally, the upper bound is given by
\begin{align*}
\Psi_{\sigma}(\nu_t, \mu_t)(x)\log\frac{\nu_t\left(x\right)}{e^{-U^{\pi}\left(x\right)}}
&\leq \left(\log K_{\Psi_{\sigma}} + \frac{\nu_0(x)}{K_{\Psi_{\sigma}}e^{-U^{\pi}\left(x\right)}}\kappa_s\right) \Psi_{\sigma}(\nu_t, \mu_t) \left(x\right)\\ 
&\leq \left(\log K_{\Psi_{\sigma}} + \frac{\nu_0(x)}{K_{\Psi_{\sigma}}e^{-U^{\pi}\left(x\right)}}\kappa_s\right) K_{\Psi_{\sigma}} e^{-U^{\pi}\left(x\right)}\\ &= K_{\Psi_{\sigma}} e^{-U^{\pi}\left(x\right)}\log K_{\Psi_{\sigma}} + \kappa_s\nu_0\left(x\right) =: f_1\left(x\right).
\end{align*}
Now, consider the second term $\nu_t\left(x\right)\log \frac{\nu_t\left(x\right)}{e^{-U^{\pi}\left(x\right)}}$. Observe that 
\begin{equation*}
    e^{-\alpha t} + \int_0^t \alpha e^{-\alpha\left(t-s\right)} \mathrm{d}s = 1.
\end{equation*}
Hence, applying the convexity of the map \(\psi: z \mapsto z\log z\) to Duhamel's formula \eqref{equation:duhamel-nu}, we have by Jensen's inequality that
\begin{align*}
\frac{\nu_t\left(x\right)}{e^{-U^{\pi}\left(x\right)}} \log \frac{\nu_t\left(x\right)}{e^{-U^{\pi}\left(x\right)}} &= \psi\left(\frac{\nu_t\left(x\right)}{e^{-U^{\pi}\left(x\right)}}\right)\\ &= \psi \left(e^{-\alpha t} \frac{\nu_0\left(x\right)}{e^{-U^{\pi}\left(x\right)}} + \int_0^t \alpha e^{-\alpha\left(t-s\right)}\frac{\Psi_{\sigma}(\nu_s, \mu_s)\left(x\right)}{e^{-U^{\pi}\left(x\right)}} \mathrm{d}s\right)\\
&\leq e^{-\alpha t} \psi \left(\frac{\nu_0\left(x\right)}{e^{-U^{\pi}\left(x\right)}}\right) + \int_0^t \alpha e^{-\alpha\left(t-s\right)}\psi\left(\frac{\Psi_{\sigma}(\nu_s, \mu_s)\left(x\right)}{e^{-U^{\pi}\left(x\right)}}\right) \mathrm{d}s\\
&= e^{-\alpha t} \frac{\nu_0\left(x\right)}{e^{-U^{\pi}\left(x\right)}}\log \frac{\nu_0\left(x\right)}{e^{-U^{\pi}\left(x\right)}} + \int_0^t \alpha e^{-\alpha\left(t-s\right)}\frac{\Psi_{\sigma}(\nu_s, \mu_s)\left(x\right)}{e^{-U^{\pi}\left(x\right)}} \log \frac{\Psi_{\sigma}(\nu_s, \mu_s)\left(x\right)}{e^{-U^{\pi}\left(x\right)}} \mathrm{d}s.
\end{align*}
Therefore, multiplying the inequality above by $e^{-U^{\pi}\left(x\right)},$ we obtain
\begin{align*}
    \nu_t\left(x\right) \log \frac{\nu_t\left(x\right)}{e^{-U^{\pi}\left(x\right)}} &\leq e^{-\alpha t} \nu_0\left(x\right)\log \frac{\nu_0\left(x\right)}{e^{-U^{\pi}\left(x\right)}} + \int_0^t \alpha e^{-\alpha\left(t-s\right)}\Psi_{\sigma}(\nu_s, \mu_s)\left(x\right) \log \frac{\Psi_{\sigma}(\nu_s, \mu_s)\left(x\right)}{e^{-U^{\pi}\left(x\right)}} \mathrm{d}s\\
    &\leq e^{-\alpha t} \nu_0\left(x\right)\log \frac{\nu_0\left(x\right)}{e^{-U^{\pi}\left(x\right)}} + \int_0^t \alpha e^{-\alpha \left(t-s\right)} \Psi_{\sigma}(\nu_s,\mu_s) \left(x\right)\log K_{\Psi_{\sigma}} \mathrm{d}s \\
    &\leq e^{-\alpha t} \nu_0\left(x\right)\log \frac{\nu_0\left(x\right)}{e^{-U^{\pi}\left(x\right)}} + \int_0^t \alpha e^{-\alpha \left(t-s\right)} K_{\Psi_{\sigma}} e^{-U^{\pi}\left(x\right)}\log K_{\Psi_{\sigma}} \mathrm{d}s \\
    &\leq \max\left(\nu_0\left(x\right)\log \frac{\nu_0\left(x\right)}{e^{-U^{\pi}\left(x\right)}},0\right) + K_{\Psi_{\sigma}} e^{-U^{\pi}\left(x\right)}\log K_{\Psi_{\sigma}} =: -g_2 \left(x\right),
\end{align*}
where the second and third inequalities follow from \eqref{eq: estimate-psi}. 

For the lower bound, we observe that 
\[
\nu_t\left(x\right) \log \frac{\nu_t\left(x\right)}{e^{-U^{\pi}\left(x\right)}} = \frac{\nu_t\left(x\right)}{e^{-U^{\pi}\left(x\right)}} e^{-U^{\pi}\left(x\right)}\log \frac{\nu_t\left(x\right)}{e^{-U^{\pi}\left(x\right)}} \geq -\frac 1e e^{-U^{\pi}\left(x\right)} =: -f_2 \left(x\right),
\]
where we used the fact that the map $z \mapsto z\log z$ is continuous with the global minimum at $z=1/e.$ The conclusion follows if  we set \(f(x) := f_1(x) + f_2(x)\) and \(g(y) := g_1(y) + g_2(y)\). One could argue similarly to obtain $\hat{g}$ and $\hat{f}.$ 
\end{proof}
\begin{proof}[Proof of Lemma \ref{lemma: NI-equals-KL}]
\emph{Step 1: Proof of the left-hand side inequality.}
Using \eqref{eq:convexF} from Assumption \ref{assumption: assump-F-conv-conc}, it follows that
    \begin{align*}
		V^{\sigma}(\nu, \mu_{\sigma}^*) - V^{\sigma}(\nu_{\sigma}^*, \mu_{\sigma}^*) &\geq \int_{\mathcal{X}} \frac{\delta F}{\delta \nu}(\nu_{\sigma}^*, \mu_{\sigma}^*,x) (\nu-\nu_{\sigma}^*)(\mathrm{d}x) + \frac{\sigma^2}{2}\operatorname{D_{KL}}(\nu|\pi) - \frac{\sigma^2}{2}\operatorname{D_{KL}}(\nu_{\sigma}^*|\pi)\\
        &= \int_{\mathcal{X}} \left(\frac{\delta F}{\delta \nu}(\nu_{\sigma}^*, \mu_{\sigma}^*,x) + \frac{\sigma^2}{2}\log\left(\frac{\nu_{\sigma}^*(x)}{\pi(x)}\right)\right) (\nu-\nu_{\sigma}^*)(\mathrm{d}x)\\ &- \frac{\sigma^2}{2}\int_{\mathcal{X}} \log\left(\frac{\nu_{\sigma}^*(x)}{\pi(x)}\right)(\nu-\nu_{\sigma}^*)(\mathrm{d}x)
        + \frac{\sigma^2}{2}\int_{\mathcal{X}} \log\left(\frac{\nu(x)}{\pi(x)}\right)\nu(\mathrm{d}x)\\ &- \frac{\sigma^2}{2}\int_{\mathcal{X}} \log\left(\frac{\nu_{\sigma}^*(x)}{\pi(x)}\right)\nu_{\sigma}^*(\mathrm{d}x)
        = \frac{\sigma^2}{2}\operatorname{D_{KL}}(\nu|\nu_{\sigma}^*),
	\end{align*}
 where the last equality follows from Proposition \ref{prop: 2.5}. Similarly, using \eqref{eq:concaveF} from Assumption \ref{assumption: assump-F-conv-conc} and Proposition \ref{prop: 2.5}, it follows that
        \begin{equation*}
            V^{\sigma}(\nu_{\sigma}^*, \mu) - V^{\sigma}(\nu_{\sigma}^*, \mu_{\sigma}^*) \leq \ - \frac{\sigma^2}{2}\operatorname{D_{KL}}(\mu|\mu_{\sigma}^*). 
        \end{equation*}
Recalling that $\operatorname{NI}(\nu, \mu) = \max_{\mu' \in \mathcal{P}(\mathcal{Y})}V^{\sigma}(\nu,\mu') - \min_{\nu' \in \mathcal{P}(\mathcal{X})}V^{\sigma}(\nu',\mu),$ and adding the previous inequalities gives
\begin{align*}
    \operatorname{NI}(\nu, \mu) = \max_{\mu' \in \mathcal{P}(\mathcal{Y})}V^{\sigma}(\nu,\mu') - \min_{\nu' \in \mathcal{P}(\mathcal{X})}V^{\sigma}(\nu',\mu) &\geq V^{\sigma}(\nu, \mu_{\sigma}^*) - V^{\sigma}(\nu_{\sigma}^*, \mu)\\ &\geq \frac{\sigma^2}{2}\left(\operatorname{D_{KL}}(\nu|\nu_{\sigma}^*) + \operatorname{D_{KL}}(\mu|\mu_{\sigma}^*)\right).
\end{align*}
\emph{Step 2: Proof of the right-hand side inequality.} From \eqref{eq:psi-linear-F} and \eqref{eq:phi-linear-F}, we have, for any $(\nu', \mu') \in \mathcal{P}(\mathcal{X}) \times \mathcal{P}(\mathcal{Y}),$ that
\begin{multline}
\label{eq:psi-F}
\begin{aligned}
    \int_{\mathcal{X}} \frac{\delta F}{\delta \nu}(\nu, \mu, x)(\Psi_{\sigma}(\nu, \mu)-\nu)(\mathrm{d}x) &+ \frac{\sigma^2}{2}\operatorname{D_{KL}}(\Psi_{\sigma}(\nu, \mu)|\pi)\\ 
    &\leq \int_{\mathcal{X}} \frac{\delta F}{\delta \nu}(\nu, \mu, x)(\nu'-\nu)(\mathrm{d}x) + \frac{\sigma^2}{2}\operatorname{D_{KL}}(\nu'|\pi), 
\end{aligned}
\end{multline}
\begin{multline}
\label{eq:phi-F}
\begin{aligned}
    \int_{\mathcal{Y}} \frac{\delta F}{\delta \mu}(\nu, \mu, y)(\Phi_{\sigma}(\nu, \mu)-\mu)(\mathrm{d}y) &- \frac{\sigma^2}{2}\operatorname{D_{KL}}(\Phi_{\sigma}(\nu, \mu)|\rho)\\ &\geq \int_{\mathcal{Y}} \frac{\delta F}{\delta \mu}(\nu, \mu, y)(\mu'-\mu)(\mathrm{d}y) - \frac{\sigma^2}{2}\operatorname{D_{KL}}(\mu'|\rho).
\end{aligned}
\end{multline}
Using Proposition \ref{prop: 2.5}, we observe that $\left(\Psi_{\sigma}, \Phi_{\sigma}\right)$ as defined in \eqref{eq:argminpsi-density} and \eqref{eq:argmaxphi-density} also satisfy
\begin{equation}
\label{eq: foc-psi}
    \frac{\delta F}{\delta \nu}(\nu, \mu, x) + \frac{\sigma^2}{2}\log \frac{\Psi_{\sigma}(\nu, \mu)(x)}{\pi(x)} = C,
\end{equation}
\begin{equation}
\label{eq: foc-phi}
    \frac{\delta F}{\delta \mu}(\nu, \mu, y) - \frac{\sigma^2}{2}\log \frac{\Phi_{\sigma}(\nu, \mu)(y)}{\rho(y)} = \Tilde{C},
\end{equation}
for all $(x,y) \in \mathcal{X} \times \mathcal{Y}$ Lebesgue almost surely, where $C, \Tilde{C} \in \mathbb R$. 

Given $(\nu, \mu) \in \mathcal{P}(\mathcal{X}) \times \mathcal{P}(\mathcal{Y}),$ we denote $\nu_{\sigma}^*(\mu) = \argmin_{\nu'} V^{\sigma}(\nu', \mu)$ and $\mu_{\sigma}^*(\nu) = \argmax_{\mu'} V^{\sigma}(\nu, \mu').$ Therefore, we have that
\begin{align*}
    &\max_{\mu' \in \mathcal{P}(\mathcal{Y})}V^{\sigma}(\nu,\mu') - V^{\sigma}(\nu,\mu) = V^{\sigma}\left(\nu, \mu_{\sigma}^*(\nu)\right) - V^{\sigma}(\nu, \mu)\\
    &\leq \int_{\mathcal{Y}} \frac{\delta F}{\delta \mu}(\nu, \mu, y) \left(\mu_{\sigma}^*(\nu) - \mu\right)(\mathrm{d}y) - \frac{\sigma^2}{2}\operatorname{D_{KL}}(\mu_{\sigma}^*(\nu)|\rho) + \frac{\sigma^2}{2}\operatorname{D_{KL}}(\mu|\rho)\\
    &\leq \int_{\mathcal{Y}} \frac{\delta F}{\delta \mu}(\nu, \mu, y)(\Phi_{\sigma}(\nu, \mu)-\mu)(\mathrm{d}y) - \frac{\sigma^2}{2}\operatorname{D_{KL}}(\Phi_{\sigma}(\nu, \mu)|\rho) + \frac{\sigma^2}{2}\operatorname{D_{KL}}(\mu|\rho)\\
    &= \int_{\mathcal{Y}} \Tilde{C} \left(\Phi_{\sigma}(\nu, \mu) - \mu\right)(\mathrm{d}y) + \frac{\sigma^2}{2} \int_{\mathcal{Y}} \log \frac{\Phi_{\sigma}(\nu, \mu)(y)}{\rho(y)} \left(\Phi_{\sigma}(\nu, \mu) - \mu\right)(\mathrm{d}y)\\ 
    &- \frac{\sigma^2}{2}\operatorname{D_{KL}}(\Phi_{\sigma}(\nu, \mu)|\rho) + \frac{\sigma^2}{2}\operatorname{D_{KL}}(\mu|\rho)\\
    &= -\frac{\sigma^2}{2} \int_{\mathcal{Y}} \log \frac{\Phi_{\sigma}(\nu, \mu)(y)}{\rho(y)} \mu(\mathrm{d}y) + \frac{\sigma^2}{2} \int_{\mathcal{Y}} \log \frac{\mu(y)}{\rho(y)} \mu(\mathrm{d}y) = \frac{\sigma^2}{2}\operatorname{D_{KL}}(\mu|\Phi_{\sigma}(\nu, \mu)),
\end{align*}
where the first inequality follows from \eqref{eq:concaveF} in Assumption \ref{assumption: assump-F-conv-conc}, the second inequality follows from \eqref{eq:phi-F} with $\mu' = \mu_{\sigma}^*(\nu),$ and the second equality follows from \eqref{eq: foc-phi}. Similarly, using \eqref{eq:convexF} from Assumption \ref{assumption: assump-F-conv-conc}, \eqref{eq:psi-F} with $\nu' = \nu_{\sigma}^*(\mu)$ and \eqref{eq: foc-psi}, we have that
    \begin{equation*}
        \min_{\nu' \in \mathcal{P}(\mathcal{X})}V^{\sigma}(\nu',\mu) - V^{\sigma}(\nu,\mu) \geq -\frac{\sigma^2}{2}\operatorname{D_{KL}}\left(\nu|\Psi_{\sigma}(\nu, \mu)\right).
    \end{equation*}
Therefore, we can finish the proof by adding the two inequalities above and recalling that $\operatorname{NI}(\nu, \mu) = \max_{\mu' \in \mathcal{P}(\mathcal{Y})}V^{\sigma}(\nu,\mu') - \min_{\nu' \in \mathcal{P}(\mathcal{X})}V^{\sigma}(\nu',\mu).$

\emph{Proof of the equality case.} Assume that $F$ is bilinear, i.e., $F(\nu, \mu) = \int_{\mathcal{Y}}\int_{\mathcal{X}} f(x,y) \nu(\mathrm{d}x) \mu(\mathrm{d}y),$ for some function $f:\mathcal{X} \times \mathcal{Y} \to \mathbb{R}.$ Then \eqref{eq:psi-linear-F} and \eqref{eq:phi-linear-F} become
\begin{equation}
\label{eq:psi-minimizer-bilinear}
    \Psi_{\sigma}(\mu) = \argmin_{\nu' \in \mathcal{P}(\mathcal{X})} V^{\sigma}(\nu', \mu),
\end{equation}
\begin{equation}
\label{eq:phi-maximizer-bilinear}
    \Phi_{\sigma}(\nu) = \argmax_{\mu' \in \mathcal{P}(\mathcal{Y})} V^{\sigma}(\nu, \mu').
\end{equation}
Therefore, using \eqref{eq:psi-minimizer-bilinear} and \eqref{eq:argminpsi-density}, we obtain
\begin{align*}
    V^{\sigma}(\nu, \mu) - \min_{\nu' \in \mathcal{P}(\mathcal{X})}V^{\sigma}(\nu',\mu) &= V^{\sigma}(\nu, \mu) - V^{\sigma}(\Psi_{\sigma}(\mu), \mu)\\
    &= \int_{\mathcal{Y}}\int_{\mathcal{X}} f(x,y) \nu(\mathrm{d}x) \mu(\mathrm{d}y) + \frac{\sigma^2}{2}\operatorname{D_{KL}}(\nu|\pi)\\ &- \int_{\mathcal{Y}}\int_{\mathcal{X}} f(x,y) \Psi_{\sigma}(\mu)(\mathrm{d}x) \mu(\mathrm{d}y) - \frac{\sigma^2}{2}\operatorname{D_{KL}}(\Psi_{\sigma}(\mu)|\pi)\\
    &= \int_{\mathcal{Y}}\int_{\mathcal{X}} f(x,y) \nu(\mathrm{d}x) \mu(\mathrm{d}y) + \frac{\sigma^2}{2}\operatorname{D_{KL}}(\nu|\pi) + \frac{\sigma^2}{2}\log Z(\mu),
\end{align*}
where 
\begin{equation*}
    Z(\mu) = \int_{\mathcal{X}} \exp\left(-\frac{2}{\sigma^2}\int_{\mathcal{Y}}f(x,y)\mu(\mathrm{d}y)-U^{\pi}(x)\right)\mathrm{d}x.
\end{equation*}
On the other hand, using \eqref{eq:argminpsi-density} again, a straightforward calculation shows that
\begin{equation*}
    \frac{\sigma^2}{2}\operatorname{D_{KL}}(\nu|\Psi_{\sigma}(\mu)) = \int_{\mathcal{Y}}\int_{\mathcal{X}} f(x,y) \nu(\mathrm{d}x) \mu(\mathrm{d}y) + \frac{\sigma^2}{2}\operatorname{D_{KL}}(\nu|\pi) + \frac{\sigma^2}{2}\log Z(\mu).
\end{equation*}
Hence, 
\begin{equation*}
    V^{\sigma}(\nu, \mu) - \min_{\nu' \in \mathcal{P}(\mathcal{X})}V^{\sigma}(\nu',\mu) = \frac{\sigma^2}{2}\operatorname{D_{KL}}(\nu|\Psi_{\sigma}(\mu)).
\end{equation*}
Similarly, using \eqref{eq:phi-maximizer-bilinear} and \eqref{eq:argmaxphi-density}, we can show that
\begin{equation*}
    \max_{\mu' \in \mathcal{P}(\mathcal{Y})}V^{\sigma}(\nu,\mu') - V^{\sigma}(\nu,\mu) = \frac{\sigma^2}{2}\operatorname{D_{KL}}(\mu|\Phi_{\sigma}(\nu)),
\end{equation*}
and hence 
\begin{equation*}
    \operatorname{NI}(\nu, \mu) = \frac{\sigma^2}{2}\left(\operatorname{D_{KL}}(\nu|\Psi_{\sigma}(\mu)) + \operatorname{D_{KL}}(\mu|\Phi_{\sigma}(\nu))\right).
\end{equation*}
\end{proof}

\subsection{Existence and uniqueness of the MF-BR flow}
In this subsection, we present the proof of our main result concerning the existence and uniqueness of the Mean-Field Best Response (MF-BR) flow, i.e., Proposition \ref{prop:Existence-flows-FP}. The proof follows a classical Picard iteration technique. Lemma \ref{thm:ContractivityBD} shows that a Picard iteration that we use for proving existence of the MF-BR flow admits a unique fixed point in an appropriate complete metric space, which then helps us to conclude the proof of Proposition \ref{prop:Existence-flows-FP}.

Before presenting the proof of Proposition \ref{prop:Existence-flows-FP}, we state and prove a useful auxiliary result. The lemma below is an adaptation of the second part of \cite[Proposition $4.2$]{https://doi.org/10.48550/arxiv.2202.05841} to the min-max setting \eqref{eq:F-nonlinear}. In contrast to \cite{https://doi.org/10.48550/arxiv.2202.05841}, we work with the total variation instead of the Wasserstein distance, which helps us to simplify some aspects of the argument, and to avoid imposing an additional assumption of Lipschitz continuity of the flat derivative of $F$ (cf.\ \cite[(2.1)]{https://doi.org/10.48550/arxiv.2202.05841}).
\begin{lemma}
\label{lemma:PsiPhiLipschitz}
Suppose that Assumption \ref{assumption: boundedness-first-flat} and \ref{assump:F} hold. Then there exist constants $L_{\Psi_{\sigma}}$, $L_{\Phi_{\sigma}} > 0$ such that for all $(\nu, \mu), (\nu', \mu') \in \mathcal{P}(\mathcal{X}) \times \mathcal{P}(\mathcal{Y})$, it holds that
\begin{equation}
\label{eq: Lipschitz-psi}
    \Big|\Psi_{\sigma}(\nu, \mu)(x) - \Psi_{\sigma}(\nu', \mu')(x)\Big| \leq L_{\Psi_{\sigma}}e^{-U^{\pi}(x)}\left(\operatorname{TV}(\nu,\nu') + \operatorname{TV}(\mu,\mu')\right),
\end{equation}
\begin{equation}
\label{eq: Lipschitz-phi}
    \Big|\Phi_{\sigma}(\nu, \mu)(y) - \Phi_{\sigma}(\nu', \mu')(y)\Big| \leq L_{\Phi_{\sigma}}e^{-U^{\rho}(y)}\left(\operatorname{TV}(\nu,\nu') + \operatorname{TV}(\mu,\mu')\right),
\end{equation}
and hence the maps \(\Psi_{\sigma} : \mathcal P \left(\mathcal{X}\right) \times \mathcal P \left(\mathcal{Y}\right) \to \mathcal P \left(\mathcal{X}\right)\) and \(\Phi_{\sigma} : \mathcal P \left(\mathcal{X}\right) \times \mathcal P \left(\mathcal{Y}\right) \to \mathcal P \left(\mathcal{Y}\right)\) are $\operatorname{TV}$-Lipschitz in the sense that there exist $L,L' > 0$ such that
\begin{equation*}
    \operatorname{TV}\left(\Psi_{\sigma}(\nu, \mu), \Psi_{\sigma}(\nu', \mu')\right) \leq L\left(\operatorname{TV}\left(\nu, \nu'\right) + \operatorname{TV}\left(\mu, \mu'\right)\right),
\end{equation*}
\begin{equation*}
    \operatorname{TV}\left(\Phi_{\sigma}(\nu, \mu), \Phi_{\sigma}(\nu', \mu')\right) \leq L'\left(\operatorname{TV}\left(\nu, \nu'\right) + \operatorname{TV}\left(\mu, \mu'\right)\right).
\end{equation*}
\end{lemma}
\begin{proof}
From Assumption \ref{assump:F}, using \eqref{eq: 2.5i}, \eqref{eq: 2.5ii} and the estimate \(\left|e^x - e^y\right| \leq e^{\max\{x,y\}} \left|x-y\right|\), we have
\begin{multline}
\label{eq:bound-diff-psi}
\left | \exp \left(- \frac{2}{\sigma^2} \frac{\delta F}{\delta \nu}(\nu, \mu, x) - U^{\pi} \left(x\right) \right) - \exp \left(- \frac{2}{\sigma^2} \frac{\delta F}{\delta \nu}(\nu', \mu', x) - U^{\pi} \left(x\right) \right) \right| \\
\leq \frac{2}{\sigma^2} \exp \left(\frac{2}{\sigma^2}C_{\nu}\right)e^{-U^{\pi}(x)}\left(\operatorname{TV}(\nu,\nu') + \operatorname{TV}(\mu,\mu')\right).
\end{multline}
Integrating the previous inequality with respect to \(x\), we obtain
\begin{equation}
\label{eq:bound-diff-psi-Z}
    \left|Z(\nu, \mu) - Z(\nu', \mu')\right| \leq \frac{2}{\sigma^2} \exp \left(\frac{2}{\sigma^2}C_{\nu}\right)\left(\operatorname{TV}(\nu,\nu') + \operatorname{TV}(\mu,\mu')\right).
\end{equation}
Therefore, we have that
\begin{align*}
    \Big|\Psi_{\sigma}(\nu, \mu)(x) - \Psi_{\sigma}(\nu', \mu')(x)\Big| &=  \Bigg|\frac{1}{Z(\nu, \mu)}\exp \left(- \frac{2}{\sigma^2} \frac{\delta F}{\delta \nu}(\nu, \mu, x) - U^{\pi} \left(x\right) \right)\\ 
    &- \frac{1}{Z(\nu', \mu')}\exp \left(- \frac{2}{\sigma^2} \frac{\delta F}{\delta \nu}(\nu, \mu, x) - U^{\pi} \left(x\right) \right)\\ &+ \frac{1}{Z(\nu', \mu')}\exp \left(- \frac{2}{\sigma^2} \frac{\delta F}{\delta \nu}(\nu, \mu, x) - U^{\pi} \left(x\right) \right)\\ &- \frac{1}{Z(\nu', \mu')}\exp \left(- \frac{2}{\sigma^2} \frac{\delta F}{\delta \nu}(\nu', \mu', x) - U^{\pi} \left(x\right) \right)\Bigg|\\
    &\leq \exp \left(- \frac{2}{\sigma^2} \frac{\delta F}{\delta \nu}(\nu, \mu, x) - U^{\pi} \left(x\right) \right) \frac{\Big|Z(\nu', \mu') - Z(\nu, \mu)\Big|}{Z(\nu, \mu)Z(\nu', \mu')}\\ 
    &+ \frac{1}{Z(\nu', \mu')}\Bigg|\exp \left(- \frac{2}{\sigma^2} \frac{\delta F}{\delta \nu}(\nu, \mu, x) - U^{\pi} \left(x\right) \right)\\ &- \exp \left(- \frac{2}{\sigma^2} \frac{\delta F}{\delta \nu}(\nu', \mu', x) - U^{\pi} \left(x\right) \right)\Bigg|.
\end{align*}
Using estimates \eqref{eq:bound-psi-unnormalized-density}, \eqref{eq:bound-psi-normalization-constant}, \eqref{eq:bound-diff-psi} and \eqref{eq:bound-diff-psi-Z}, we arrive at the Lipschitz property \eqref{eq: Lipschitz-psi} with $L_{\Psi_{\sigma}} \coloneqq \frac{2}{\sigma^2}\exp\left(\frac{4C_{\nu}}{\sigma^2}\right)\left(1+\exp\left(\frac{4C_{\nu}}{\sigma^2}\right)\right) >0$. Proving \eqref{eq: Lipschitz-phi} follows the same steps as above but with $L_{\Phi_{\sigma}} \coloneqq \frac{2}{\sigma^2}\exp\left(\frac{4C_{\mu}}{\sigma^2}\right)\left(1+\exp\left(\frac{4C_{\mu}}{\sigma^2}\right)\right)$.

Now, integrating \eqref{eq: Lipschitz-psi} on $\mathcal{X},$ and applying \cite[Lemma $2.1$]{tsybakov2008introduction}, that is $\operatorname{TV}(m,m') = \frac{1}{2}\int_{\mathcal{X}} |m(x)-m'(x)|\mathrm{d}x,$ for any $m,m' \in \mathcal{P}_{\text{ac}}(\mathcal{X}),$ it follows that
\begin{equation*}
    \operatorname{TV}\left(\Psi_{\sigma}(\nu, \mu), \Psi_{\sigma}(\nu', \mu')\right) \leq \frac{L_{\Psi_{\sigma}}}{2}\left(\operatorname{TV}\left(\nu, \nu'\right) + \operatorname{TV}\left(\mu, \mu'\right)\right),
\end{equation*}
and we set $L \coloneqq \frac{L_{\Psi_{\sigma}}}{2} > 0$. One similarly obtains that $\Phi_{\sigma}$ is TV-Lipschitz with constant $L' \coloneqq \frac{L_{\Phi_{\sigma}}}{2}$. 
\end{proof}

\begin{proof}[Proof of Proposition \ref{prop:Existence-flows-FP}]
\emph{Step 1: Existence of gradient flow on $[0,T].$}
    By Duhamel's formula we can rewrite equations in \eqref{eq: mean-field-br} as
\begin{equation*}
\nu_t(x) = e^{-\alpha t}\nu_0(x) + \int_0^t \alpha e^{-\alpha (t-s)} \Psi_{\sigma}(\nu_s, \mu_s)(x) \mathrm{d}s,
\end{equation*}
\begin{equation*}
\mu_t(y) = e^{-\alpha t}\mu_0(y) + \int_0^t \alpha e^{-\alpha (t-s)} \Phi_{\sigma}(\nu_s, \mu_s)(y) \mathrm{d}s.
\end{equation*}
Based on these expressions, we will define a Picard iteration scheme as follows. Fix $T > 0$ and for each $n \geq 1$, fix $\nu_0^{(n)} = \nu_0^{(0)} = \nu_0$ and $\mu_0^{(n)} = \mu_0^{(0)} = \mu_0.$ Then define $(\nu_t^{(n)})_{t \in [0,T]}$ and $(\mu_t^{(n)})_{t \in [0,T]}$ by
\begin{equation}
\label{eq:Picard-nu}
\nu^{(n)}_t(x) = e^{-\alpha t}\nu_0(x) + \int_0^t \alpha e^{-\alpha (t-s)} \Psi_{\sigma}(\nu^{(n-1)}_s, \mu^{(n-1)}_s)(x) \mathrm{d}s,
\end{equation}
\begin{equation}
\label{eq:Picard-mu}
\mu^{(n)}_t(y) = e^{-\alpha t}\mu_0(y) + \int_0^t \alpha e^{-\alpha (t-s)} \Phi_{\sigma}(\nu^{(n-1)}_s, \mu^{(n-1)}_s)(y) \mathrm{d}s.
\end{equation}
For fixed $T > 0$, we consider the sequence of flows $\left( (\nu_t^{(n)}, \mu_t^{(n)})_{t \in [0,T]} \right)_{n=0}^{\infty}$ in \\$\left(\mathcal{P}_{\text{ac}}(\mathcal{X})^{[0,T]} \times \mathcal{P}_{\text{ac}}(\mathcal{Y})^{[0,T]}, \mathcal{TV}^{[0,T]}\right),$ where, for any $(\nu_t, \mu_t)_{t \in [0,T]} \in \mathcal{P}_{\text{ac}}(\mathcal{X})^{[0,T]} \times \mathcal{P}_{\text{ac}}(\mathcal{Y})^{[0,T]}$, the distance $\mathcal{TV}^{[0,T]}$ is defined by
\begin{equation*}
\mathcal{TV}^{[0,T]} \left( (\nu_t, \mu_t)_{t \in [0,T]}, (\nu_t',\mu_t')_{t \in [0,T]} \right) := \int_0^T \operatorname{TV}(\nu_t, \nu_t') \mathrm{d}t + \int_0^T \operatorname{TV}(\mu_t, \mu_t') \mathrm{d}t.
\end{equation*}
Since $\left(\mathcal{P}(\mathcal{X}), \operatorname{TV}\right)$ is complete, we can apply the argument from \cite[Lemma A.5]{SiskaSzpruch2020} with $p=1$ to conclude that $\left(\mathcal{P}(\mathcal{X})^{[0,T]}, \int_0^T \operatorname{TV}(\nu_t, \nu_t') \mathrm{d}t\right)$ and $\left(\mathcal{P}(\mathcal{Y})^{[0,T]}, \int_0^T \operatorname{TV}(\mu_t, \mu_t') \mathrm{d}t\right)$ are complete. Therefore, one can deduce that $\left(\mathcal{P}(\mathcal{X})^{[0,T]} \times \mathcal{P}(\mathcal{Y})^{[0,T]}, \mathcal{TV}^{[0,T]}\right)$ is also complete. 

On the other hand, it is straightforward to check that $\left(\mathcal{P}_{\text{ac}}(\mathcal{X}), \operatorname{TV}\right)$ is closed. Indeed, take a sequence $(\mu_n)_{n \geq 1} \subset \mathcal{P}_{\text{ac}}(\mathcal{X})$ such that $\mu_n \to \mu$ in $\operatorname{TV}$ for some $\mu \in \mathcal{P}(\mathcal{X})$. 
By Definition \ref{def:KRwasserstein}, since $\mu_n \to \mu$ in $\operatorname{TV}$, it follows that $\mu_n(A) \to \mu(A)$ for all sets $A \in \mathcal{B}(\mathcal{X})$, where $\mathcal{B}(\mathcal{X})$ is the Borel $\sigma$-algebra on $\mathcal{X}$. Since $(\mu_n)_{n \geq 1} \subset \mathcal{P}_{\text{ac}}(\mathcal{X})$, choosing $A$ with Lebesgue measure $0$ implies that $\mu_n(A) = 0$ for all $n \geq 1$. Hence, $\mu(A) = 0$, i.e. $\mu \in \mathcal{P}_{\text{ac}}(\mathcal{X})$. Therefore, $\left(\mathcal{P}_{\text{ac}}(\mathcal{X}), \operatorname{TV}\right)$ is closed. Then clearly both $\left(\mathcal{P}_{\text{ac}}(\mathcal{X})^{[0,T]}, \int_0^T \operatorname{TV}(\nu_t, \nu_t') \mathrm{d}t\right)$ and $\left(\mathcal{P}_{\text{ac}}(\mathcal{Y})^{[0,T]}, \int_0^T \operatorname{TV}(\mu_t, \mu_t') \mathrm{d}t\right)$ are closed and therefore $\left(\mathcal{P}_{\text{ac}}(\mathcal{X})^{[0,T]} \times \mathcal{P}_{\text{ac}}(\mathcal{Y})^{[0,T]}, \mathcal{TV}^{[0,T]}\right)$ is closed. But then since $\mathcal{P}_{\text{ac}}(\mathcal{X})^{[0,T]} \times \mathcal{P}_{\text{ac}}(\mathcal{Y})^{[0,T]} \subset \mathcal{P}(\mathcal{X})^{[0,T]} \times \mathcal{P}(\mathcal{Y})^{[0,T]}$ and the latter is complete in $\operatorname{TV}$-norm, we obtain that \newline $\left(\mathcal{P}_{\text{ac}}(\mathcal{X})^{[0,T]} \times \mathcal{P}_{\text{ac}}(\mathcal{Y})^{[0,T]}, \mathcal{TV}^{[0,T]}\right)$ is complete.

We consider the Picard iteration mapping $\varphi \left((\nu_t^{(n-1)}, \mu_t^{(n-1)})_{t \in [0,T]}\right) := (\nu_t^{(n)}, \mu_t^{(n)})_{t \in [0,T]}$ defined via \eqref{eq:Picard-nu} and \eqref{eq:Picard-mu} and show that $\varphi$ admits a unique fixed point $(\nu_t, \mu_t)_{t \in [0,T]}$ in the complete space $\left(\mathcal{P}_{\text{ac}}(\mathcal{X})^{[0,T]} \times \mathcal{P}_{\text{ac}}(\mathcal{Y})^{[0,T]}, \mathcal{TV}^{[0,T]}\right)$. Then this fixed point is the solution to \eqref{eq: mean-field-br}.

\begin{lemma}
\label{thm:ContractivityBD}
	The mapping $\varphi\left((\nu_t^{(n-1)}, \mu_t^{(n-1)})_{t \in [0,T]}\right) := (\nu_t^{(n)}, \mu_t^{(n)})_{t \in [0,T]}$ defined via \eqref{eq:Picard-nu} and \eqref{eq:Picard-mu} admits a unique fixed point in $\left(\mathcal{P}_{\text{ac}}(\mathcal{X})^{[0,T]} \times \mathcal{P}_{\text{ac}}(\mathcal{Y})^{[0,T]}, \mathcal{TV}^{[0,T]}\right)$.
\end{lemma} 
	\begin{proof}[Proof of Lemma \ref{thm:ContractivityBD}]
 \emph{Step 1: The sequence of flows $\left( (\nu_t^{(n)}, \mu_t^{(n)})_{t \in [0,T]} \right)_{n=0}^{\infty}$ is a Cauchy sequence in $\left(\mathcal{P}_{\text{ac}}(\mathcal{X})^{[0,T]} \times \mathcal{P}_{\text{ac}}(\mathcal{Y})^{[0,T]}, \mathcal{TV}^{[0,T]}\right).$}
 
From \cite[Lemma $2.1$]{tsybakov2008introduction}, that is $\operatorname{TV}(m,m') = \frac{1}{2}\int_{\mathcal{X}} |m(x)-m'(x)|\mathrm{d}x,$ for any $m,m' \in \mathcal{P}_{\text{ac}}(\mathcal{X}),$ and \eqref{eq:Picard-nu}, we have that
\begin{multline}
\label{eq: TV-nu}
\begin{aligned}
    &\operatorname{TV}\left(\nu^{(n)}_t, \nu^{(n-1)}_t\right) = \frac{1}{2}\int_{\mathcal{X}}\Big|\nu^{(n)}_t(x)-\nu^{(n-1)}_t(x)\Big|\mathrm{d}x\\ 
    &= \frac{1}{2}\int_{\mathcal{X}}\Big|\int_0^t \alpha e^{-\alpha (t-s)} \Psi_{\sigma}(\nu^{(n-1)}_s, \mu^{(n-1)}_s)(x) \mathrm{d}s - \int_0^t \alpha e^{-\alpha (t-s)}\Psi_{\sigma}(\nu^{(n-2)}_s, \mu^{(n-2)}_s)(x) \mathrm{d}s\Big|\mathrm{d}x\\ 
    &\leq \frac{\alpha}{2}\int_{\mathcal{X}}\int_0^t \Big|\Psi_{\sigma}(\nu^{(n-1)}_s, \mu^{(n-1)}_s)(x) - \Psi_{\sigma}(\nu^{(n-2)}_s, \mu^{(n-2)}_s)(x)\Big| \mathrm{d}s\mathrm{d}x\\ 
    &\leq \frac{\alpha}{2}\int_{\mathcal{X}}\int_0^t L_{\Psi_{\sigma}}e^{-U^{\pi}(x)}\left(\operatorname{TV}(\nu^{(n-1)}_s,\nu^{(n-2)}_s) + \operatorname{TV}(\mu^{(n-1)}_s,\mu^{(n-2)}_s)\right)\mathrm{d}s\mathrm{d}x\\
    &= \frac{\alpha L_{\Psi_{\sigma}}}{2}\int_0^t \left(\operatorname{TV}(\nu^{(n-1)}_s,\nu^{(n-2)}_s) + \operatorname{TV}(\mu^{(n-1)}_s,\mu^{(n-2)}_s)\right)\mathrm{d}s,
\end{aligned}
\end{multline}
where in the first inequality we used the fact that $e^{-\alpha (t-s)} \leq 1$, for all $s \in [0,t]$, and in the second inequality we used \eqref{eq: Lipschitz-psi}.

A similar argument using \eqref{eq: Lipschitz-phi} leads to
\begin{equation}
\label{eq: TV-mu}
    \operatorname{TV}\left(\mu^{(n)}_t, \mu^{(n-1)}_t\right) \leq \frac{\alpha L_{\Phi_{\sigma}}}{2}\int_0^t \left(\operatorname{TV}(\nu^{(n-1)}_s,\nu^{(n-2)}_s) + \operatorname{TV}(\mu^{(n-1)}_s,\mu^{(n-2)}_s)\right)\mathrm{d}s.
\end{equation}
Adding inequalities \eqref{eq: TV-nu} and \eqref{eq: TV-mu} and setting $\alpha_L \coloneqq \frac{\alpha L_{\Psi_{\sigma}}+\alpha L_{\Phi_{\sigma}}}{2} > 0$ gives
\begin{align*}
    &\operatorname{TV}\left(\nu^{(n)}_t, \nu^{(n-1)}_t\right) + \operatorname{TV}\left(\mu^{(n)}_t, \mu^{(n-1)}_t\right) \leq \alpha_L \int_0^t \left(\operatorname{TV}\left(\nu^{(n-1)}_s, \nu^{(n-2)}_s\right) + \operatorname{TV}\left(\mu^{(n-1)}_s, \mu^{(n-2)}_s\right)\right) \mathrm{d}s\\
   &\leq \left(\alpha_L\right)^{n-1}  \int_0^t \int_0^{t_1} \ldots \int_0^{t_{n-2}} \left(\operatorname{TV}\left(\nu^{(1)}_{t_{n-1}}, \nu^{(0)}_{t_{n-1}}\right) + \operatorname{TV}\left(\mu^{(1)}_{t_{n-1}}, \mu^{(0)}_{t_{n-1}}\right)\right) \mathrm{d}t_{n-1} \ldots \mathrm{d}t_2 \mathrm{d}t_1\\ 
    &\leq \left(\alpha_L\right)^{n-1}\frac{t^{n-2}}{(n-2)!}\int_0^t \left(\operatorname{TV}\left(\nu^{(1)}_{t_{n-1}}, \nu^{(0)}_{t_{n-1}}\right) + \operatorname{TV}\left(\mu^{(1)}_{t_{n-1}}, \mu^{(0)}_{t_{n-1}}\right)\right) \mathrm{d}t_{n-1},
  \end{align*}
	where in the third inequality we used the bound $\int_0^{t_{n-2}} \mathrm{d}t_{n-1} \leq \int_0^{t}  \mathrm{d}t_{n-1}$.
    Hence, we obtain
		\begin{align*}
		&\int_0^T \operatorname{TV}\left(\nu^{(n)}_t, \nu^{(n-1)}_t\right) + \operatorname{TV}\left(\mu^{(n)}_t, \mu^{(n-1)}_t\right)\mathrm{d}t \\ 
  &\leq \left(\alpha_L\right)^{n-1}\frac{T^{n-1}}{(n-2)!}\int_0^T \left(\operatorname{TV}\left(\nu^{(1)}_{t_{n-1}}, \nu^{(0)}_{t_{n-1}}\right) + \operatorname{TV}\left(\mu^{(1)}_{t_{n-1}}, \mu^{(0)}_{t_{n-1}}\right)\right)\mathrm{d}t_{n-1}.
	\end{align*}
 Using the definition of $\mathcal{TV}^{[0,T]},$ the last inequality becomes
 \begin{align*}
    &\mathcal{TV}^{[0,T]} \left( (\nu_t^{(n)}, \mu_t^{(n)})_{t \in [0,T]}, (\nu_t^{(n-1)}, \mu_t^{(n-1)})_{t \in [0,T]} \right)\\ &\leq \left(\alpha_L\right)^{n-1} \frac{T^{n-1}}{(n-2)!} \mathcal{TV}^{[0,T]} \left( (\nu_t^{(1)}, \mu_t^{(1)})_{t \in [0,T]}, (\nu_t^{(0)}, \mu_t^{(0)})_{t \in [0,T]} \right).
    \end{align*}
	By choosing $n$ sufficiently large, we conclude that $\left( (\nu_t^{(n)}, \mu_t^{(n)})_{t \in [0,T]} \right)_{n=0}^{\infty}$ is a Cauchy sequence. By completeness of $\left(\mathcal{P}_{\text{ac}}(\mathcal{X})^{[0,T]} \times \mathcal{P}_{\text{ac}}(\mathcal{Y})^{[0,T]}, \mathcal{TV}^{[0,T]}\right),$ the sequence admits a limit point $(\nu_t, \mu_t)_{t \in [0,T]} \in \left(\mathcal{P}_{\text{ac}}(\mathcal{X})^{[0,T]} \times \mathcal{P}_{\text{ac}}(\mathcal{Y})^{[0,T]}, \mathcal{TV}^{[0,T]}\right)$.

\emph{Step 2: The limit point $(\nu_t, \mu_t)_{t \in [0,T]}$ is a fixed point of $\varphi.$}

From Step 1, we obtain that for Lebesgue-almost all $t \in [0,T]$ we have
	\begin{equation*}
	\operatorname{TV}(\nu_t^{(n)},\nu_t) \to 0, \quad \operatorname{TV}(\mu_t^{(n)},\mu_t) \to 0, \quad \text{ as } n \to \infty.
	\end{equation*}
Therefore, by \eqref{eq: Lipschitz-psi} and \eqref{eq: Lipschitz-phi}, for Lebesgue-almost all $t \in [0,T]$ and any fixed $(x,y) \in \mathcal{X} \times \mathcal{Y},$ we have that $\Psi_{\sigma}(\nu_t^{(n-1)}, \mu_t^{(n-1)})(x)$ and $\Phi_{\sigma}(\nu_t^{(n-1)}, \mu_t^{(n-1)})(y)$ from \eqref{eq:Picard-nu} and \eqref{eq:Picard-mu} converge to $\Psi_{\sigma}(\nu_t, \mu_t)(x)$ and $\Phi_{\sigma}(\nu_t, \mu_t)(y),$ respectively, as $n \to \infty.$ Therefore, letting $n \to \infty$ in \eqref{eq:Picard-nu} and \eqref{eq:Picard-mu} and using the dominated convergence theorem (which is possible since $\Psi_{\sigma}, \Phi_{\sigma}$ are uniformly bounded due to Assumption \ref{assumption: boundedness-first-flat}), we conclude that $(\nu_t, \mu_t)_{t \in [0,T]}$ is a fixed point of $\varphi.$

\emph{Step 3: The fixed point $(\nu_t, \mu_t)_{t \in [0,T]}$ of $\varphi$ is unique.}

Suppose, for the contrary, that $\varphi$ admits two fixed points $(\nu_t, \mu_t)_{t \in [0,T]}$ and $(\bar{\nu}_t, \bar{\mu}_t)_{t \in [0,T]}$ such that $\nu_0 = \bar{\nu}_0$ and $\mu_0 = \bar{\mu}_0.$ Then repeating the same calculations from \eqref{eq: TV-nu} and \eqref{eq: TV-mu}, we arrive at
\begin{equation*}
        \operatorname{TV}(\nu_t, \bar{\nu}_t) + \operatorname{TV}(\mu_t,\bar{\mu}_t) \leq \alpha_L \int_0^t \left(\operatorname{TV}(\mu_s,\bar{\mu}_s) + \operatorname{TV}(\nu_s,\bar{\nu}_s)\right)\mathrm{d}s.
    \end{equation*}
    For each $t \in [0,T],$ denote $f(t) \coloneqq \int_0^t \left(\operatorname{TV}(\mu_s,\bar{\mu}_s) + \operatorname{TV}(\nu_s,\bar{\nu}_s)\right) \mathrm{d}s.$ Observe that $f \geq 0$ and $f(0) = 0.$ Then, by Gronwall's lemma, we obtain
    \begin{equation*}
        0 \leq f(t) \leq e^{\alpha_L t}f(0) = 0,
    \end{equation*}
    and hence 
    \begin{equation*}
        \operatorname{TV}(\nu_t, \bar{\nu}_t) + \operatorname{TV}(\mu_t,\bar{\mu}_t) = 0,
    \end{equation*}
    for Lebesgue-almost all $t \in [0,T],$ which implies
    \begin{equation*}
        \nu_t = \bar{\nu}_t, \quad \mu_t = \bar{\mu}_t,
    \end{equation*}
    for Lebesgue-almost all $t \in [0,T].$ Therefore, the fixed point $(\nu_t, \mu_t)_{t \in [0,T]}$ of $\phi$ must be unique.

    From Steps 1, 2 and 3, we obtain the existence and uniqueness of a pair of flows $(\nu_t, \mu_t)_{t \in [0,T]}$ satisfying \eqref{eq: mean-field-br} for any $T > 0.$
\end{proof}
Having proved Lemma \ref{thm:ContractivityBD}, we return to the proof of Proposition \ref{prop:Existence-flows-FP}.

\emph{Step 2: Existence of gradient flow on $[0,\infty).$}
From Lemma \ref{thm:ContractivityBD}, for any $T > 0,$ there exists unique flow $(\nu_t, \mu_t)_{t \in [0,T]}$ satisfying \eqref{eq: mean-field-br}. It remains to prove that the existence of this flow could be extended to $[0,\infty).$ Let $(\nu_t, \mu_t)_{t \in [0,T]}, (\nu'_t, \mu'_t)_{t \in [0,T]} \in \mathcal{P}_{\text{ac}}(\mathcal{X}) \times \mathcal{P}_{\text{ac}}(\mathcal{Y}).$ Then, using the calculations from Lemma \ref{thm:ContractivityBD}, we have that
\begin{equation*}
    \operatorname{TV}\left(\nu_t, \nu'_t\right) + \operatorname{TV}\left(\mu_t, \mu'_t\right) \leq \alpha_L \int_0^t \left(\operatorname{TV}\left(\nu_s, \nu'_s\right) + \operatorname{TV}\left(\mu_s, \mu'_s\right)\right) \mathrm{d}s, 
\end{equation*}
which shows that $(\nu_t, \mu_t)_{t \in [0,T]}$ do not blow up in any finite time, and therefore we can extend $(\nu_t, \mu_t)_{t \in [0,T]}$ globally to $(\nu_t, \mu_t)_{t \in [0,\infty)}.$ By \eqref{eq:argminpsi-density}, $\Psi_{\sigma}(\nu, \mu)$ admits a density of the form
\begin{multline*}
    \Psi_{\sigma}(\nu_t, \mu_t)(x) \coloneqq \frac{1}{Z(\nu_t, \mu_t)}\exp\left({-\frac{2}{\sigma^2}\frac{\delta F}{\delta \nu}(\nu_t, \mu_t, x) - U^{\pi}(x)}\right), \text{ with }\\ Z(\nu_t, \mu_t) = \int \exp\left({-\frac{2}{\sigma^2}\frac{\delta F}{\delta \nu}(\nu_t, \mu_t, x) - U^{\pi}(x)}\right) \mathrm{d}x.
\end{multline*}
For any fixed $x \in \mathcal{X},$ the flat derivative $t \mapsto \frac{\delta F}{\delta \nu}(\nu_t, \mu_t, x)$ is continuous on $[0, \infty)$ due to the fact that $\nu_t \in C\left([0, \infty), \mathcal{P}_{\text{ac}}(\mathcal{X})\right),$ $\mu_t \in C\left([0, \infty), \mathcal{P}_{\text{ac}}(\mathcal{Y})\right),$ and $(\nu, \mu) \mapsto \frac{\delta F}{\delta \nu}(\nu, \mu, x)$ is continuous. Moreover, the flat derivative $\frac{\delta F}{\delta \nu}(\nu_t, \mu_t, x)$ is bounded for every \(x \in \mathcal{X}\) and all $t \geq 0$ due to Assumption \ref{assumption: boundedness-first-flat}. Therefore, both terms $\frac{1}{Z(\nu_t, \mu_t)}$ and $\exp\left({-\frac{2}{\sigma^2}\frac{\delta F}{\delta \nu}(\nu_t, \mu_t, x) - U^{\pi}(x)}\right)$ are continuous in $t$ and bounded for every \(x \in \mathcal{X}\). Hence, we have that $\Psi_{\sigma}(\nu_t, \mu_t)(x)$ is continuous in $t$ and bounded for every \(x\ \in \mathcal{X}\). The same argument gives that $\Phi_{\sigma}(\nu_t, \mu_t)(y)$ is continuous in $t$ and bounded for every \(y\ \in \mathcal{X}\). But then this implies that the integrands in \eqref{equation:duhamel-nu} and \eqref{equation:duhamel-mu} are continuous in $s$ for all $t \geq 0$ and bounded for all \((x,y) \in \mathcal{X} \times \mathcal{Y}\). Hence, $\nu_t \in C^1\left([0, \infty), \mathcal{P}_{\text{ac}}(\mathcal{X})\right)$ and $\mu_t \in C^1\left([0, \infty), \mathcal{P}_{\text{ac}}(\mathcal{Y})\right).$
\end{proof}

\subsection{Additional results}
\label{subsec:additional}
We finally present three results: the first two concerning the existence and uniqueness of MNEs for games of the form \eqref{eq:F-nonlinear} and the last one illustrating how the regularized game is an approximation of the non-regularized game.
\begin{theorem}[\cite{conforti_game_2020}, Theorem $3.2$]
\label{thm: 2.6}
Let $p \geq 1.$ Suppose that $F$ admits first-order flat derivative (cf. Definition \ref{def:fderivative}) and that Assumption \ref{assumption: boundedness-first-flat} and the following hold:
\begin{enumerate}
    \item \label{eq: non-empty-convex} For any $(\nu, \mu) \in \mathcal{P}_p(\mathcal{X}) \times \mathcal{P}_p(\mathcal{Y}),$ the sets $\argmin_{\nu' \in \mathcal{P}_p(\mathcal{X})} V^{\sigma}(\nu',\mu)$ and \newline $\argmin_{\mu' \in \mathcal{P}_p(\mathcal{Y})} \{-V^{\sigma}(\nu,\mu')\}$ are non-empty and convex,
    \item \label{eq:continuity-F} The map $(\nu,\mu) \mapsto F(\nu, \mu)$ is jointly $\mathcal{W}_p$-continuous,
    \item \label{eq:U-growth} There exist $C_{\pi} \geq C'_{\pi} > 0$ and $C_{\rho} \geq C'_{\rho} > 0$ such that $U^{\pi}(x) \geq C'_{\pi}|x|^p - C_{\pi}$ and $U^{\rho}(y) \geq C'_{\rho}|y|^p-C_{\rho},$ for all $(x,y) \in \mathcal{X} \times \mathcal{Y}.$
\end{enumerate}
Then there exists at least one MNE $(\nu_{\sigma}^*, \mu_{\sigma}^*)$ of the game \eqref{eq:F-nonlinear}.
\end{theorem}
\begin{proof}
    The proof closely follows the one of \cite[Theorem $3.2$]{conforti_game_2020}. 
    For any $(\nu, \mu) \in \mathcal{P}_p(\mathcal{X}) \times \mathcal{P}_p(\mathcal{Y}),$ define $$\mathcal{R}_1(\mu) := \argmin_{\nu' \in \mathcal{P}_p(\mathcal{X})} V^{\sigma}(\nu',\mu), \quad \mathcal{R}_2(\nu) := \argmin_{\mu' \in \mathcal{P}_p(\mathcal{Y})} \{-V^{\sigma}(\nu,\mu')\}.$$ Note that, due to Condition \eqref{eq: non-empty-convex}, the sets $\mathcal{R}_1(\mu)$ and $\mathcal{R}_2(\nu)$ are non-empty and convex, for any $(\nu, \mu) \in \mathcal{P}_p(\mathcal{X}) \times \mathcal{P}_p(\mathcal{Y}),$ and so is $$\mathcal{R}(\nu, \mu) := \left\{(\bar{\nu}, \bar{\mu}) \in \mathcal{P}_p(\mathcal{X}) \times \mathcal{P}_p(\mathcal{Y}): \bar{\nu} \in \mathcal{R}_1(\mu), \bar{\mu} \in \mathcal{R}_2(\nu)\right\}.$$
    
    Due to Proposition \ref{prop: 2.5}, any $\nu \in \mathcal{R}_1(\mu)$ satisfies the first-order condition
    \begin{equation*}
        \nu(x) = \frac{1}{Z(\nu, \mu)}e^{-\frac{2}{\sigma^2}\frac{\delta F}{\delta \nu}(\nu, \mu, x) - U^{\pi}(x)},
    \end{equation*}
    where $Z(\nu, \mu) > 0$ is a normalization constant so that $\nu$ is a probability measure. By Assumption \ref{assumption: boundedness-first-flat} and Condition \eqref{eq:U-growth}, we have
    \begin{equation*}
        \nu(x) = \frac{1}{Z(\nu, \mu)}e^{-\frac{2}{\sigma^2}\frac{\delta F}{\delta \nu}(\nu, \mu, x) - U^{\pi}(x)} \leq \frac{1}{Z(\nu, \mu)}e^{\frac{2}{\sigma^2}C_{\nu} - C_{\pi}'|x|^p + C_{\pi}}.
    \end{equation*}
    Integrating the inequality above gives
    \begin{equation*}
        Z(\nu, \mu) \leq e^{C_{\pi} + \frac{2}{\sigma^2}C_{\nu}}\int_{\mathcal{X}}e^{- C_{\pi}'|x|^p}\mathrm{d}x < \infty,
    \end{equation*}
    since $|x|^p > 0,$ and hence $Z(\nu, \mu)$ is uniformly bounded. 
    
    Let $p' > p \geq 1.$ Then, by Assumption \ref{assumption: boundedness-first-flat} and Condition \eqref{eq:U-growth}, it follows that
    \begin{align*}
        \int_{\mathcal{X}} |x|^{p'} \nu(\mathrm{d}x) &= \frac{1}{Z(\nu,\mu)}\int_{\mathcal{X}} |x|^{p'}e^{-\frac{2}{\sigma^2}\frac{\delta F}{\delta \nu}(\nu, \mu, x) - U^{\pi}(x)} \mathrm{d}x \\
        &\leq \frac{1}{Z(\nu,\mu)}e^{C_{\pi} + \frac{2}{\sigma^2}C_{\nu}}\int_{\mathcal{X}}|x|^{p'}e^{- C_{\pi}'|x|^p}\mathrm{d}x < \infty.
    \end{align*}
    Therefore, we obtain 
    \begin{equation*}
        \overline{C}^{\nu} := \sup_{\nu \in \mathcal{R}_1(\mu)} \int_{\mathcal{X}} |x|^{p'} \nu(\mathrm{d}x) < \infty.
    \end{equation*}
    A similar argument gives
    \begin{equation*}
        \overline{C}^{\mu} := \sup_{\mu \in \mathcal{R}_2(\nu)} \int_{\mathcal{Y}} |y|^{p'} \mu(\mathrm{d}y) < \infty.
    \end{equation*}
    Define $$\mathcal{S}^{\nu} := \left\{\nu' \in \mathcal{P}_p(\mathcal{X}): \int_{\mathcal{X}} |x|^{p'} \nu'(\mathrm{d}x) \leq \overline{C}^{\nu}\right\}, \quad \mathcal{S}^{\mu} := \left\{\mu' \in \mathcal{P}_p(\mathcal{Y}): \int_{\mathcal{Y}} |y|^{p'} \mu'(\mathrm{d}y) \leq \overline{C}^{\mu}\right\},$$ and $$\mathcal{S} := \left\{(\Tilde{\nu}, \Tilde{\mu}) \in \mathcal{P}_p(\mathcal{X}) \times \mathcal{P}_p(\mathcal{Y}): \Tilde{\nu} \in \mathcal{S}^{\nu}, \Tilde{\mu} \in \mathcal{S}^{\mu}\right\}.$$ Above we showed that $\mathcal{R}(\nu, \mu) \subset \mathcal{S},$ for any $(\nu, \mu) \in \mathcal{P}_p(\mathcal{X}) \times \mathcal{P}_p(\mathcal{Y}),$ and so $\mathcal{S}$ is non-empty. Recall that $\mathcal{S}^{\nu}, \mathcal{S}^{\mu}$ are $\mathcal{W}_p$-compact (see, e.g., Subsection 1.3 in \cite{10.1214/20-AIHP1140}). A straightforward calculation shows that $\mathcal{S}^{\nu}, \mathcal{S}^{\mu}$ are also convex. Therefore, $\mathcal{S}$ is non-empty, $\mathcal{W}_p$-compact and convex. 
    
    Next, we show that the graph of the mapping $\mathcal{S} \ni (\nu,\mu) \mapsto \mathcal{R}(\nu,\mu) \subset \mathcal{S}$ is $\mathcal{W}_p$-closed, i.e. given $(\nu_{\infty}, \mu_{\infty}), (\nu'_{\infty}, \mu'_{\infty}) \in \mathcal{S}$, for any $(\nu_n, \mu_n) \to (\nu_{\infty}, \mu_{\infty})$ and $(\nu'_n, \mu'_n) \to (\nu'_{\infty}, \mu'_{\infty})$ in $\mathcal{W}_p$ and $(\nu_n', \mu_n') \in \mathcal{R}(\nu_n, \mu_n)$, it follows that $(\nu'_{\infty}, \mu'_{\infty}) \in \mathcal{R}(\nu_{\infty}, \mu_{\infty})$.

    By Condition \eqref{eq:continuity-F} and the lower semi-continuity of $\nu \mapsto \operatorname{D_{KL}}(\nu|\pi)$, we have 
    \begin{multline}
    \label{eq:lsc-nu}
    \begin{aligned}
        V^{\sigma}(\nu'_{\infty}, \mu_{\infty}) + \frac{\sigma^2}{2}\operatorname{D_{KL}}(\mu_{\infty}|\rho) &= F(\nu'_{\infty}, \mu_{\infty}) + \frac{\sigma^2}{2}\operatorname{D_{KL}}(\nu'_{\infty}|\pi)\\
        &\leq \liminf_{n \to \infty} F(\nu_n', \mu_n) + \frac{\sigma^2}{2}\liminf_{n \to \infty}\operatorname{D_{KL}}(\nu_n'|\pi)\\
        &\leq \liminf_{n \to \infty} \left(F(\nu_n', \mu_n) + \frac{\sigma^2}{2}\operatorname{D_{KL}}(\nu_n'|\pi)\right).
    \end{aligned}
    \end{multline}
    Since $\nu_n' \in \mathcal{R}_1(\mu_n),$ for each $n,$ we have
    \begin{equation*}
        \nu_n' \in \argmin_{\nu' \in \mathcal{P}_p(\mathcal{X})} V^{\sigma}(\nu', \mu_n) = \argmin_{\nu' \in \mathcal{P}_p(\mathcal{X})} \left(F(\nu', \mu_n) + \frac{\sigma^2}{2}\operatorname{D_{KL}}(\nu'|\pi)\right),
    \end{equation*}
    and hence \eqref{eq:lsc-nu} becomes
    \begin{align*}
        V^{\sigma}(\nu'_{\infty}, \mu_{\infty}) + \frac{\sigma^2}{2}\operatorname{D_{KL}}(\mu_{\infty}|\rho) &\leq \liminf_{n \to \infty} \left(F(\nu, \mu_n) + \frac{\sigma^2}{2}\operatorname{D_{KL}}(\nu|\pi)\right)\\
        &= F(\nu, \mu_{\infty}) + \frac{\sigma^2}{2}\operatorname{D_{KL}}(\nu|\pi) = V^{\sigma}(\nu,\mu_{\infty}) + \frac{\sigma^2}{2}\operatorname{D_{KL}}(\mu_{\infty}|\rho),
    \end{align*}
    for any $\nu \in \mathcal{S}^{\nu}$. Therefore, $\nu'_{\infty} \in \mathcal{R}_1(\mu_{\infty}).$

    Similarly, by Condition \eqref{eq:continuity-F} and the upper semi-continuity of $\mu \mapsto -\operatorname{D_{KL}}(\mu|\rho)$, we have 
    \begin{multline}
    \label{eq:lsc-mu}
    \begin{aligned}
        V^{\sigma}(\nu_{\infty}, \mu'_{\infty}) - \frac{\sigma^2}{2}\operatorname{D_{KL}}(\nu_{\infty}|\pi) &= F(\nu_{\infty}, \mu'_{\infty}) - \frac{\sigma^2}{2}\operatorname{D_{KL}}(\mu'_{\infty}|\rho)\\
        &\geq \limsup_{n \to \infty} F(\nu_n, \mu_n') + \frac{\sigma^2}{2} \limsup_{n \to \infty}\left(-\operatorname{D_{KL}}(\mu_n'|\rho)\right)\\
        &\geq \limsup_{n \to \infty} \left(F(\nu_n, \mu_n') - \frac{\sigma^2}{2}\operatorname{D_{KL}}(\mu_n'|\rho)\right).
    \end{aligned}
    \end{multline}
    Since $\mu_n' \in \mathcal{R}_2(\nu_n),$ for each $n,$ we have
    \begin{equation*}
        \mu_n' \in \argmax_{\mu' \in \mathcal{P}_p(\mathcal{Y})} V^{\sigma}(\nu_n, \mu') = \argmax_{\mu' \in \mathcal{P}_p(\mathcal{Y})} \left(F(\nu_n, \mu') - \frac{\sigma^2}{2}\operatorname{D_{KL}}(\mu'|\rho)\right),
    \end{equation*}
    and hence \eqref{eq:lsc-mu} becomes
    \begin{align*}
        V^{\sigma}(\nu_{\infty}, \mu'_{\infty}) - \frac{\sigma^2}{2}\operatorname{D_{KL}}(\nu_{\infty}|\pi) &\geq \limsup_{n \to \infty} \left(F(\nu_n, \mu) - \frac{\sigma^2}{2}\operatorname{D_{KL}}(\mu|\rho)\right)\\ &= F(\nu_{\infty}, \mu) - \frac{\sigma^2}{2}\operatorname{D_{KL}}(\mu|\rho) = V^{\sigma}(\nu_{\infty}, \mu) - \frac{\sigma^2}{2}\operatorname{D_{KL}}(\nu_{\infty}|\pi),
    \end{align*}
    for any $\mu \in \mathcal{S}^{\mu}$. Therefore, $\mu'_{\infty} \in \mathcal{R}_2(\nu_{\infty})$. Hence, we obtain that $(\nu'_{\infty}, \mu'_{\infty}) \in \mathcal{R}(\nu_{\infty}, \mu_{\infty}),$ as required.

    Putting everything together, we showed that:
    \begin{itemize}
        \item $\mathcal{R}(\nu,\mu)$ is non-empty and convex,
        \item $\mathcal{S}$ is non-empty, $\mathcal{W}_p$-compact and convex,
        \item the graph of the mapping $\mathcal{S} \ni (\nu,\mu) \mapsto \mathcal{R}(\nu,\mu) \subset \mathcal{S}$ is $\mathcal{W}_p$-closed.
    \end{itemize}
    Hence, by the Kakutani-Fan-Glicksberg fixed point theorem (see, e.g., \cite[Corollary 17.55]{aliprantis2007infinite}), the set of fixed points of the mapping $\mathcal{S} \ni (\nu,\mu) \mapsto \mathcal{R}(\nu,\mu) \subset \mathcal{S}$ is non-empty. Thus, there exists $(\nu_{\sigma}^*, \mu_{\sigma}^*) \in \mathcal{S}$ such that $(\nu_{\sigma}^*, \mu_{\sigma}^*) \in \mathcal{R}(\nu_{\sigma}^*, \mu_{\sigma}^*).$ Then, $\nu_{\sigma}^* \in \argmin_{\nu' \in \mathcal{P}_p(\mathcal{X})} V^{\sigma}(\nu',\mu_{\sigma}^*)$ and $\mu_{\sigma}^* \in \argmin_{\mu' \in \mathcal{P}_p(\mathcal{Y})} \{-V^{\sigma}(\nu_{\sigma}^*,\mu')\}.$ Therefore, $V(\nu_{\sigma}^*, \mu) \leq V(\nu_{\sigma}^*, \mu_{\sigma}^*) \leq V(\nu, \mu_{\sigma}^*)$, for all $(\nu,\mu) \in \mathcal{P}_p(\mathcal{X}) \times \mathcal{P}_p(\mathcal{X})$, i.e., $(\nu_{\sigma}^*, \mu_{\sigma}^*)$ is a MNE of the game \eqref{eq:F-nonlinear}. 
\end{proof}
\begin{lemma}[Uniqueness of MNE]
\label{lemma:Uniqueness-saddle-point}
     For $V^{\sigma}$ given by \eqref{eq:F-nonlinear}, if Assumption \ref{assumption: assump-F-conv-conc} holds and $(\nu_{\sigma}^*, \mu_{\sigma}^*) \in \mathcal{P}_{\text{ac}}(\mathcal{X}) \times \mathcal{P}_{\text{ac}}(\mathcal{Y})$ is a saddle point of $V^{\sigma}$, that is $V^{\sigma}(\nu_{\sigma}^*, \mu) \leq V^{\sigma}(\nu_{\sigma}^*, \mu_{\sigma}^*) \leq V^{\sigma}(\nu, \mu_{\sigma}^*)$, for all $(\nu, \mu) \in \mathcal{P}(\mathcal{X}) \times \mathcal{P}(\mathcal{Y})$, then it is unique. 
\end{lemma}
\begin{proof}
Suppose to the contrary that $(\nu_{\sigma}^*, \mu_{\sigma}^*), (\hat{\nu}_{\sigma}^*, \hat{\mu}_{\sigma}^*)  \in \mathcal{P}_{\text{ac}}(\mathcal{X}) \times \mathcal{P}_{\text{ac}}(\mathcal{Y})$ are two saddle points of $V.$ Then, from Proposition \ref{prop: 2.5}, we can write the first order condition
\begin{equation*}
    \frac{\delta F}{\delta \nu}(\nu_{\sigma}^*, \mu_{\sigma}^*, x) + \frac{\sigma^2}{2}\log \left(\frac{\nu_{\sigma}^*(x)}{\pi(x)}\right) = \operatorname{constant}, 
\end{equation*}
\begin{equation*}
    \frac{\delta F}{\delta \mu}(\nu_{\sigma}^*, \mu_{\sigma}^*, y) - \frac{\sigma^2}{2}\log \left(\frac{\mu_{\sigma}^*(y)}{\rho(y)}\right) = \operatorname{constant},
\end{equation*}
for all $(x, y) \in \mathcal{X} \times \mathcal{Y}$ Lebesgue almost surely and the same two equations also hold for $(\hat{\nu}_{\sigma}^*, \hat{\mu}_{\sigma}^*).$ Then, by Assumption \ref{assumption: assump-F-conv-conc}, we have 
\begin{align*}
    V^{\sigma}(\hat{\nu}_{\sigma}^*, \mu_{\sigma}^*) - V^{\sigma}(\nu_{\sigma}^*, \mu_{\sigma}^*) &\geq \int_{\mathcal{X}} \left(\frac{\delta F}{\delta \nu}(\nu_{\sigma}^*, \mu_{\sigma}^*, x) + \frac{\sigma^2}{2}\log \left(\frac{\nu_{\sigma}^*(x)}{\pi(x)}\right)\right)(\hat{\nu}_{\sigma}^*-\nu_{\sigma}^*)(\mathrm{d}x)\\ &+ \frac{\sigma^2}{2}\operatorname{D_{KL}}(\hat{\nu}_{\sigma}^*|\nu_{\sigma}^*) = \frac{\sigma^2}{2}\operatorname{D_{KL}}(\hat{\nu}_{\sigma}^*|\nu_{\sigma}^*),
\end{align*}
\begin{align*}
    V^{\sigma}(\nu_{\sigma}^*, \hat{\mu}_{\sigma}^*) - V^{\sigma}(\nu_{\sigma}^*, \mu_{\sigma}^*) &\leq \int_{\mathcal{Y}} \left(\frac{\delta F}{\delta \mu}(\nu_{\sigma}^*, \mu_{\sigma}^*, y) - \frac{\sigma^2}{2}\log \left(\frac{\mu_{\sigma}^*(y)}{\rho(y)}\right)\right)(\hat{\mu}_{\sigma}^*-\mu_{\sigma}^*)(\mathrm{d}y)\\ &- \frac{\sigma^2}{2}\operatorname{D_{KL}}(\hat{\mu}_{\sigma}^*|\mu_{\sigma}^*) = - \frac{\sigma^2}{2}\operatorname{D_{KL}}(\hat{\mu}_{\sigma}^*|\mu_{\sigma}^*),
\end{align*}
where the equalities follow from the first order condition. Swapping $\nu_{\sigma}^*$ and $\mu_{\sigma}^*$ with $\hat{\nu}_{\sigma}^*$ and $\hat{\mu}_{\sigma}^*$ in the inequalities above, we get the analogous
\begin{align*}
    V^{\sigma}(\nu_{\sigma}^*, \hat{\mu}_{\sigma}^*) - V^{\sigma}(\hat{\nu}_{\sigma}^*, \hat{\mu}_{\sigma}^*) \geq \frac{\sigma^2}{2}\operatorname{D_{KL}}(\nu_{\sigma}^*|\hat{\nu}_{\sigma}^*),
\end{align*}
\begin{align*}
    V^{\sigma}(\hat{\nu}_{\sigma}^*, \mu_{\sigma}^*) - V^{\sigma}(\hat{\nu}_{\sigma}^*, \hat{\mu}_{\sigma}^*) \leq - \frac{\sigma^2}{2}\operatorname{D_{KL}}(\mu_{\sigma}^*|\hat{\mu}_{\sigma}^*).
\end{align*}
Multiplying the second and the forth inequalities by $-1$ and adding all inequalities gives
\begin{equation*}
    0 \geq \frac{\sigma^2}{2}\operatorname{D_{KL}}(\hat{\nu}_{\sigma}^*|\nu_{\sigma}^*) + \frac{\sigma^2}{2}\operatorname{D_{KL}}(\hat{\mu}_{\sigma}^*|\mu_{\sigma}^*) + \frac{\sigma^2}{2}\operatorname{D_{KL}}(\nu_{\sigma}^*|\hat{\nu}_{\sigma}^*) + \frac{\sigma^2}{2}\operatorname{D_{KL}}(\mu_{\sigma}^*|\hat{\mu}_{\sigma}^*).
\end{equation*}
Since $\operatorname{D_{KL}}(m|m') \geq 0$ for all $m, m' \in \mathcal{P}(\mathcal{M}),$ where $\mathcal{M} \subseteq \mathbb R^d$, with equality if and only if $m=m',$ it follows that
\begin{equation*}
    \nu_{\sigma}^* = \hat{\nu}_{\sigma}^* \text{ and }  \mu_{\sigma}^* = \hat{\mu}_{\sigma}^*,
\end{equation*}
and hence $V^{\sigma}$ has a unique saddle point. 
\end{proof}
The following result shows the relation between the regularized game and the non-regularized one as $\sigma \to 0.$
\begin{proposition}
\label{prop:gamma-convergence}
    Suppose that the assumptions of Theorem \ref{thm: 2.6} hold, that $F$ admits a saddle point $(\nu^*, \mu^*)$ on $\mathcal{P}_p(\mathcal{X}) \times \mathcal{P}_p(\mathcal{Y}),$ and that there exist $C_{\pi}, C_{\rho} > 0$ such that $U^{\pi}(x) \leq C_{\pi}\left(1+|x|^p\right)$ and $U^{\rho}(y) \leq C_{\rho}\left(1+|y|^p\right),$ for all $(x,y) \in \mathcal{X} \times \mathcal{Y}.$ Then, given the saddle point $(\nu_{\sigma}^*, \mu_{\sigma}^*) \in \mathcal{P}_p(\mathcal{X}) \times \mathcal{P}_p(\mathcal{Y})$ of $V^{\sigma},$ it holds
    \begin{equation*}
        \lim_{\sigma \to 0} V^{\sigma}(\nu_{\sigma}^*, \mu_{\sigma}^*) = F(\nu^*,\mu^*).
    \end{equation*}
\end{proposition}
\begin{proof}
Let $f:\mathcal{X} \cup \mathcal{Y} \to (0, \infty)$ be the Gaussian kernel given by $f(x) = (2\pi)^{-\frac{d}{2}}\exp\left(-\frac{1}{2}|x|^2\right).$ Then, define the mollifier $f_{\sigma}(x) \coloneqq \sigma^{-d} f\left(\frac{x}{\sigma}\right).$ Given $(\nu,\mu) \in \mathcal{P}_{p}(\mathcal{X}) \times \mathcal{P}_{p}(\mathcal{Y}),$ the mollifications $(\nu_{\sigma}, \mu_{\sigma}) \coloneqq (\nu \ast f_{\sigma}, \mu \ast f_{\sigma})$ of $(\nu, \mu)$ are given by
\begin{equation*}
    \nu_{\sigma}(x) \coloneqq \int_{\mathcal{X}} f_{\sigma}(x-z) \nu(\mathrm{d}z),\quad
    \mu_{\sigma}(y) \coloneqq \int_{\mathcal{X}} f_{\sigma}(y-w) \mu(\mathrm{d}w).
\end{equation*}
Then since $h(x) \coloneqq x\log(x)$ is convex, it follows by Jensen's inequality that
\begin{align*}
    \int_{\mathcal{X}} h(\nu_{\sigma}(x)) \mathrm{d}x &\leq \int_{\mathcal{X}} \int_{\mathcal{X}} h(f_{\sigma}(x-z)) \nu(\mathrm{d}z)\mathrm{d}x = \int_{\mathcal{X}} h(f_{\sigma}(x))\mathrm{d}x\\ &= \int_{\mathcal{X}} h\left(\sigma^{-d} f\left(\frac{x}{\sigma}\right)\right)\mathrm{d}x = \int_{\mathcal{X}} h(f(x))\mathrm{d}x - d\log \sigma.
\end{align*}
On the other hand, using the fact that $U^{\pi}(x) \leq C_{\pi}\left(1+|x|^p\right)$ for some $C_{\pi} > 0,$ we have
\begin{align*}
    \int_{\mathcal{X}} \nu_{\sigma}(x)\log \pi(x) \mathrm{d}x &= \int_{\mathcal{X}} \int_{\mathcal{X}} \log \pi(x) f_{\sigma}(x-z) \nu(\mathrm{d}z)\mathrm{d}x = -\int_{\mathcal{X}} \int_{\mathcal{X}} U^{\pi}(x+z) f_{\sigma}(x) \mathrm{d}x\nu(\mathrm{d}z)\\ &\geq -C_{\pi}\int_{\mathcal{X}} \int_{\mathcal{X}} (1+|x+z|^p) f_{\sigma}(x) \mathrm{d}x\nu(\mathrm{d}z) = -C_{\pi} \left(1 + \int_{\mathcal{X}}|z|^p\nu(\mathrm{d}z)\right) < \infty.
\end{align*}
Therefore, we have
\begin{equation}
\label{eq:upper-bound-mollifier-nu}
    \operatorname{D_{KL}}(\nu_{\sigma}|\pi) \leq \int_{\mathcal{X}} h(f(x))\mathrm{d}x - d\log \sigma + C_{\pi} \left(1 + \int_{\mathcal{X}}|z|^p\nu(\mathrm{d}z)\right) < \infty.
\end{equation}
Analogously, we obtain
\begin{equation}
\label{eq:upper-bound-mollifier-mu}
    \operatorname{D_{KL}}(\mu_{\sigma}|\rho) \leq \int_{\mathcal{Y}} h(f(y))\mathrm{d}y - d\log \sigma + C_{\rho} \left(1 + \int_{\mathcal{Y}}|w|^p\mu(\mathrm{d}w)\right) < \infty.
\end{equation}
Recall that since $(\nu_{\sigma}^*, \mu_{\sigma}^*) \in \mathcal{P}_p(\mathcal{X}) \times \mathcal{P}_p(\mathcal{Y})$ is the saddle point of $V^{\sigma},$ it holds
\begin{equation*}
    V^{\sigma}(\nu_{\sigma}^*, \mu) \leq V^{\sigma}(\nu_{\sigma}^*, \mu_{\sigma}^*) \leq V^{\sigma}(\nu, \mu_{\sigma}^*), \quad \text{for all } (\nu, \mu) \in \mathcal{P}_p(\mathcal{X}) \times \mathcal{P}_p(\mathcal{Y}).
\end{equation*}
Let $\mu \in \mathcal{P}_p(\mathcal{Y}).$ Then, we have
    \begin{align*}
        V^{\sigma}(\nu_{\sigma}^*, \mu_{\sigma}^*) &+ \frac{\sigma^2}{2}\left(\int_{\mathcal{Y}} h(f(y))\mathrm{d}y - d\log \sigma + C_{\rho} \left(1 + \int_{\mathcal{Y}}|w|^p\mu(\mathrm{d}w)\right)\right)\\ &\geq V^{\sigma}(\nu_{\sigma}^*, \mu_{\sigma}^*) + \frac{\sigma^2}{2}\operatorname{D_{KL}}(\mu_{\sigma}|\rho)\\ &\geq V^{\sigma}(\nu_{\sigma}^*, \mu_{\sigma}) + \frac{\sigma^2}{2}\operatorname{D_{KL}}(\mu_{\sigma}|\rho)\\ &= F(\nu_{\sigma}^*, \mu_{\sigma}) + \frac{\sigma^2}{2}\operatorname{D_{KL}}(\nu_{\sigma}^*|\pi) \geq F(\nu_{\sigma}^*, \mu_{\sigma}) \geq \inf_{\nu \in \mathcal{P}_p(\mathcal{X})} F(\nu, \mu_{\sigma}),
\end{align*}
where the first inequality follows from \eqref{eq:upper-bound-mollifier-mu}, the second inequality follows since $(\nu_{\sigma}^*, \mu_{\sigma}^*) \in \mathcal{P}_{p}(\mathcal{X}) \times \mathcal{P}_{p}(\mathcal{Y})$ is the saddle point of $V^{\sigma},$ and the third inequality follows since $\operatorname{D_{KL}}(\cdot|\pi) \geq 0.$ By \cite[Lemma 5.2]{santambrogio2015optimal}, we have $\mu_{\sigma} \to \mu$ in $\mathcal{W}_p$ as $\sigma \to 0,$ and since $F$ is jointly $\mathcal{W}_p$-continuous, it follows by taking liminf as $\sigma \to 0,$
\begin{equation*}
    \liminf_{\sigma \to 0}  V^{\sigma}(\nu_{\sigma}^*, \mu_{\sigma}^*) \geq \inf_{\nu \in \mathcal{P}_p(\mathcal{X})} F(\nu, \mu), \text{ for all } \mu \in \mathcal{P}_p(\mathcal{Y}).
\end{equation*}
Hence,
\begin{equation}
\label{eq:liminf-bound}
    \liminf_{\sigma \to 0}  V^{\sigma}(\nu_{\sigma}^*, \mu_{\sigma}^*) \geq \sup_{\mu \in \mathcal{P}_p(\mathcal{Y})} \inf_{\nu \in \mathcal{P}_p(\mathcal{X})} F(\nu, \mu).
\end{equation}
Let $\nu \in \mathcal{P}_p(\mathcal{X}).$ Then, we have 
\begin{align*}
    V^{\sigma}(\nu_{\sigma}^*, \mu_{\sigma}^*) &\leq V^{\sigma}(\nu_{\sigma}, \mu_{\sigma}^*)\\ &\leq F(\nu_{\sigma}, \mu_{\sigma}^*) + \frac{\sigma^2}{2}\operatorname{D_{KL}}(\nu_{\sigma}|\pi)\\
    &\leq F(\nu_{\sigma}, \mu_{\sigma}^*) + \frac{\sigma^2}{2}\left(\int_{\mathcal{X}} h(f(x))\mathrm{d}x - d\log \sigma + C_{\pi} \left(1 + \int_{\mathcal{X}}|z|^p\nu(\mathrm{d}z)\right)\right)\\
    &\leq \sup_{\mu \in \mathcal{P}_p(\mathcal{Y})} F(\nu_{\sigma}, \mu) + \frac{\sigma^2}{2}\left(\int_{\mathcal{X}} h(f(x))\mathrm{d}x - d\log \sigma + C_{\pi} \left(1 + \int_{\mathcal{X}}|z|^p\nu(\mathrm{d}z)\right)\right),
\end{align*}
where the first inequality follows since $(\nu_{\sigma}^*, \mu_{\sigma}^*) \in \mathcal{P}_{p}(\mathcal{X}) \times \mathcal{P}_{p}(\mathcal{Y})$ is the saddle point of $V^{\sigma},$ the second inequality follows since $-\operatorname{D_{KL}}(\cdot|\rho) \leq 0,$ and the third inequality follows from \eqref{eq:upper-bound-mollifier-nu}. By \cite[Lemma 5.2]{santambrogio2015optimal}, we have $\nu_{\sigma} \to \nu$ in $\mathcal{W}_p$ as $\sigma \to 0,$ and since $F$ is jointly $\mathcal{W}_p$-continuous, it follows by taking limsup as $\sigma \to 0,$
\begin{equation*}
    \limsup_{\sigma \to 0}  V^{\sigma}(\nu_{\sigma}^*, \mu_{\sigma}^*) \leq \sup_{\mu \in \mathcal{P}_p(\mathcal{Y})} F(\nu, \mu), \text{ for all } \nu \in \mathcal{P}_p(\mathcal{X}).
\end{equation*}
Hence,
\begin{equation}
\label{eq:limsup-bound}
    \limsup_{\sigma \to 0}  V^{\sigma}(\nu_{\sigma}^*, \mu_{\sigma}^*) \leq \inf_{\nu \in \mathcal{P}_p(\mathcal{X})} \sup_{\mu \in \mathcal{P}_p(\mathcal{Y})}  F(\nu, \mu).
\end{equation}
Combining \eqref{eq:liminf-bound} and \eqref{eq:limsup-bound}, we obtain
\begin{equation*}
    \sup_{\mu \in \mathcal{P}_p(\mathcal{Y})} \inf_{\nu \in \mathcal{P}_p(\mathcal{X})} F(\nu, \mu) \leq \liminf_{\sigma \to 0}  V^{\sigma}(\nu_{\sigma}^*, \mu_{\sigma}^*) \leq \limsup_{\sigma \to 0}  V^{\sigma}(\nu_{\sigma}^*, \mu_{\sigma}^*) \leq \inf_{\nu \in \mathcal{P}_p(\mathcal{X})} \sup_{\mu \in \mathcal{P}_p(\mathcal{Y})}  F(\nu, \mu)
\end{equation*}
Since $F$ admits a saddle point $(\nu^*, \mu^*)$ on $\mathcal{P}_p(\mathcal{X}) \times \mathcal{P}_p(\mathcal{Y}),$ we have
\begin{equation*}
    \inf_{\nu \in \mathcal{P}_p(\mathcal{X})} \sup_{\mu \in \mathcal{P}_p(\mathcal{Y})} F(\nu,\mu) = \sup_{\mu \in \mathcal{P}_p(\mathcal{Y})} \inf_{\nu \in \mathcal{P}_p(\mathcal{X})} F(\nu,\mu) = F(\nu^*, \mu^*),
\end{equation*}
and therefore
\begin{equation*}
    \liminf_{\sigma \to 0}  V^{\sigma}(\nu_{\sigma}^*, \mu_{\sigma}^*) = \limsup_{\sigma \to 0}  V^{\sigma}(\nu_{\sigma}^*, \mu_{\sigma}^*) = F(\nu^*, \mu^*),
\end{equation*}
hence the conclusion follows.
\end{proof}

\section{Notation and definitions}
\label{app: AppB}
In this section we recall some important definitions. Following \cite[Definition $5.43$]{Carmona2018ProbabilisticTO}, 
we start with the notion of differentiability on the space of probability measure that we utilize throughout the paper.
\begin{definition}
\label{def:fderivative} 
Fix $p \geq 0.$ For any $\mathcal{M} \subseteq \mathbb R^d,$ let $\mathcal{P}_p(\mathcal{M})$ be the space of probability measures on $\mathcal{M}$ with finite $p$-th moments. A function $F:\mathcal P_p(\mathcal{M}) \to \mathbb R$ admits first-order flat derivative on $\mathcal P_p(\mathcal{M}),$ if there exists a function $\frac{\delta F}{\delta \nu}: \mathcal P_p(\mathcal{M}) \times \mathcal{M}\rightarrow \mathbb R,$ such that
\begin{enumerate}
\item the map $\mathcal P_p(\mathcal{M}) \times \mathcal{M} \ni (m, x) \mapsto \frac{\delta F}{\delta m}(m,x)$ is jointly continuous with respect to the product topology, where $\mathcal P_p(\mathcal{M})$ is endowed with the weak topology,
\item for any $m \in \mathcal P_p(\mathcal{M}),$ there exists $C>0$ such that, for all $x\in \mathcal{M},$ we have 
\[
\left|\frac{\delta F}{\delta m}(m,x)\right|\leq C\left(1+|x|^p\right),
\]
\item for all $m, m' \in \mathcal P_p(\mathcal{M}),$ it holds that
\begin{equation}
\label{def:FlatDerivative}
F(m')-F(m)=\int_{0}^{1}\int_{\mathcal{M}}\frac{\delta F}{\delta m}(m + \varepsilon (m'-m),x)\left(m'- m\right)(\mathrm{d}x)\mathrm{d}\varepsilon.
\end{equation}
\end{enumerate}
The functional $\frac{\delta F}{\delta m}$ is then called the flat derivative of $F$ on $\mathcal P_p(\mathcal{M}).$ We note that $\frac{\delta F}{\delta m}$ exists up to an additive constant, and thus we make the normalizing convention $\int_{\mathcal{M}} \frac{\delta F}{\delta m}(m,x) m(\mathrm{d}x) = 0.$
\end{definition}
If, for any fixed $x \in \mathcal{M},$ the map $m \mapsto \frac{\delta F}{\delta m}(m,x)$ satisfies Definition \ref{def:fderivative}, we say that $F$ admits a second-order flat derivative denoted by $\frac{\delta^2 F}{\delta m^2}.$ Consequently, by Definition \ref{def:fderivative}, there exists a functional $\frac{\delta^2 F}{\delta m^2}: \mathcal P_p(\mathcal{M}) \times\mathcal{M} \times\mathcal{M} \rightarrow \mathbb R$ such that
\begin{equation}
\label{def:2FlatDerivative}
\frac{\delta F}{\delta m}(m',x)-\frac{\delta F}{\delta m}(m,x) = \int_{0}^{1}\int_{\mathcal{M}}\frac{\delta^2 F}{\delta m^2}(\nu + \varepsilon (m'-m),x, x')\left(m'- m\right)(\mathrm{d}x')\mathrm{d}\varepsilon.
\end{equation}
Now, we show a measure-space equivalent of Schwarz’s theorem on symmetry of second order flat derivatives.
\begin{lemma}[Symmetry of second order flat derivatives]
\label{lemma: symmetry-flat}
Let $\mathcal{X}, \mathcal{Y} \subseteq \mathbb R^d.$ Assume $F:\mathcal{P}(\mathcal{X}) \times \mathcal{P}(\mathcal{Y}) \to \mathbb R$ admits second order flat derivative and the maps $(\nu,\mu,y,x) \mapsto \frac{\delta^2 F}{\delta \nu \delta \mu} (\nu, \mu, y, x), (\nu,\mu,x,y) \mapsto \frac{\delta^2 F}{\delta \mu \delta \nu} (\nu, \mu, x, y)$ are jointly continuous in all variables. Then we have
 \begin{equation*}
        \int_{\mathcal{Y}} \int_{\mathcal{X}} \frac{\delta^2 F}{\delta \mu \delta \nu} (\nu, \mu, x, y) (\nu'-\nu)(\mathrm{d}x)(\mu'-\mu)(\mathrm{d}y) = \int_{\mathcal{X}} \int_{\mathcal{Y}} \frac{\delta^2 F}{\delta \nu \delta \mu} (\nu, \mu, y, x) (\mu'-\mu)(\mathrm{d}y)(\nu'-\nu)(\mathrm{d}x).
    \end{equation*}
\end{lemma}
\begin{proof}
    Let $(\nu,\mu), (\nu',\mu') \in \mathcal{P}(\mathcal{X}) \times \mathcal{P}(\mathcal{Y}).$ Then $\nu+s(\nu'-\nu)$ and $\mu+t(\mu'-\mu)$ are probability measures for all $(s,t) \in [0,1] \times [0,1].$ Since $F$ admits second order flat derivatives, it follows that the map $f: [0,1] \times [0,1] \ni (s,t) \mapsto F(\nu+s(\nu'-\nu), \mu+t(\mu'-\mu))$ is twice differentiable, and so
    \begin{equation*}
        \partial_t \partial_s f(s,t) = \int_{\mathcal{Y}} \int_{\mathcal{X}} \frac{\delta^2 F}{\delta \mu \delta \nu} (\nu+s(\nu'-\nu), \mu+t(\mu'-\mu), x, y) (\nu'-\nu)(\mathrm{d}x)(\mu'-\mu)(\mathrm{d}y),
    \end{equation*}
    \begin{equation*}
        \partial_s \partial_t f(s,t) = \int_{\mathcal{X}} \int_{\mathcal{Y}} \frac{\delta^2 F}{\delta \nu \delta \mu} (\nu+s(\nu'-\nu), \mu+t(\mu'-\mu), y, x) (\mu'-\mu)(\mathrm{d}y)(\nu'-\nu)(\mathrm{d}x).
    \end{equation*}
    By Schwarz's Theorem, we have $\partial_t \partial_s f(s,t) =  \partial_s \partial_t f(s,t),$ for all $(s,t) \in [0,1] \times [0,1],$ and then setting $s=t=0$ gives
    \begin{equation*}
        \int_{\mathcal{Y}} \int_{\mathcal{X}} \frac{\delta^2 F}{\delta \mu \delta \nu} (\nu, \mu, x, y) (\nu'-\nu)(\mathrm{d}x)(\mu'-\mu)(\mathrm{d}y) = \int_{\mathcal{X}} \int_{\mathcal{Y}} \frac{\delta^2 F}{\delta \nu \delta \mu} (\nu, \mu, y, x) (\mu'-\mu)(\mathrm{d}y)(\nu'-\nu)(\mathrm{d}x).
    \end{equation*}
\end{proof}

\begin{definition}[TV distance between probability measures; \cite{tsybakov2008introduction}, Definition $2.4$]
\label{def:KRwasserstein}
    Let $\left(\mathcal{M}, \mathcal{A}\right)$ be a measurable space and let $P$ and $Q$ be probability measures on $\left(\mathcal{M}, \mathcal{A}\right)$. Assume that $\mu$ is a $\sigma$-finite measure on $\left(\mathcal{M}, \mathcal{A}\right)$ such that $P$ and $Q$ are absolutely continuous with respect to $\mu$ and let $p$ and $q$ denote their probability density functions, respectively. The total variation distance between $P$ and $Q$ is defined as:
    \begin{equation*}
        \operatorname{TV}(P,Q) \coloneqq \sup_{A \in \mathcal{A}}\left|P(A) - Q(A)\right| = \sup_{A \in \mathcal{A}} \left|\int_{A} (p-q) \mathrm{d}\mu\right|.
    \end{equation*}
\end{definition}
\end{document}